\documentclass[twoside,a4paper,reqno,11pt]{amsart}
\usepackage{amsfonts, amsbsy, amsmath, amsthm, amssymb, latexsym, verbatim, enumerate}
\usepackage{mathrsfs}
\usepackage{pdfsync}
\usepackage[top=28mm,right=28mm,bottom=28mm,left=28mm]{geometry}
\usepackage{stmaryrd}
\usepackage{bm}
\usepackage{fancyhdr}
\usepackage[pdftex]{color,graphicx}

\makeatletter
\newcommand{\imod}[1]{\allowbreak\mkern4mu({\operator@font mod}\,\,#1)}
\makeatother

\renewcommand{\a}{\alpha}
\renewcommand{\b}{\beta}

 \renewcommand{\L}{\Lambda}
\renewcommand{\l}{\lambda} 
 
 \renewcommand{\to}{\rightarrow}
 \newcommand{\s}{\sigma}
\renewcommand{\o}{\omega} 
 \newcommand{\C}{\mathcal{C}}

\newcommand{\leqs}{\leqslant}
\newcommand{\geqs}{\geqslant}

 \newcommand{\vs}{\vspace{3mm}}
\newcommand{\la}{\langle}
\newcommand{\ra}{\rangle}

\newtheorem{theorem}{Theorem}

\newtheorem{corol}{Corollary}
\newtheorem{hyp}{Hypothesis}

\newtheorem{thm}{Theorem}[section]
\newtheorem{prop}[thm]{Proposition}
\newtheorem{lem}[thm]{Lemma}

\theoremstyle{definition}
\newtheorem{defn}[thm]{Definition}
\newtheorem{remk}[thm]{Remark}
\newtheorem{definition}{Definition}

\newtheorem{remark}{Remark}

\begin{document}

\author{Timothy C. Burness}
 \address{T.C. Burness, School of Mathematics, University of Bristol, Bristol BS8 1TW, UK}
 \email{t.burness@bristol.ac.uk}
 
\author{Claude Marion}
\address{C. Marion, Section de Math\'ematiques, Station 8, \'{E}cole Polytechnique F\'{e}d\'{e}rale de Lausanne, CH-1015 Lausanne, Switzerland}
\email{claude.marion@epfl.ch}

\author{Donna M. Testerman}
\address{D.M. Testerman, Section de Math\'ematiques, Station 8, \'{E}cole Polytechnique F\'{e}d\'{e}rale de Lausanne, CH-1015 Lausanne, Switzerland}
\email{\texttt{donna.testerman@epfl.ch}}

\title[On irreducible subgroups of simple algebraic groups]{On irreducible subgroups of simple algebraic groups}

\subjclass[2010]{Primary 20G05; Secondary 20E28, 20E32}

\thanks{Burness was supported by EPSRC grant EP/I019545/1, and he thanks the Section de Math\'{e}matiques at EPFL for their generous hospitality. The second and third authors were  supported by the Fonds National Suisse de la Recherche Scientifique, grant number 200021-153543. We thank Paul Levy, Gunter Malle and an anonymous referee for helpful comments.}

\begin{abstract}
Let $G$ be a simple algebraic group over an algebraically closed field $K$ of characteristic $p\geqs 0$, let $H$ be a proper closed subgroup of $G$ and let $V$ be a nontrivial irreducible
$KG$-module, which is $p$-restricted, tensor indecomposable and rational. Assume that the restriction of $V$ to $H$ is irreducible. In this paper, we study the triples $(G,H,V)$ of this form when $G$ is a classical group and $H$ is positive-dimensional. Combined with earlier work of Dynkin, Seitz, Testerman and others, our main theorem reduces the problem of classifying the triples $(G,H,V)$ to the case where $G$ is an orthogonal group, $V$ is a spin module and $H$ normalizes an orthogonal decomposition of the natural $KG$-module.
\end{abstract}

\date{\today}
\maketitle


\section{Introduction}\label{s:intro}

Let $G$ be a simple algebraic group defined over an algebraically closed field $K$, let $H$ be a closed positive-dimensional subgroup of $G$, and let $V$ be a nontrivial rational irreducible $KG$-module. We say that $(G,H,V)$ is an \emph{irreducible triple} if $V$ is irreducible as a $KH$-module. Triples of this form arise naturally in the investigation of maximal subgroups of classical algebraic groups, and their study can be traced back to work of Dynkin \cite{Dynkin1}  in the 1950s, who considered the special case where $H$ is connected and $K = \mathbb{C}$. In the 1980s, Seitz \cite{Seitz2} and Testerman \cite{Test1} extended the analysis to arbitrary algebraically closed fields (still assuming $H$ is connected), and more recent work of Ghandour \cite{g_paper} has completed the classification of irreducible triples for exceptional groups. Therefore, in this paper we focus on classical groups and their disconnected subgroups.  

In \cite{Ford1, Ford2}, Ford determines the irreducible triples $(G,H,V)$ where $G$ is a classical group, $H$ is disconnected, $H^0$ is simple and the composition factors of $V|_{H^0}$ are $p$-restricted. Our main aim is to extend Ford's analysis by removing the restrictive conditions on the structure of $H^0$ and the composition factors of $V|_{H^0}$. The cases for which $V|_{H^0}$ is irreducible are easily deduced from the work of Seitz \cite{Seitz2}, so we focus on the situation where $V|_{H}$ is irreducible, but $V|_{H^0}$ is reducible. By Clifford theory, the highest weights of the $KH^0$-composition factors of $V$ are $H$-conjugate and this severely restricts the possibilities for $V$. Since the triples with $H$ maximal have recently been determined in \cite{BGMT,BGT}, in this paper we will adopt the following hypothesis:

\begin{hyp}\label{h:our}
\emph{The group $G$ is a simply connected cover of a simple classical algebraic group defined over an algebraically closed field $K$ of characteristic $p \geqs 0$, $H$ is a closed positive-dimensional subgroup of $G$, and $V$ is a nontrivial $p$-restricted irreducible tensor indecomposable rational $KG$-module such that the following conditions hold:
\begin{itemize}\addtolength{\itemsep}{0.2\baselineskip}
\item[H1.] $V \neq W^{\tau}$ for any automorphism $\tau$ of $G$, where $W$ is the natural module;
\item[H2.] $HZ(G)/Z(G)$ is disconnected and non-maximal in $G/Z(G)$.
\end{itemize}}
\end{hyp}
 
Let $G$ be a classical group as in Hypothesis \ref{h:our}, let $n$ denote the rank of $G$ and let $\{\l_1, \ldots, \l_n\}$ be a set of fundamental dominant weights for $G$ (we adopt the standard labelling given in Bourbaki \cite{Bou}). We will write $V_G(\l)$ for the irreducible $KG$-module with highest weight $\l$. 

\begin{remark}\label{r:hyp}
Condition H1 in Hypothesis \ref{h:our} is equivalent to assuming $V \neq W,W^*$, and also $V \neq V_G(\l_3), \, V_G(\l_4)$ if $G=D_4$. This hypothesis is unavoidable. For example, we cannot feasibly determine all the almost simple subgroups of $G$ that act irreducibly on $W$ or $W^*$ (indeed, even the dimensions of the irreducible modules for simple groups are not known, in general). In particular, H1 is a condition adopted in \cite{BGMT} and \cite{Seitz2}. 
\end{remark} 
 
Suppose $G$ is of type $B_n$ or $D_n$, and let $R(W)=R$ be the radical of the corresponding bilinear form on $W$ (recall that either $R=0$, or $p=2$, $\dim W$ is odd and $\dim R=1$). An \emph{orthogonal decomposition} of $W$ is a decomposition of the form 
$$W = W_1 + \cdots + W_t,$$ 
where the $W_i$ are pairwise orthogonal subspaces of $W$. Note that if $W = W_1 + \cdots + W_t$ is such a decomposition, then the $W_i$ are non-degenerate spaces such that $W_i \cap (\sum_{j \neq i}W_j) \subseteq R$ for each $i$ (in particular, if 
$p \neq 2$ then $W = W_1 + \cdots + W_t$ is a direct sum). We say that a subgroup $H$ of $G$ \emph{normalizes} such a decomposition if it permutes the $W_i$.

\begin{definition}
Let $G$ be a classical group of type $B_n$ or $D_n$, as in Hypothesis \ref{h:our}. A closed subgroup $H$ of $G$ is a \emph{decomposition subgroup} if one of the following holds:
\begin{itemize}\addtolength{\itemsep}{0.2\baselineskip}
\item[(a)] $H$ normalizes an orthogonal decomposition $W = W_1 + \cdots + W_t$; or
\item[(b)] $(G,p)=(D_n,2)$, $H$ stabilizes a $1$-dimensional non-singular subspace $U$ of $W$, and $H$ normalizes an orthogonal decomposition of the natural module for the stabilizer $G_U=B_{n-1}$. 
\end{itemize}
Similarly, if $(G,p)=(C_n,2)$ then $H$ is a decomposition subgroup of $G$ if it is the image of a decomposition subgroup of the dual group $\widetilde{G}=B_n$ with respect to a bijective morphism $\varphi:\widetilde{G} \to G$. 
\end{definition}

\begin{theorem}\label{t:main}
Suppose $G$, $H$ and $V = V_G(\l)$ satisfy the conditions in Hypothesis \ref{h:our}, and assume that $V|_{H^0}$ is reducible. Then $V|_{H}$ is irreducible only if one of the following holds: 
\begin{itemize}\addtolength{\itemsep}{0.2\baselineskip}
\item[{\rm (a)}] $(G,H,V)$ is one of the cases in Table \ref{tab:main}; or
\item[{\rm (b)}] $G$ is of type $B_n$ or $D_n$ (or type $C_n$ if $p=2$), $V$ is a spin module and $H$ is a decomposition subgroup.
\end{itemize}
Moreover, if (a) holds then $V|_{H}$ is irreducible.
\end{theorem}

\renewcommand{\arraystretch}{1.1}
\begin{table}[h]
\begin{center}
$$\begin{array}{lllll} \hline
& G & H & \l & \mbox{Conditions} \\ \hline
{\rm (i)} & A_n & T_n.X & \l_k,\, 1<k<n & \mbox{$X < {\rm Sym}_{n+1}$ is $\ell$-transitive, $\ell = \min\{k,n+1-k\}$} \\
{\rm (ii)} & B_4 & B_1^3.X & \l_3 & \mbox{$p=2$, $X=Z_3$ or ${\rm Sym}_3$} \\
{\rm (iii)} & C_4 & C_1^3.X & \l_3 & \mbox{$p=2$, $X=Z_3$ or ${\rm Sym}_3$} \\
{\rm (iv)} & C_4 & C_1^3.Z_3 & \l_2,\, \l_3 & p \neq 2 \;\; \mbox{($p \neq 2,3$ if $\l=\l_3$)} \\ 
{\rm (v)} & D_4 & C_1^3.Z_3 & \l_1+\l_4,\, \l_3+\l_4 & p=2 \\ 
{\rm (vi)} & D_8 & C_1^4.X & \l_7 & \mbox{$p \neq 3$, $X<{\rm Sym}_4$ is transitive} \\ \hline
\end{array}$$
\caption{The irreducible triples $(G,H,V)$ in Theorem \ref{t:main}}
\label{tab:main}
\end{center}
\end{table}
\renewcommand{\arraystretch}{1}

\begin{remark}\label{r:main}
In case (i) of Table \ref{tab:main}, $T_n$ denotes a maximal torus of $G$. In all cases, $H^0$ is the connected component of a maximal subgroup of $G$, with the exception of the cases labelled (ii) and (iii), where $H$ is contained in a subgroup $D_4<G$. Also note that in cases (iii) to (vi), $W|_{H^0}$ is the tensor product of the natural modules of the simple components of $H^0$. In case (ii), $H$ is the image of a subgroup $C_1^3.X<C_4$ as in (iii), under an isogeny $\varphi:C_4 \to B_4$. In cases (v) and (vi), we record $H$ and $V$ up to ${\rm Aut}(G)$-conjugacy (so in case (vi) for example, if $\tilde{H}$ is the image of $H$ under a graph automorphism of $G$, then $V_G(\l_8)$ is an irreducible $K\tilde{H}$-module). Finally, let us note that the situation in part (b) of Theorem \ref{t:main} is very special and we refer the reader to Section \ref{s:spin} for further details.
\end{remark}

\begin{remark}
If $V|_{H^0}$ is irreducible, then $(G,H^0,V)$ is one of the cases in \cite[Table 1]{Seitz2} and
$$H^0 \leqs H \leqs N_G(H^0) \mbox{ and } C_G(H^0) \leqs Z(G).$$
More precisely, we are either in the situation described in part (b) of Theorem \ref{t:main}, or 
in terms of Seitz's notation in \cite[Table 1]{Seitz2}, one of the following holds (modulo scalars):
\begin{itemize}\addtolength{\itemsep}{0.2\baselineskip}
\item[(a)] $N_G(H^0) = H^0$;
\item[(b)] $N_G(H^0) = H^0.2$ and $(G,H^0,V)$ is one of the cases labelled I$_4$, I$_5$, I$_6$ (with $n=3$ in the notation of \cite[Table 1]{Seitz2}), II$_1$, S$_1$, S$_7$, MR$_1$, MR$_4$;
\item[(c)] $N_G(H^0) = H^0.{\rm Sym}_3=D_4.{\rm Sym}_3$ and $(G,H^0,V)$ is the case labelled S$_8$.
\end{itemize}
Here we refer the reader to the proof of \cite[Theorem 2.5.1]{BGMT} to see that $N_G(H^0) = H^0.2$ in the case labelled S$_7$, and that $N_G(H^0)=H^0.{\rm Sym}_3$ in case S$_8$.
\end{remark}

Let us briefly describe our approach to the proof of Theorem \ref{t:main}. Suppose $V|_{H^0}$ is reducible. If $H$ is a maximal subgroup of $G$ then $(G,H,V)$ can be read off from the main theorems of \cite{BGMT,BGT}, so let us assume $H < M < G$ with $M$ maximal. Since $V|_{H}$ is irreducible, it follows that $V|_{M}$ is also irreducible and we can consider the possibilities for the irreducible triple $(G,M,V)$, which are determined in \cite{Seitz2} (if $M$ is connected) and \cite{BGMT,BGT} (if $M$ is disconnected). We can then proceed by studying the possible embeddings of $H$ in $M$. 

Our next result is a combination of the main theorems in \cite{Seitz2, BGMT, BGT}, together with Theorem \ref{t:main}. Note that we assume $n \geqs 3$ if $G=B_n$, $n \geqs 2$ if $G=C_n$, and $n \geqs 4$ if $G=D_n$.

\begin{corol}\label{t:main4}
Let $G$ be a simply connected cover of a simple classical algebraic group over an algebraically closed field $K$ of characteristic $p \geqs 0$. Let $H$ be a positive-dimensional closed subgroup of $G$, and let $V=V_G(\l)$ be a nontrivial $p$-restricted irreducible tensor indecomposable rational $KG$-module such that $V|_{H}$ is irreducible. Then one of the following holds:
\begin{itemize}\addtolength{\itemsep}{0.2\baselineskip}
\item[{\rm (a)}] $V=W^{\tau}$ for some automorphism $\tau$ of $G$, where $W$ is the natural module;
\item[{\rm (b)}] $G$ is of type $B_n$ or $D_n$ (or type $C_n$ if $p=2$) and $V$ is a spin module;
\item[{\rm (c)}] $(G,H,V)$ is recorded in Table \ref{t:maint3}.
\end{itemize}
\end{corol}

\begin{remark}\label{r:main4}
Let us make some comments on the cases in Table \ref{t:maint3}:
\begin{itemize}\addtolength{\itemsep}{0.2\baselineskip}
\item[{\rm (i)}] In Table \ref{t:maint3}, we write $X$ for ${\rm Sym}_3$ or $Z_3$, $Z$ for $Z_2$ or $1$, and $Y$ denotes any $k$-transitive subgroup of ${\rm Sym}_{n+1}$.
\item[{\rm (ii)}] For $G$ of type $A_n$, the highest weight $\l$ is recorded up to conjugacy by a graph automorphism of $G$ (this is consistent with \cite[Table 1]{Seitz2}). For instance, in the first row of the table, we have $G = A_n$, $H=T_n.Y$ and $\l=\l_k$ with $2 \leqs k \leqs (n+1)/2$. By applying a suitable graph automorphism, we see that $V_G(\l_k)|_{H}$ is irreducible for all $2 \leqs k \leqs n-1$. 
\item[{\rm (iii)}] In the fourth column, we describe the restriction of $\l$ to a suitable maximal torus of the derived subgroup $[H^0,H^0]$ (this is denoted by $\l|_{H^0}$), in terms of highest weights for the simple components. The one exception is the case $(G,H)=(A_n,T_n.Y)$, where 
$H^0=T_n$ is a maximal torus of $G$ and thus $[H^0,H^0]=1$.
\item[{\rm (iv)}] In the fifth column, $\kappa$ denotes the number of $KH^0$-composition factors of $V|_{H^0}$.
\item[{\rm (v)}] In the final column, we record various conditions on $G$, $H$ and $\l$ that are necessary and sufficient for the irreducibility of $V|_{H}$. In a few cases, the conditions for irreducibility are rather complicated and so we record them here (the example in case (b) was discovered by Ford in \cite{Ford1}): 

\vspace{1mm}

\begin{itemize}\addtolength{\itemsep}{0.2\baselineskip}
\item[(a)] $(G,H,\l) = (A_{2m-1},C_m,a\l_k+b\l_{k+1})$: $1 \leqs k < m$; $a+b=p-1$; $p \ne 2$; $a \ne 0$ if $k=m-1$.
\item[(b)] $(G,H,\l) = (B_n,D_n.2,\sum_{i}a_i\l_i)$: $p \ne 2$; $a_n=1$; if $a_i,a_j \neq 0$, where $i<j<n$ and $a_k=0$ for all $i<k<j$, then $a_i+a_j \equiv i-j \imod{p}$; if $i<n$ is maximal such that $a_i \neq 0$, then $2a_i \equiv -2(n-i)-1 \imod{p}$.
\item[(c)] $(G,H,\l) = (C_{2m},C_m^2.2,\l_{2m-1}+a\l_{2m})$: $0 \leqs a < p$; $(m,a) \ne (1,0)$; $2a+3 \equiv 0 \imod{p}$.
\item[(d)] $(G,H,\l) = (D_n,B_{n-1}, a\l_k+b\l_{n-i})$: $1 \leqs k \leqs n-2$; $i \in \{0,1\}$; $ab \ne 0$; $a+b+n-k-1 \equiv 0 \imod{p}$. 
\end{itemize}
\end{itemize}
\end{remark}

\renewcommand{\arraystretch}{1.1}
\footnotesize
\begin{table}
\begin{center}
$$\begin{array}{llllcl} \hline
G & H & \l & \l|_{H^0} & \kappa & \mbox{Conditions} \\ \hline
A_n & T_n.Y & \l_k & - & \binom{n+1}{k} & 2 \leqs k \leqs (n+1)/2 \\
A_{2m-1} & C_{m} & k\l_1 & k\o_1 & 1 & m \geqs 2, \, k \geqs 2 \\
A_{2m-1} & C_{m} & a\l_k+b\l_{k+1} & a\o_k+b\o_{k+1} & 1 & \mbox{See Remark \ref{r:main4}(v)(a)} \\
A_{2m} & B_{m} & \l_k & \o_k & 1 & 2 \leqs k < m, \, p \ne 2 \\
A_{2m} & B_{m} & \l_{m} & 2\o_{m} & 1 & m \geqs 2, \, p \ne 2 \\ 
A_{2m-1} & D_{m}.Z & \l_k & \o_k & 1 & 2 \leqs k <m-1, \, p \ne 2 \\
A_{2m-1} & D_{m}.Z & \l_{m-1} & \o_{m-1}+\o_{m} & 1 & m \geqs 4, \, p \ne 2 \\
A_{2m-1} & D_{m}.2 & \l_{m} & 2\o_{m} & 2 & m \geqs 2, \, p \ne 2 \\
A_{m(m+2)} & A_m^2.2 & \l_2 & \o_{2} \otimes 2\o_{1} & 2 & m \geqs 2, \, p \ne 2 \\ 
A_{(m^2+m-2)/2} & A_m & \l_2 & \o_1+\o_3 & 1 & m \geqs 3, \, p \ne 2 \\
A_{m(m+3)/2} & A_m & \l_2 & 2\o_1+\o_2 & 1 & m \geqs 2, \, p \ne 2 \\
A_{26} & E_6 & \l_2 & \o_3 & 1 & p \ne 2 \\
A_{26} & E_6 & \l_3 & \o_4 & 1 & p \ne 2,3 \\
A_{26} & E_6 & \l_4 & \o_2+\o_5 & 1 & p \ne 2,3 \\
A_{15} & D_5 & \l_2 & \o_3 & 1 & p \ne 2 \\
A_{15} & D_5 & \l_3 & \o_2+\o_4 & 1 & p \ne 2,3 \\
& & & & & \\
B_n & D_n.Z & \sum_{i=1}^{n-1}a_i\l_i & \sum_{i=1}^{n-2}a_i\o_i + a_{n-1}(\o_{n-1}+\o_n) & 1 & n \geqs 3, \, p = 2 \\
B_n & D_n.2 & \sum_{i=1}^n a_i\l_i &\sum_{i=1}^{n-1}a_i\o_i + (a_{n-1}+1)\o_n & 2 & \mbox{See Remark \ref{r:main4}(v)(b)} \\
B_{12} & F_4 & 2\l_1 & 2\o_4 & 1 & p = 3 \\
B_6 & C_3 & 2\l_1 & 2\o_2 & 1 & p = 3 \\
B_4 & B_1^3.X & \l_3 & \o_{1} \otimes \o_{1} \otimes 3\o_{1} & 3 & p = 2 \\
B_3 & G_2 & \l_2 & \o_2 & 1 & p = 2 \\
B_3 & G_2 & \l_1+\l_2 & \o_1+\o_2 & 1 & p = 2 \\
B_3 & G_2 & k\l_1 & k\o_1 & 1 & k \geqs 2, \, p \ne 2 \\
B_3 & G_2 & a\l_2+b\l_3 & b\o_1+a\o_2 & 1 & ab \ne 0, \, 2a+b+2 \equiv 0 \, (p) \\
B_3 & G_2 & a\l_1+b\l_2 & a\o_1 + b\o_2 & 1 & a \geqs 2,\, b \geqs 1,\, a+b+1 \equiv 0 \, (p) \\
B_3 & A_2.Z & 2\l_1 & 2\o_1+2\o_2 & 1 & p = 3 \\
& & & & & \\
C_n & D_n.Z & \sum_{i=1}^{n-1}a_i\l_i & \sum_{i=1}^{n-2}a_i\o_i + a_{n-1}(\o_{n-1}+\o_n) & 1 & n \geqs 3, \, p = 2 \\
C_{2m} & C_m^2.2 & \l_{2m-1}+a\l_{2m} & (a+1)\o_{m}\oplus (\o_{m-1}+a\o_{m}) & 2 & \mbox{See Remark \ref{r:main4}(v)(c)} \\ 
C_{28} & E_7 &  \l_2 & \o_6 & 1 & p \ne 2 \\
C_{28} & E_7 &  \l_3 & \o_5 & 1 & p \ne 2,3 \\
C_{28} & E_7 &  \l_4 & \o_4 & 1 & p \ne 2,3 \\
C_{28} & E_7 &  \l_5 & \o_2+\o_3 & 1 & p \ne 2,3,5 \\
C_{16} & D_6 & \l_2 & \o_4 & 1 & p \ne 2 \\
C_{16} & D_6 & \l_3 & \o_3 + \o_5 & 1 & p \ne 2,3 \\
C_{10} & A_5.Z & \l_2 & \o_2+\o_4 & 1 & p \ne 2 \\
C_{10} & A_5.2 & \l_3 & \o_1+2\o_4 & 2 & p \ne 2,3 \\
C_7  & C_3 & \l_2 & 2\o_2 &  1 & p \ne 2,7 \\
C_7 & C_3 & \l_3 & \o_1+2\o_2 & 1 & p \ne 2,3 \\
C_4 & C_1^3.X & \l_2 & 0 \otimes 2\o_{1} \otimes 2\o_{1} & 3 & p \ne 2 \\
C_4 & C_1^3.X & \l_3 & \o_{1} \otimes \o_{1} \otimes 3\o_{1} & 3 & p \ne 3 \\
C_3 & G_2 & \l_2 & \o_2 & 1 & p = 2 \\
C_3 & G_2 & \l_1+\l_2 & \o_1+\o_2 & 1 & p = 2 \\
& & & & & \\
D_n & B_{n-1} & k\l_{n-i} & k\o_{n-i} & 1 & \mbox{$k \geqs 2$, $i \in \{0,1\}$} \\
D_n & B_{n-1} & a\l_k+b\l_{n-i} & a\o_k+b\o_{n-i} & 1 & \mbox{See Remark \ref{r:main4}(v)(d)} \\
D_{2m} & (D_m^2.2).2 & \l_1+\l_{n-i} & (\o_{1}+\o_m) \oplus \o_{m-i} & 4 & \mbox{$m \geqs 3$ odd, $i \in \{0,1\}$, $p=2$} \\
D_4 & C_1^3.X & \l_i+\l_4 & \o_{1} \otimes \o_{1} \otimes 3\o_{1} & 3 & \mbox{$i \in \{1,3\}$, $p=2$} \\ \hline
\end{array}$$
\caption{Positive-dimensional irreducible subgroups of classical groups}
\label{t:maint3}
\end{center}
\end{table}
\renewcommand{\arraystretch}{1}
\normalsize

\vs

Our final result concerns chains of irreducibly acting subgroups. Let $G$ and $V$ be given as in Hypothesis \ref{h:our} and write $\ell=\ell(G,V)$ for the length of the longest chain of closed positive-dimensional subgroups
$$H_{\ell}<H_{\ell-1} < \cdots < H_2 < H_1 = G$$
such that $V|_{H_{\ell}}$ is irreducible. We call such a sequence of subgroups an \emph{irreducible chain}. If $G$ is an orthogonal group (or a symplectic group with $p=2$) and $V$ is a spin module, then $\ell(G,V)$ can be arbitrarily large, and it is easy to see that the same is true if $V=W$ or $W^*$.  

\begin{theorem}\label{t:chains}
Suppose $G$ and $V$ satisfy the conditions in Hypothesis \ref{h:our} and assume $V$ is not a spin module. Then either $\ell(G,V) \leqs 5$, or $G=A_n$ and $\l \in \{\l_2, \l_3, \l_{n-2}, \l_{n-1}\}$.
\end{theorem}

The upper bound in this theorem is best possible. In fact, if we exclude the exceptional cases then either $\ell(G,V) \leqs 4$, or $G \in \{B_4,C_4\}$, $\l=\l_3$, $p=2$ and 
$$A_1^3.Z_3 < A_1^3.{\rm Sym}_3 < D_4 < D_4.2 < G$$
is an irreducible chain of length $5$ (see Theorem \ref{t:chai} for a more precise statement). The exceptions with $G=A_n$ are genuine in the strong sense that $\ell(G,V)$ can be arbitrarily large. We refer the reader to Section \ref{s:chains} for further details.

\section{Preliminaries}\label{s:prel}

\subsection{Notation and terminology}\label{ss:nota}

Most of our notation is fairly standard. As in Hypothesis \ref{h:our}, let $G$ be a simply connected cover of a simple classical algebraic group, which is defined over an algebraically closed field $K$ of characteristic $p \geqs 0$. Fix a Borel subgroup $B=UT$ of $G$, where $T$ is a maximal torus of $G$ and $U$ is the unipotent radical of $B$. Let $\Pi(G) =\{\a_1, \ldots, \a_{n}\}$ be the corresponding base of the root system $\Sigma(G)$ of $G$, where $n$ denotes the rank of $G$. Let 
$\{\l_1, \ldots, \l_{n}\}$ be the fundamental dominant weights for $T$ corresponding to $\Pi(G)$.

There is a bijection between the set of dominant weights of $G$ and the set of isomorphism classes of irreducible $KG$-modules; if $\l$ is a dominant weight then we use $V_G(\l)$  to denote the unique irreducible $KG$-module with highest weight $\l$.
We also recall that if $p>0$ then a dominant weight $\l = \sum_{i}a_i\l_i$ is $p$-\emph{restricted} if $a_i<p$ for all $i$. By Steinberg's tensor product theorem, every irreducible $KG$-module decomposes in a unique way as a tensor product $V_0 \otimes V_1^{\s_p} \otimes \cdots \otimes V_r^{\s_{p^r}}$, where $V_i$ is a $p$-restricted irreducible $KG$-module, $\s_{p^i}:G \to G$ is a standard Frobenius morphism (with $\s_{p^i}=1$ if $p=0$), and $V_i^{\s_{p^i}}$ (which we will also denote by $V_i^{(p^i)}$) is the $KG$-module obtained by preceding the action of $G$ on $V_i$ by the endomorphism $\s_{p^i}$. It is convenient to say that every dominant weight is $p$-restricted if $p=0$.

In addition, ${\rm Lie}(G)$ denotes the Lie algebra of $G$, and $U_{\a} = \{x_{\a}(t) \mid t \in K\}$ is the root subgroup of $G$ corresponding to a root $\a \in \Sigma(G)$. If $x \in G$ then $t_x : G \to G$ is the inner automorphism of $G$ induced by conjugation by $x$, so $t_x(g) = xgx^{-1}$ for all $g \in G$. We write $T_i$ for an $i$-dimensional torus. If $H$ is a closed positive-dimensional subgroup of $G$ and $T_{H^0}$ is a maximal torus of $[H^0,H^0]$ contained in $T$, then we abuse notation by writing $\mu|_{H^0}$ to denote the restriction of a $T$-weight $\mu$ to the subtorus $T_{H^0}$.  We define a partial order $\preccurlyeq$ on the set of weights for $T$, where $\mu \preccurlyeq \nu$ if and only if $\mu = \nu - \sum_{i=1}^{n}c_i\a_i$ for some non-negative integers $c_i$ (in this situation, we say that $\mu$ is \emph{under} $\nu$). Finally, we set $\mathbb{N}_0 = \mathbb{N} \cup \{0\}$, we write ${\rm Sym}_{n}$ and ${\rm Alt}_{n}$ for the symmetric and alternating groups of degree $n$, and we denote a cyclic group of order $m$ by $Z_m$ (or just $m$). 

Recall that a map $\varphi:G_1 \to G_2$ of algebraic groups is a \emph{morphism} if it is a group homomorphism that is also a morphism of the underlying varieties. In particular, it is important to note that an injective morphism does not necessarily induce an isomorphism $G_1 \cong \varphi(G_1)$ of algebraic groups. If $G_1$ is a semisimple algebraic group with root system $\Phi$, then we will say that $G_1$ is of \emph{type} $\Phi$ (and we will sometimes denote this by writing $G_1 = \Phi$). For example, ${\rm SL}_{2}(K)$ and ${\rm PGL}_{2}(K)$ are both simple algebraic groups of type $A_1 = B_1 = C_1$, and ${\rm PSp}_{4}(K)$ and ${\rm SO}_{5}(K)$ are both of type $B_2=C_2$. Finally, note that if $H$ is a closed positive-dimensional subgroup of an algebraic group $G$, and $\varphi:H \to G$ is the inclusion map, then the differential $d\varphi:{\rm Lie}(H) \to {\rm Lie}(G)$ is an injective Lie algebra homomorphism (since $\varphi:H \to \varphi(H)$ is an isomorphism of algebraic groups).

\subsection{Diagonal embeddings}\label{ss:diag}

Let $G/Z$ be a central product, where $G = G_1 \times \cdots \times G_t$ and $Z \leqs Z(G)$. A subgroup $H/Z$ of $G/Z$ is a \emph{subdirect product} if each of the projection maps $\pi_i:H \to G_i$ is surjective. In the context of algebraic groups, the related notion of a diagonally embedded subgroup is defined as follows:

\begin{defn}\label{d:de}
Let $H$ be a closed subgroup of $G = G_1\times \cdots \times G_t$ where the $G_i$ are isomorphic simply connected simple algebraic groups. We say that $H$ is \emph{diagonally embedded} in $G$ if each projection $\pi_i: H \to G_i$ is a bijective morphism. Note that we do not require each projection map $\pi_i$ to induce an isomorphism $H \cong \pi_i(H)$ of algebraic groups.
\end{defn} 

The next lemma is a well known result of Steinberg (see \cite[Theorem 30]{steinberg} and \cite[10.13]{steinberg2}), which describes the bijective endomorphisms of a simple algebraic group. Here $t_x$ and $\sigma_q$ are defined as in Section \ref{ss:nota}, and we adopt Steinberg's definition of a \emph{graph automorphism} of a simple algebraic group $G$ (see \cite[Section 10]{steinberg}). In particular, a graph automorphism is an isomorphism of algebraic groups unless $(G,p) = (C_2,2)$, $(G_2,3)$ or $(F_4,2)$.

\begin{lem}\label{l:steinberg}
Let $G$ be a simple algebraic group over an algebraically closed field of characteristic $p \geqs 0$. Let $\varphi:G \to G$ be a bijective morphism. Then $\varphi = t_{x}\sigma_q\gamma^k$ for some $x \in G$, $p$-power $q$ and integer $k \in \{0,1\}$, where $\gamma$ is a graph automorphism of $G$. Moreover, if $G$ is classical and $(G,p) \neq (C_2,2)$, then $\varphi$ is an isomorphism of algebraic groups if and only if $\s_q=1$.
\end{lem}

\begin{lem}\label{l:neww}
Let $\varphi:H \to G$ be a surjective morphism of algebraic groups and let $d\varphi:{\rm Lie}(H) \to {\rm Lie}(G)$ be the corresponding differential map. Then $d\varphi({\rm Lie}(H))$ is a $KG$-submodule of ${\rm Lie}(G)$, and hence also an ideal of ${\rm Lie}(G)$.
\end{lem}

\begin{proof}
Let ${\rm Ad}_G: G \to {\rm GL}({\rm Lie}(G))$ be the adjoint representation of $G$. 
We must consider ${\rm Ad}_G(g)(d\varphi(X))$, for $g\in G$ and $X \in {\rm Lie}(H)$. As above, let $t_g:G\to G$ denote conjugation by $g$. Then ${\rm Ad}_G(g)(d\varphi(X)) = (dt_g\circ d\varphi)(X) = d(t_g\circ\varphi)(X)$. Since $\varphi$ is surjective, $g = \varphi(h)$ for some $h\in H$, so we have 
$${\rm Ad}_G(g)(d\varphi(X))=d(t_{\varphi(h)}\circ\varphi)(X) = d(\varphi\circ t_h)(X) = 
d\varphi\circ {\rm Ad}_H(h)(X)\in d\varphi({\rm Lie}(H)).$$ 
Therefore, $d\varphi({\rm Lie}(H))$ is ${\rm Ad}_G$-invariant and hence a $KG$-submodule of ${\rm Lie}(G)$.

Finally, let $V$ be a $KG$-module with corresponding representation $\rho:G\to {\rm GL}(V)$, and let $S$ be a $G$-invariant subspace of $V$. Then $S$ is invariant under the action of $d\rho({\rm Lie}(G))$. We conclude that $d\varphi({\rm Lie}(H))$ is an ideal of ${\rm Lie}(G)$.
\end{proof}
 
Recall that a morphism $\varphi:H \to G$ of algebraic groups is an \emph{isogeny} if it is surjective with finite kernel. If such a map exists, we say that $H$ is \emph{isogenous} to $G$ (this is not a symmetric relation).

\begin{lem}\label{l:isog}
Let $G$ be a simply connected simple classical algebraic group of rank $m$ over an algebraically closed field $K$ of characteristic $p \geqs 0$, let $H$ be a connected algebraic group and let $\varphi:H \to G$ be an isogeny. Then $\varphi$ is a bijection. Moreover, if $d\varphi \neq 0$ then either 
$\varphi$ is an isomorphism of algebraic groups, or one of the following holds:
\begin{itemize}\addtolength{\itemsep}{0.2\baselineskip}
\item[{\rm (i)}] $G$ and $H$ are both of type $A_m$, with $p$ dividing $m+1$; 
\item[{\rm (ii)}] $G$ and $H$ are both of type $B_m$, $C_m$ or $D_m$, with $p=2$;
\item[{\rm (iii)}] $(G,H)$ is of type $(B_m,C_m)$ or $(C_m,B_m)$, with $p=2$.
\end{itemize}
\end{lem}

\begin{proof}
First we claim that $H$ is also a simple group of rank $m$. Clearly, if $N$ is a proper nontrivial connected normal subgroup of $H$, then $\varphi(N)$ is a  
proper nontrivial connected normal subgroup of $G$, which is not possible since $G$ is simple. Therefore, $H$ is simple. If $T_H$ is a maximal torus of $H$, then $\varphi(T_H)$ is a maximal torus of $G$ (see \cite[Proposition 11.14]{Borel}, for example), and $\dim T_H = \dim \varphi(T_H)$. Therefore, $H$ has rank $m$. Now, by comparing dimensions, we deduce that $G$ and $H$ have the same root system, unless $p=2$ and $(G,H) = (B_m,C_m)$ or $(C_m,B_m)$. Note that if $p \neq 2$ and $G=B_m$ then $H$ is also of type $B_m$ because an isogeny from $B_m$ to $C_m$ only exists when $p=2$. Similarly, if $p \neq 2$ and $G=C_m$ then $H$ is of type $C_m$. 

To see that $\varphi$ is a bijection (of abstract groups), first observe that $\ker(\varphi) \leqs Z(H)$ since $H$ is simple, so the claim is trivial if $p=2$ and $H= B_m$ or $C_m$. Now assume $p \neq 2$ if $H=B_m$ or $C_m$. As above, $G$ and $H$ have the same root system. In particular, if $H_{{\rm sc}}$ denotes the simply connected group with the same root system as $H$, then $H_{{\rm sc}}$ and $G$ are isomorphic algebraic groups (this follows from the classification of simple algebraic groups over $K$, using the fact that $G$ is simply connected). Set $\psi = \varphi \circ \pi$, where $\pi:H_{{\rm sc}} \to H$ is the natural isogeny. Then $\psi: H_{{\rm sc}} \to G$ is an isogeny with kernel $L \leqs Z(H_{{\rm sc}})$, so $H_{{\rm sc}}/L \cong G$ as abstract groups. In particular, $Z(H_{{\rm sc}}/L) \cong Z(G) \cong Z(H_{{\rm sc}})$, so $L=1$ is the only possibility. Therefore $\psi$ is injective, and thus $\varphi$ is also injective. We conclude that $\varphi$ is a bijection.   

To complete the proof, we may assume that $d\varphi \ne 0$ and $(G,H,p)$ is not one of the cases labelled (i) -- (iii) in the statement of the lemma. As above, $G$ and $H$ are both simple groups of the same type and rank. Since ${\rm Lie}(H)$ is simple (see \cite[Table 1]{Hoge}), it follows that $d\varphi$ is an isomorphism of Lie algebras and thus $\varphi$ is an isomorphism of algebraic groups.
\end{proof}

\begin{lem}\label{l:basecase}
Let $J$ be a closed connected subdirect product of $G_1\times G_2$, where $G_1$ and $G_2$ are isomorphic simply connected simple classical algebraic groups. Then $J$ is diagonally embedded in $G_1 \times G_2$, and either $J=G_1 \times G_2$ or $J \cong G_1$ as algebraic groups.
\end{lem}

\begin{proof}
Let $\pi_i:J \to G_i$ be the $i$-th projection map and set $L = \ker(d\pi_1) \cap \ker(d\pi_2)$.
Note that $L = 0$ since $J$ is a closed positive-dimensional subgroup of $G_1 \times G_2$. Without loss of generality, we will assume that $d\pi_1 \neq 0$. 

First assume $\ker(\pi_1)$ is infinite. Since $\pi_2$ is injective on $\ker(\pi_1)$, we have
$\dim \ker(\pi_1)  = \dim \pi_2(\ker(\pi_1))$. Moreover, the surjectivity of $\pi_2$ implies that 
$\pi_2(\ker(\pi_1))$ is an infinite normal subgroup of $G_2$, so the simplicity of $G_2$ implies that $\pi_2(\ker(\pi_1)) = G_2$ and thus $\dim \ker(\pi_1) = \dim G_2$. Therefore, $\dim J = \dim G_1+  \dim G_2$ and we conclude that $J=G_1\times G_2$. 

For the remainder, we may assume that $\ker(\pi_1)$ is finite, so $\pi_1$ is an isogeny and Lemma \ref{l:isog} implies that $\pi_1$ is a bijection and either $J\cong G_1$, or $(G_1,J)$ is one of the cases labelled (i) -- (iii). In particular, $J$ is simple and $\ker(\pi_2)$ is finite. By a further application of Lemma \ref{l:isog}, we see that $\pi_2$ is also a bijection and thus $J$ is diagonally embedded. To complete the proof it remains to show that $J \cong G_1$ as algebraic groups. Seeking a contradiction, let us assume that $\ker(d\pi_i) \ne 0$ for $i = 1,2$, so ${\rm im}(d\pi_i)$ is a proper non-zero ideal of ${\rm Lie}(G_i)$ (see Lemma \ref{l:neww}).

First assume $p=2$ and $(G_1,J)$ is of type $(B_m,B_m)$, $(C_m,C_m)$, $(B_m,C_m)$ or $(C_m,B_m)$. For $m \geqs 2$, the ideal structure of ${\rm Lie}(J)$ is described in \cite[Section 5]{DS}. Excluding the case where $J = C_m$ is simply connected and $m \geqs 3$ is odd, we observe that ${\rm Lie}(J)$ has an irreducible socle $S$ (as a $KJ$-module), which immediately implies that $L$ contains $S$. This is a contradiction, since $L=0$. Now assume $J = C_m$ is simply connected and $m \geqs 3$ is odd. If $G_1=C_m$ then $J \cong G_1$ since $G_1$ is simply connected, so let us assume $G_1=B_m$. The socle of ${\rm Lie}(J)$ is of the form $Z \oplus M$, where $Z=Z({\rm Lie}(J))$ is $1$-dimensional and $M$ is a nontrivial irreducible module. Without loss of generality, we may assume that $M$ is not contained in $\ker(d\pi_1)$, which implies that $\ker(d\pi_1) = Z$. Therefore ${\rm im}(d\pi_1)$ is an ideal of ${\rm Lie}(G_1)$ of codimension $1$, but this is not compatible with the ideal structure described in \cite{DS}. Therefore, once again we have reached a contradiction. Finally, if  $J = B_1=C_1$ is adjoint then ${\rm Lie}(J)$ has an irreducible socle and we can repeat the argument given above.

Next suppose that $G_1$ and $J$ are both of type $A_m$, where $p$ divides $m+1$. We may assume that $m \geqs 2$. Seeking a contradiction, suppose that $J$ is not simply connected. By inspecting \cite[Table 1]{Hoge} we deduce that ${\rm im}(d\pi_i) = Z({\rm Lie}(G_i))$ and ${\rm ker}(d\pi_i)$ is the commutator subalgebra of ${\rm Lie}(J)$ for $i=1,2$. But this implies that $L\ne 0$, which is a contradiction. To complete the proof, we may assume that $G_1$ and $J$ are both of type $D_m$, with $p=2$. If $m$ is odd then we can repeat the previous argument, using \cite[Table 1]{Hoge}, so let us assume $m$ is even. If $J$ is not simply connected then ${\rm Lie}(J)$ has an irreducible socle $S$ (see \cite[Section 5]{DS}), which must be contained in $L$. Once again, this is a contradiction. 
\end{proof}

The next result is a natural generalization of Lemma \ref{l:basecase}.

\begin{prop}\label{p:useful}
Let $J$ be a closed connected subdirect product of $G_1\times \dots \times G_t$, where the $G_i$ are isomorphic simply connected simple classical algebraic groups.
Then the following hold:
\begin{itemize}\addtolength{\itemsep}{0.2\baselineskip}
\item[{\rm (i)}] $J$ is semisimple.
\item[{\rm (ii)}] There exists a positive integer $r\leqs t$ such that $J=J_1\cdots J_r$, 
where each $J_i$ is isomorphic to $G_1$.
\item[{\rm (iii)}] There exist integers $0=t_0<  t_1 < t_2 < \dots< t_r = t$ 
such that $J_{i}$ is diagonally embedded in $G_{t_{i-1}+1}\times \cdots \times G_{t_i}$.
\end{itemize}
\end{prop}

\begin{proof}
We use induction on $t$, noting that the case $t=1$ is trivial, and Lemma \ref{l:basecase} handles the case $t=2$. Let us assume $t \geqs 3$, and let $\pi_i:J \to G_i$ be the $i$-th projection map. As in the proof of the previous lemma, we may assume that $d\pi_1 \neq 0$.

If $\ker(\pi_i)$ is finite for any $i$, then $\dim J = \dim G_1$ and thus $\ker(\pi_1)$ is also finite. Then by arguing as in the proof of Lemma \ref{l:basecase}, we deduce that $J$ is diagonally embedded in $G_1 \times \cdots \times G_t$ and $J \cong G_1$. For the remainder, we may assume that $\ker(\pi_i)$ is infinite for all $i$. Since $\pi_i(R_u(J))$ is a proper normal subgroup of $\pi_i(J) = G_i$ for each $i$, it follows that $J$ is reductive. Similarly, by considering $\pi_i(Z(J))$, we deduce that $Z(J)$ is finite and thus $J$ is semisimple. In particular, we may write
$$J = J_1 \cdots J_r,$$
where each $J_i$ is simple. Note that $\pi_i(J_j)$ is a connected normal subgroup of $G_i$ for all $i,j$, so $\pi_i(J_j)=1$ or $G_i$. Let $\sigma$ be the projection map 
$$\sigma:J\rightarrow G_2\times \cdots \times G_t.$$
We now consider two cases.

\vs

\noindent \emph{Case 1.} $\ker(\sigma)$ is infinite.

\vs

First assume $\ker(\sigma)$ is infinite, so $\ker(\sigma)^0$ is a connected positive-dimensional normal subgroup of $J$. By relabelling the $J_i$, if necessary, we may assume that   
$$\ker(\sigma)^0 = J_1 \cdots J_a$$
for some $a \in \{1, \ldots, r\}$ (see \cite[Theorem 27.5(c)]{Humphreys}, for example). Now $\pi_1(J_i) = G_1$ for all $1 \leqs i \leqs a$, so the injectivity of $\pi_1$ on $\ker(\sigma)^0$, together with the simplicity of $G_1$, implies that $a=1$ and $J_1$ is of type $G_1$. In particular, $r>1$ since we are assuming that $\ker(\pi_1)$ is infinite. Also note that $\pi_i(J_2 \cdots J_r) = G_i$ for all $i \geqs 2$.

We claim that $\pi_1(J_i)=1$ for all $2 \leqs i \leqs r$, so $J_2 \cdots J_r \leqs G_2 \times \cdots \times G_t$ is a subdirect product and the result follows by induction. To justify the claim, let $i \in \{2, \ldots, r\}$ and consider $\pi_1|_{J_1J_i}: J_1J_i \to G_1$. Since $\pi_1(J_1)=G_1$, we have $\pi_1(J_1J_i)=G_1$.  Now $\ker(\pi_1|_{J_1J_i})^0$ is a connected normal subgroup of $J_1J_i$, so $\ker(\pi_1|_{J_1J_i})^0=1, J_1, J_i$ or $J_1J_i$. It is easy to see that $\ker(\pi_1|_{J_1J_i})^0=J_i$ is the only possibility, so $\pi_1(J_i)=1$ as claimed.

\vs

\noindent \emph{Case 2.} $\ker(\sigma)$ is finite.

\vs

To complete the proof, we may assume that $\ker(\sigma)$ is finite. Now $J/\ker(\sigma)$ is connected and reductive, and it is isomorphic to a subdirect product of $G_2 \times \cdots \times G_t$. By induction, there exists $s \in \{1, \ldots, t-1\}$ such that $J/\ker(\sigma) = L_1\cdots L_{s}$ and $L_i \cong G_1$ for each $i$. But $J=J_1 \cdots J_r$ and the $J_i$ are simple, so $r=s$ and $\dim J_i = \dim G_1$ for all $i$ (indeed, $J_i$ is isogenous to $G_1$), hence $\dim J= r \dim G_1$. Note that $r>1$ since $\ker(\pi_1)$ is infinite. 

Since $\pi_1: J \to G_1$ is surjective, we have $\dim \ker(\pi_1) = (r-1)\dim G_1$. Moreover, by relabelling the $J_i$ if necessary, we may assume that $\ker(\pi_1)^0= J_2\cdots J_r$ (see \cite[Theorem 27.5(c)]{Humphreys}). Therefore, $\pi_1(J_1) = G_1$ and $J_2 \cdots J_r \leqs G_2\times \cdots \times G_t$. By a further relabelling, we may assume that there exists an integer $b \in \{1, \ldots, t\}$ such that $\pi_i(J_1)=G_i$ for $1\leqs i \leqs b$, and $\pi_i(J_1)=1$ if $b<i\leqs t$. 

First we claim that $b<t$. Seeking a contradiction, suppose that $b=t$, so $\pi_i(J_1)=G_i$ for all $i$. Let $j \in \{2, \ldots, r\}$ and consider $\pi_i|_{J_1J_j}: J_1J_j \to G_i$. By arguing as above, we deduce that $\ker(\pi_i|_{J_1J_j})^0=J_j$, so $\pi_i(J_j) = 1$ for all $i$. Therefore $J_j=1$ and thus $r=1$, which is a contradiction.

Since $b<t$ and $J_1 \leqs G_1 \times \cdots \times G_b$ is a subdirect product, by induction we deduce that $J_1$ is diagonally embedded in $G_1 \times \cdots \times G_b$ and  $J_1 \cong G_1$. If we fix $i \in \{1, \ldots, b\}$ and $j \in \{2, \ldots, r\}$, then $\ker(\pi_i|_{J_1J_j})^0=J_j$ and thus $J_2 \cdots J_r \leqs \ker(\pi_i)$. Therefore, $J_2 \cdots J_r \leqs G_{b+1} \times \cdots \times G_t$ is a subdirect product (since $\pi_i(J)=G_i$ and $\pi_i(J_1) = 1$ for all $i>b$) and the result follows by induction. 
\end{proof}

\subsection{Irreducible triples}\label{ss:irred}

Define $G,H$ and $V$ as in Hypothesis \ref{h:our}. The next result records a basic observation (see \cite[Remark 2]{BGT}).

\begin{lem}\label{l:hred}
If $V|_{H}$ is irreducible, then $H$ does not normalize a nontrivial connected unipotent subgroup of $G$. In particular, $H^0$ is reductive.
\end{lem}

Suppose $V|_{H}$ is irreducible, but $V|_{H^0}$ is reducible. Then Clifford theory implies that 
\begin{equation}\label{e:vx}
V|_{H^0}=V_1 \oplus \cdots \oplus V_m,
\end{equation}
where $m \geqs 2$ and the $V_i$ are irreducible $KH^0$-modules that are transitively permuted under the induced action of $H/H^0$. 

\begin{remk}\label{r:semisimple}
Since the irreducibility of $V|_{H}$ implies that $H^0$ is reductive, we have $H^0 = JZ(H^0)$ where $J = [H^0,H^0]$ is the derived subgroup of $H^0$. Now $Z(H^0)$ acts as scalars on the $V_i$ so that they are also irreducible on restriction to $J$. In particular, the irreducibility of $V|_{H}$ implies that the $KJ$-composition factors of $V|_{J}$ are transitively permuted under the induced action of $H/H^0$.
\end{remk}

If the $V_i$ in \eqref{e:vx} are isomorphic as $KH^0$-modules, then $V|_{H^0}$ is said to be \emph{homogeneous}. For example, $V|_{H^0}$ is homogeneous if $N_G(H^0) = H^0C_G(H^0)$. The following result is \cite[Proposition 2.6.2]{BGMT}.

\begin{prop}\label{p:niso}
If $H$ is a cyclic extension of $H^0$, then the irreducible $KH^0$-modules $V_i$ in \eqref{e:vx} are pairwise non-isomorphic. In particular, $V|_{H^0}$ is not homogeneous.
\end{prop}

We will also need the following lemma.

\begin{lem}\label{l:ten}
Let $V_1$ and $V_2$ be $p$-restricted irreducible $KG$-modules and set $V= V_1 \otimes V_2$. Then one of the following holds:
\begin{itemize}\addtolength{\itemsep}{0.2\baselineskip}
\item[{\rm (i)}] $V$ is irreducible, $G = B_n$ or $C_n$, $p = 2$ and $V_1$ and $V_2$ can be arranged so that $V_i = V_G(\mu_i)$ and $\mu_1$ (respectively $\mu_2$) has support on the short (respectively, long) roots;
\item[{\rm (ii)}] $V$ has non-isomorphic composition factors.
\end{itemize}
\end{lem}

\begin{proof}
By \cite[(1.6)]{Seitz2}, $V$ is irreducible if and only if $G$ and $V$ satisfy the conditions in (i), so let us assume $V$ is reducible. Write $V_i = V_G(\mu_i)$ and note that $V$ has a composition factor of highest weight $\mu = \mu_1+\mu_2$ occurring with multiplicity $1$. Then any other composition factor has highest weight $\nu \ne \mu$ and the result follows.
\end{proof}

\section{Subgroup structure}\label{s:ss}

\subsection{A reduction theorem}\label{ss:rt}

Let $G$ be a simple classical algebraic group with natural module $W$. Following \cite[Section 1]{LS}, we introduce six natural, or \emph{geometric}, collections of closed subgroups, labelled $\C_i$ for $1 \leqs i \leqs 6$, and we set $\C=\bigcup_{i}\C_i$. These subgroups are defined in terms of the underlying geometry of $W$, and a rough description of the subgroups in each $\C_i$ collection is given in Table \ref{t:subs} (note that the subgroups in the collection $\C_5$ are finite). There are two types of subgroups in the $\C_4$ collection (indicated by the two rows in Table \ref{t:subs}); following \cite{BGT}, we write $\C_4 = \C_4(i) \cup \C_4(ii)$ accordingly. The following result is \cite[Theorem 1]{LS} (we use the term \emph{non-geometric} for the subgroups arising in part (ii)).

\renewcommand{\arraystretch}{1.1}
\begin{table}
\begin{center}
\begin{tabular}{cl} \hline
 & Rough description \\ \hline
$\C_1$ & Stabilizers of subspaces of $W$ \\
$\C_2$ & Stabilizers of orthogonal decompositions $W=\bigoplus_{i}W_i$, $\dim W_i=a$ \\
$\C_3$ & Stabilizers of totally singular decompositions $W=W_1 \oplus W_2$ \\
$\C_4$ & Stabilizers of tensor product decompositions $W=W_1 \otimes W_2$ \\
 & Stabilizers of tensor product decompositions $W=\bigotimes_i W_i$, $\dim W_i=a$ \\
$\C_5$ & Normalizers of symplectic-type $r$-groups, $r \neq p $ prime  \\
$\C_6$ & Classical subgroups \\ \hline
\end{tabular}
\caption{The $\C_i$ collections}
\label{t:subs}
\end{center}
\end{table}
\renewcommand{\arraystretch}{1}

\begin{thm}\label{t:ls}
Let $G$ be a simple classical algebraic group with natural module $W$, and let $H$ be a closed subgroup of $G$. Then one of the following holds:
\begin{itemize}\addtolength{\itemsep}{0.2\baselineskip}
\item[{\rm (i)}] $H$ is contained in a member of $\C$;
\item[{\rm (ii)}] modulo scalars, $H$ is almost simple and $E(H)$ (the unique quasisimple normal subgroup of $H$) is irreducible on $W$. Further, if $G={\rm SL}(W)$ then $E(H)$ does not fix a non-degenerate form on $W$. In addition, if $H$ is infinite then $E(H)=H^0$ is 
tensor indecomposable on $W$.
\end{itemize}
\end{thm}

\subsection{Geometric subgroups of ${\rm GO}(W)$}\label{ss:dn2}

In our inductive proof of Theorem \ref{t:main}, we will need to consider the subgroup structure of $G = {\rm GO}(W)$, which is the full isometry group of a non-degenerate quadratic form on $W$. Here $\dim W = 2n \geqs 6$ and thus $G^0 = {\rm SO}(W)$ is a simple group of type $D_n$. The notion of a \emph{geometric subgroup} extends naturally to $G$ and we can define the subgroup collections $\C_1, \ldots, \C_6$ as above. It is straightforward to check that the proof of the main theorem of \cite{LS} extends to this slightly more general situation (see \cite[Theorem 1$'$]{LS}), and thus Theorem \ref{t:ls} holds. In particular, any subgroup of $G$ that is not contained in a geometric subgroup is said to be \emph{non-geometric}, and these subgroups satisfy the conditions described in part (ii) of Theorem \ref{t:ls}. 

In the proofs of Propositions \ref{p:geom3} and \ref{p:geom6}, we need information on the maximal non-parabolic geometric subgroups of $G$. This is given in the following proposition.

\begin{prop}\label{p:dn2}
Let $M$ be a positive-dimensional non-parabolic geometric subgroup of $G={\rm GO}(W) = D_n.2$, which is not contained in $G^0$. Then the possibilities for $M$ are recorded in Table \ref{tab:dn2}.
\end{prop}

\begin{proof}
Here $M$ is a disconnected $\C_i$-subgroup of $G$, where $i \in \{1,2,3,4\}$ (recall that the subgroups in $\C_5$ are finite, and there are no $\C_6$-subgroups in orthogonal groups). The structure of $M$ is easily determined from the geometric description of $M$ (see \cite[Section 2.5]{BGT}, for example), and it is straightforward to determine whether or not $M$ is contained in $G^0$.

For example, if $M \in \C_1$ then $M = G_U$ is the stabilizer of a subspace $U$ of $W$ (the natural $KG$-module), and one of the following holds (recall that $M$ is non-parabolic, so $U$ is not totally singular):
\begin{itemize}\addtolength{\itemsep}{0.2\baselineskip}
\item[(a)] $U$ is non-degenerate and $\dim U$ is even;
\item[(b)] $U$ is non-degenerate, $\dim U$ is odd and $p \neq 2$;
\item[(c)] $U$ is non-singular, $\dim U = 1$ and $p=2$.
\end{itemize}
In (a) and (b), $M = {\rm GO}(U) \times {\rm GO}(U^{\perp})$ is not contained in $G^0$. Similarly, in (c), $M = B_{n-1} \times 2$ (up to isomorphism) is not in $G^0$. These are the cases labelled (i), (ii) and (iii) in Table \ref{tab:dn2}. 

Next suppose $M$ is a $\C_4(i)$ tensor product subgroup of type ${\rm SO}(W_1) \otimes {\rm SO}(W_2)$, where $W= W_1 \otimes W_2$, $\dim W_i=a_i$, $a_1 \neq a_2$ and $p \neq 2$. Note that 
$$M = {\rm GO}(W_1) \circ {\rm GO}(W_2) = ({\rm SO}(W_1) \circ {\rm SO}(W_2)).\la x_1, x_2 \ra$$
is a central product and the $x_i$ are certain involutions. More precisely, if $a_1$ and $a_2$ are both even, then we may assume that $x_1$ acts as a reflection on $W_1$ and centralizes $W_2$ (and vice versa for $x_2$). Therefore, $x_1, x_2 \in {\rm SO}(W)$ and thus $M<G^0$. On the other hand, if $a_1$ is odd (so $a_2$ is even) then we can choose $x_1$ so that it acts as $-1$ on $W_1$ and centralizes $W_2$, and $x_2$ is defined as above. Here $x_1 \in G^0$ but $x_2 \not\in G^0$, whence $M$ is not contained in $G^0$. This is the case labelled (vii) in Table \ref{tab:dn2}. 

The other cases are similar. For instance, suppose $M$ is a $\C_3$-subgroup. Geometrically, $M$ is the stabilizer of a decomposition $W = U_1 \oplus U_2$, where $U_1$ and $U_2$ are maximal totally singular subspaces of $W$, so $\dim U_1=n$ and $M = {\rm GL}(U_1).2$. Now $G^0$ contains an element interchanging $U_1$ and $U_2$ if and only if $n$ is even, so $M$ is contained in $G^0$ if and only if $n$ is even, and this explains the $n$ odd condition recorded in Table \ref{tab:dn2} (see case (vi)).
\end{proof}

\renewcommand{\arraystretch}{1.1}
\begin{table}
\begin{center}
\begin{tabular}{llcl} \hline
& $M$ & Collection & Conditions \\ \hline
(i) & $D_lD_{n-l}.2^2$ & $\mathcal{C}_1$ &  $1 \leqs l < n/2$ \\ 
(ii) & $B_{l}B_{n-l-1} \times 2^2$ & $\mathcal{C}_1$ & $0 \leqs l < n/2$, $p \neq 2$ \\
(iii) & $B_{n-1} \times 2$ & $\C_1$ & $p=2$ \\
(iv) & $(2^t \times B_l^t).{\rm Sym}_t$ & $\C_2$ & $2n=(2l+1)t$, $l \geqs 1$, $t \geqs 2$ even, $p \neq 2$ \\
(v) & $(D_l^t.2^t).{\rm Sym}_t$ & $\C_2$ & $n=lt$, $l \geqs 1$, $t \geqs 2$ \\
(vi) & $A_{n-1}T_1.2$ & $\C_3$ & $n$ odd \\
(vii) & $B_aD_b.2$ & $\C_4(i)$ & $n=(2a+1)b$, $a \geqs 1$, $b \geqs 2$, $p \neq 2$ \\ \hline
\end{tabular}
\caption{The non-parabolic geometric subgroups $M<D_n.2$ with $M \not\leqs D_n$}
\label{tab:dn2}
\end{center}
\end{table} 
\renewcommand{\arraystretch}{1}

\begin{prop}\label{p:dn2_1}
Let $M$ be one of the subgroups of $G=D_n.2$ listed in Table \ref{tab:dn2}, and assume $n \geqs 3$ and $p \neq 2$. Set $V=V_{G^0}(\l)$, where one of the following holds:
\begin{itemize}\addtolength{\itemsep}{0.2\baselineskip}
\item[{\rm (a)}] $\l = \l_{n-1}+\l_{n}$;
\item[{\rm (b)}] $\l = \l_{k}$, where $1<k<n-1$.
\end{itemize}
Then $V$ extends to a representation of $G$, and $V|_{M}$ is reducible.
\end{prop}

\begin{proof}
First observe that $\l$ is fixed under the induced action of an involutory graph automorphism of $G^0$ on the set of $T$-weights of $G$ (where $T$ is a maximal torus of $G^0$), so the representation $V=V_{G^0}(\l)$ does indeed extend to a representation of $G=G^0.2$. We will deal in turn with each of the relevant cases in Table \ref{tab:dn2} (note that case (iii) is not applicable, since we are assuming that $p \neq 2$). Seeking a contradiction, let us assume that $V|_{M}$ is irreducible. By Clifford theory (see Section \ref{ss:irred}), the $KM^0$-composition factors of $V$ are transitively permuted under the induced action of $M/M^0$.

\vs

\noindent \emph{Case 1.} $M$ is a $\C_1$-subgroup of type $D_lD_{n-l}.2^2$:

\vs

Here $1 \leqs l < n/2$ and $M^0=M_1M_2$, where $M_1 = D_l$ and $M_2 = D_{n-l}$. We will inspect the proof of \cite[Lemma 3.2.3]{BGT}. By \cite{Seitz2}, $V|_{M^0}$ is reducible and thus the Clifford theory implies that there are either two or four $KM^0$-composition factors (since $|M:M^0|=4$).

First assume $l=1$, so $M = M^0\la \tau_1,\tau_2\ra$ where $\tau_1$ is an involution inverting the $1$-dimensional torus $M_1$, and 
$\tau_2$ is an involutory graph automorphism of $M_2 = D_{n-1}$. We may assume that $M_2 = \la U_{\pm \a_2}, \ldots, U_{\pm \a_n}\ra$. Note that $M_1$ acts as scalars on the $KM^0$-composition factors of $V$, each of which is an irreducible $KM_2$-module. If there are exactly two $KM^0$-composition factors of $V$ then the argument in the proof of \cite[Lemma 3.2.3]{BGT} goes through unchanged (the details are given in the proof of \cite[Lemma 3.2.2]{BGT}), and the result follows immediately. 

Similar reasoning applies if there are four composition factors. By Clifford theory, if $\nu$ is the highest weight of a $KM^0$-composition factor, then $\nu|_{M_2} = \l|_{M_2}$ or $(\tau_2 \cdot \l)|_{M_2}$ (here $\nu|_{M_2}$ denotes the restriction of $\nu$ to a suitable maximal torus of $M_2$ contained in $T$, and similarly for $\l|_{M_2}$ and $(\tau_2 \cdot \l)|_{M_2}$). However, 
$\mu = \lambda-\a_1-\a_2-\cdots-\a_{n-1} = -\l_1+2\l_{n}$
affords the highest weight of a $KM^0$-composition factor in case (a), but clearly $\mu|_{M_2}$ is not conjugate to $\l|_{M_2}$. Case (b) is entirely similar, using $\mu = \lambda-\a_1-\cdots-\a_k$. Therefore, in both cases we have reached a contradiction.

Now assume $l \geqs 2$. As noted in the proof of \cite[Lemma 3.2.3]{BGT}, up to conjugacy we have
\begin{equation}\label{e:m1m2}
\begin{split}
M_{1} & =\la U_{\pm \a_{1}},\ldots,U_{\pm \a_{l-1}},U_{\pm (\a_{l-1}+2(\a_l+\cdots+\a_{n-2})+\a_{n-1}+\a_{n})} \ra \\
M_{2} & =\la U_{\pm \a_{l+1}},\ldots,U_{\pm \a_{n}} \ra
\end{split}
\end{equation}
and $M=M^0\la \tau_1,\tau_2 \ra$,
where $\tau_1$ and $\tau_2$ act as involutory graph automorphisms on $M_1$ and $M_2$, respectively. Let $\{\o_{1,1},\ldots,\o_{1,l}\}$ and $\{\o_{2,1},\ldots,\o_{2,n-l}\}$ be the fundamental dominant weights corresponding to the above bases of the root systems of $M_1$ and $M_2$, respectively (here $\tau_1$ acts as a transposition on $\{\o_{1,1},\ldots,\o_{1,l}\}$,  interchanging $\o_{1,l-1}$ and $\o_{1,l}$, and similarly $\tau_2$ acts on $\{\o_{2,1},\ldots,\o_{2,n-l}\}$ by interchanging the weights $\o_{2,n-l-1}$ and $\o_{2,n-l}$). Note that if $\mu = \sum_{i=1}^{n}{b_i\l_i}$ is a weight for $T$ then
\begin{equation}\label{e:weq}
\mu|_{M^0} = \sum_{i=1}^{l-1} b_{i}\o_{1,i}+(b_{l-1}+2b_{l}+\dots+2b_{n-2}+b_{n-1}+b_{n})\o_{1,l}+\sum_{i=1}^{n-l} b_{l+i}\o_{2,i}.
\end{equation}

Consider case (a). Here $\mu = \l - \a_{l}-\a_{l+1}-\cdots - \a_{n-2}-\a_{n-1}$ affords the highest weight of a $KM^0$-composition factor (see the proof of \cite[Lemma 3.2.3]{BGT}) and we calculate that
$$\mu|_{M^0} = \o_{1,l-1}+\o_{1,l}+2\o_{2,n-l},\;\; 
\l|_{M^0} = 2\o_{1,l}+\o_{2,n-l-1}+\o_{2,n-l}.$$
In particular, we observe that 
$$\mu|_{M^0} \not\in \{\l|_{M^0},(\tau_1 \cdot \l)|_{M^0},(\tau_2 \cdot \l)|_{M^0}, (\tau_1\tau_2 \cdot \l)|_{M^0}\},$$
so $\mu|_{M^0}$ is not $M$-conjugate to $\l|_{M^0}$. This is a contradiction.

In case (b) we have $\l=\l_k$, where $1<k<n-1$, and thus
$$\l|_{M^0} = \left\{\begin{array}{ll}
\o_{1,k} & 1<k<l-1 \\
\o_{1,l-1}+\o_{1,l} & k=l-1 \\
2\o_{1,l} & k=l \\
2\o_{1,l}+\o_{2,k-l} & l<k<n-1. 
\end{array}\right.$$
Set 
$$\mu = \left\{\begin{array}{ll}
\l-\a_{k} - \a_{k+1} - \cdots - \a_{l} & 1<k<l \\
\l-\a_{l} & k=l \\
\l-\a_{l} - \a_{l+1} - \cdots - \a_{k} & l<k<n-1.
\end{array}\right.$$
Then $\mu$ affords the highest weight of a $KM^0$-composition factor, and in each case it is easy to check that $\mu|_{M^0}$ is not $M$-conjugate to $\l|_{M^0}$. For example, suppose $1<k<l-1$. Then $\mu = \l_{k-1}-\l_{l}+\l_{l+1}$ and thus $\mu|_{M^0} = \o_{1,k-1}+\o_{2,1}$, which is not conjugate to $\l|_{M^0} = \o_{1,k}$. The other cases are similar. 

\vs

\noindent \emph{Case 2.} $M$ is a $\C_1$-subgroup of type $B_lB_{n-l-1} \times 2^2$:

\vs

Here $0\leqs l<n/2$ and $M^0=M_1M_2$, where $M_1=B_{l}$ and 
$M_2=B_{n-l-1}$.  Let $\{\beta_1,\ldots,\beta_l\}$ and $\{\gamma_1,\ldots,\gamma_{n-l-1}\}$ be bases of
the root systems of $M_1$ and $M_2$, respectively, and let $\{\eta_1,\dots,\eta_l\}$ and $\{\nu_1,\dots,\nu_{n-l-1}\}$ be the corresponding fundamental
dominant weights. Then up to conjugacy, we may assume that the simple root elements of $M_1$ and $M_2$ are as follows
$$x_{\beta_i}(t) = \left\{\begin{array}{ll}
x_{\alpha_i}(t) & 1\leqs i<l \\
x_{\delta}(t)x_{\epsilon}(t) & i = l 
\end{array}\right.$$
$$ 
x_{\gamma_j}(t) = \left\{\begin{array}{ll}
x_{\alpha_{l+j}}(t) & 1\leqs j< n-l-1 \\
x_{\alpha_{n-1}}(t)x_{\alpha_n}(t) & j = n-l-1
\end{array}\right.$$
for all $t \in K$, where 
$\delta = \alpha_l+\alpha_{l+1}+\cdots+\alpha_{n-1}$ and $\epsilon =  \alpha_l+\alpha_{l+1}+\cdots+\alpha_{n-2}+\alpha_n$ (see \cite[Claim 8]{Test2}, for example). Note that $V|_{M^0}$ is homogeneous. 

First consider (a). In terms of the above notation, we calculate that $\lambda|_{M^0}=2\eta_l+2\nu_{n-l-1}$. By considering the restrictions $\a_i|_{M^0}$, we see that $\lambda-\alpha_{n-1}$ 
and $\lambda-\alpha_n$ both restrict to the weight  
$\lambda|_{M^0}-\gamma_{n-l-1}$. This weight has multiplicity $1$ in the 
$KM^0$-composition factor of $V$ afforded by $\lambda$. Moreover, one checks that $\l$ is the only $T$-weight $\mu$ in $V$ such that $\l|_{M^0}-\gamma_{n-l-1} \preccurlyeq \mu|_{M^0}$ and $\l|_{M^0}-\gamma_{n-l-1} \neq \mu|_{M^0}$, so there must be a $KM^0$-composition factor with highest weight $\lambda|_{M^0}-\gamma_{n-l-1}$. However, this contradicts the homogeneity of $V|_{M^0}$.

Now consider (b). Suppose
$k\leqs l$, so $l>0$ and $\l|_{M_2}$ is trivial. Now the weight $\mu = \lambda-\alpha_k-\cdots-\alpha_l$ 
affords the highest weight of a $KM^0$-composition factor of $V$, but $\mu|_{M_2}$ is nontrivial and this contradicts 
the homogeneity of $V|_{M^0}$. Now assume $k>l$. Here $\lambda|_{M^0}=2\eta_l+\nu_{k-l+1}$. 
However, the weight $\mu = \lambda-\alpha_l-\alpha_{l+1}-\cdots-\alpha_k$ affords the highest weight of a $KM^0$-composition factor of $V$ and 
$$\mu|_{M^0} = \left\{\begin{array}{ll}
\eta_{l-1}+\nu_{k-l+2} & k<n-2 \\
\eta_{l-1}+2\nu_{n-l-1} & k=n-2
\end{array}\right.$$
Once again, this contradicts 
the homogeneity of $V|_{M^0}$.
 
\vs

\noindent \emph{Case 3.} $M$ is a $\C_2$-subgroup of type $(2^t \times B_l^t).{\rm Sym}_t$:

\vs

Here $2n=(2l+1)t$, $l \geqs 1$ and $t \geqs 2$ is even. Note that the conclusion to \cite[Lemma 4.2.1]{BGT} still applies in this situation.

First consider (a). If $l=1$ then the argument in the third paragraph on \cite[p.48]{BGT} applies, and we reach a contradiction via \cite[Lemma 4.2.1]{BGT}. Next suppose $(l,t)=(2,2)$, so $n=5$. Here we argue as in the third to last paragraph on \cite[p.50]{BGT}. (Alternatively, in the notation of \cite[Lemma 4.3.8]{BGT}, note that $\dim V = \dim V_{D_5}(\l) = 210$ and $\l|_{M^0} = 2\o_{1,2}+2\o_{2,2}$, so the $KM^0$-composition factor afforded by $\l$ has dimension $10^2=100$, which does not divide $\dim V$.) Finally, if $l \geqs 2$ and $(l,t) \neq (2,2)$ then we can argue as in the third to last paragraph on \cite[p.52]{BGT} (again, we get a contradiction via \cite[Lemma 4.2.1]{BGT}). Case (b) is similar and we omit the details.

\vs

\noindent \emph{Case 4.} $M$ is a $\C_2$-subgroup of type $(D_l^t.2^t).{\rm Sym}_t$:

\vs

Here $n=lt$, $l \geqs 1$ and $t \geqs 2$. As in the previous case, note that the conclusion to \cite[Lemma 4.2.1]{BGT} still applies.

If $l=1$ then $M = N_{G}(T)$ is the normalizer of a maximal torus and this case is ruled out as in the first paragraph in the proof of \cite[Lemma 4.3.8]{BGT}. Now assume $l \geqs 2$. Consider case (a). If $l \geqs 3$ then we can argue as on \cite[p.56]{BGT} to rule out this case, and as noted in the penultimate paragraph on \cite[p.57]{BGT}, the same argument also applies if $l=2$. Case (b) is very similar: if $l \geqs 3$ then we argue as on \cite[p.55]{BGT} (we repeatedly apply \cite[Lemma 4.2.1]{BGT}), and for $l=2$ 
we note that the argument on \cite[p.57]{BGT} goes through unchanged.

\vs

\noindent \emph{Case 5.} $M$ is a $\C_3$-subgroup of type $A_{n-1}T_1.2$:

\vs

Here $n \geqs 3$ is odd. Set $L = (M^0)'=A_{n-1}$ and note that $V|_{L}$ has exactly two composition factors. We may assume that $L = \la U_{\pm \a_1}, \ldots, U_{\pm \a_{n-1}}\ra$. As in the proof of \cite[Lemma 3.2.2]{BGT}, let $V_j$ be the sum of the $T$-weight spaces ($T$ a maximal torus of $G^0=D_n$) in $V=V_{G^0}(\l)$ of the form $\l - \sum_{i=1}^{n-1}c_i\a_i - j \a_{n}$, $j \in \mathbb{N}_0$. Since $V_j$ is $L$-stable, every $T$-weight of $V$ is of this form, with $j=0$ or $1$ (by \cite[Theorem 1]{Premet} and saturation; see \cite[Section 13.4]{Humphreyslie}). In particular, $w_0\l = -(\tau \cdot \l)$ is of this form (where $w_0$ is the longest word in the Weyl group of $G^0$, and 
$\tau$ is an involutory graph automorphism of $G^0$ that interchanges the weights $\l_{n-1}$ and $\l_n$). Therefore, if we write $\l = \sum_{i=1}^{n}a_i\l_i$ then  
$$2\l - (a_{n-1}-a_n)(\l_{n-1}-\l_n) = \sum_{i=1}^{n}c_i\a_i$$
and $c_n \in \{0,1\}$. By expressing the $\l_i$ in terms of the $\a_i$, we deduce that either $\l=\l_1$, or $n=3$ and $\l = \l_2$ or $\l_3$. This immediately eliminates cases (a) and (b). 

\vs

\noindent \emph{Case 6.} $M$ is a $\C_4(i)$-subgroup of type $B_aD_b.2$:

\vs

Here $n = (2a+1)b$ and $M^0=M_1M_2$ is semisimple, where $M_1$ is of type $B_a$ and $M_2$ is of type $D_b$. Note that the embedding of $M$ in $G$ is via a tensor product action on the natural $KG$-module $W$. Write $M=M^0\la \s\ra$, where $\s$ induces a graph automorphism on $M_2$ and centralizes $M_1$. 

Let $\Pi(M_1)=\{\b_1, \ldots,\b_a\}$ and $\Pi(M_2)=\{\gamma_1,\ldots,\gamma_b\}$
be bases of the root systems $\Sigma(M_1)$ and $\Sigma(M_2)$, respectively. Now $W$ restricts to $M_1$ as $2b$ copies of the natural module for 
$M_1$, and hence up to conjugacy, we may assume that $M_1$ lies in the subgroup 
$$\langle U_{\pm\a_i}\mid (2a+1)j+1\leqs i < (2a+1)(j+1), \, 
0\leqs j < b \rangle,$$
the derived subgroup of an $A_{2a}\times\cdots\times A_{2a}$ ($b$ copies) Levi subgroup 
of $G$. The projection of $M_1$ into each of the factors of this group
is the natural embedding of a group of type $B_a$ in $A_{2a}$. 
We may assume that 
$$\a_{(2a+1)j+i}|_{M^0}=\b_i \; \mbox{for all $0\leqs j < b$, $1\leqs i\leqs a$}$$
and
$$\a_{(2a+1)j+a+i}|_{M^0}=\b_{a-i+1} \; \mbox{for all $0\leqs j< b$, $1\leqs i\leqs a$,}$$
where $\a|_{M^0}$ denotes the restriction of $\a$ to a maximal torus $T_{M^0} < T$ of $M^0$. 

Let $P$ be the parabolic subgroup of $M^0$, which contains the opposite Borel subgroup, with Levi factor $M_1T_{M^0}$. We may assume that $P$ is contained in the parabolic subgroup of $G$, which contains the opposite Borel
subgroup of $G$, whose Levi factor has derived subgroup as given above. By comparing the flags of commutator subspaces of $W$ with respect to the two unipotent radicals, we are able to determine the restrictions of sufficiently many $T$-weights to $T_{M^0}$ in order to deduce the restrictions of the remaining simple roots. We get
$$\a_{(2a+1)j}|_{M^0}=\gamma_j-\b_0,\;\; \a_{(2a+1)b}|_{M^0}= \gamma_b-\gamma_{b-1}-(\b_0-\b_1),$$
where $1\leqs j< b$ and $\b_0=2\sum_{i=1}^a \b_i$. 
 
By \cite{Seitz2}, $V|_{M^0}$ is reducible, so the Clifford theory implies that $V$ has precisely two $KM^0$-composition factors, with highest weights
$\l|_{M^0}$ and $(\sigma \cdot \l)|_{M^0}$. In particular, if we set  $T_{M_1}=T_{M^0} \cap M_1$, then every $KM_1$-composition factor of $V$ has highest weight $\l|_{M_1}$, whence
every $T_{M_1}$-weight of $V$ is of the form
\begin{equation}\label{restriction1}
\l|_{M_1}-\sum_{j=1}^{a}n_j\b_j,\; \mbox{for some $n_j\in\mathbb{N}_0$.}
\end{equation}
However, in case (a) we find that the weight $\lambda-\alpha_n$ restricts to 
$\l|_{M_1} + \b_0-\b_1$, which contradicts \eqref{restriction1}. In (b), choose $1\leqs i<b$ such that $|(2a+1)i-k|$ is minimal, and set $\mu = \l -a_k-\a_{k+1}-\cdots-\a_{(2a+1)i}$ if $k\leqs (2a+1)i$, otherwise set
$\mu = \lambda-\a_{(2a+1)i}-\a_{(2a+1)i+1}-\cdots-\a_k$. Then 
$$\mu|_{M_1} = (\l - r-(\gamma_i-\b_0))| _{M_1}= (\l - r+\b_0)|_{M_1},$$
where either $r=0$ or $r$ is a positive root of $M_1$. Once again, this contradicts \eqref{restriction1}.

\vs

This completes the proof of Proposition \ref{p:dn2_1}.
\end{proof}

\begin{prop}\label{p:dn2_2}
Let $M$ be one of the subgroups of $G=D_n.2$ listed in Table \ref{tab:dn2}, and assume $p= 2$. Set $V=V_{G^0}(\l)$, where 
$$\l = \sum_{i=1}^{n-2}a_i\l_i+a_{n-1}(\l_{n-1}+\l_n)$$ 
is $p$-restricted. Assume that $V$ is nontrivial and $V \neq W,W^*$. Then $V$ extends to a representation of $G$, and $V|_{M}$ is reducible.
\end{prop}

\begin{proof}
As before, $\l$ is fixed by an involutory graph automorphism of $G^0$, so $V$ extends to a representation of $G=G^0.2$. Seeking a contradiction, let us assume that $V|_{M}$ is irreducible. There are four cases to consider.

\vs

\noindent \emph{Case 1.} $M$ is a $\C_1$-subgroup of type $D_lD_{n-l}.2^2$:

\vs

Here $M^0=M_1M_2$, where $M_1=D_l$, $M_2=D_{n-l}$ and $1 \leqs l < n/2$. First assume $l=1$. As in the proof of the previous proposition, we may assume that $M_2=\la U_{\pm \a_{2}},\ldots,U_{\pm \a_{n}} \ra$, and by arguing as in the proof of \cite[Lemma 3.2.3]{BGT} we quickly reduce to the case where $V|_{M^0}$ has exactly four composition factors. Let $k$ be minimal such that $a_k\ne 0$. Then $\lambda-\a_1-\cdots-\a_k$ affords the highest weight of a $KM_2$-composition factor of $V$, which is not conjugate (via a graph automorphism of $M_2$) to the composition factor afforded by $\lambda$. This contradiction eliminates the case $l=1$.

Now assume $l \geqs 2$. As before, $M=M^0\la \gamma_1,\gamma_2 \ra$, where $\gamma_1$ and $\gamma_2$ act as involutory graph automorphisms on $M_1$ and $M_2$, respectively, and we may assume that 
$M_1$ and $M_2$ are as given in \eqref{e:m1m2}.  Let $\{\o_{1,1},\ldots,\o_{1,l}\}$ and $\{\o_{2,1},\ldots,\o_{2,n-l}\}$ be the fundamental dominant weights corresponding to the bases $\Pi(M_1)$ and $\Pi(M_2)$ in \eqref{e:m1m2}, respectively. In view of \eqref{e:weq}, it is easy to see that $\l|_{M^0} = (\gamma_2\cdot \l)|_{M^0}$ and
$$(\gamma_1\cdot \l)|_{M^0} = (\gamma_1\gamma_2\cdot \l)|_{M^0} = \l|_{M^0}+\left(2(a_{l}+ \cdots +a_{n-2})+a_{n-1}+a_{n}\right)(\o_{1,l-1}-\o_{1,l}).$$
Now, by arguing as in the proof of \cite[Lemma 3.2.3]{BGT} we quickly reduce to the case $\l = a_{n-1}(\l_{n-1}+\l_{n})$. By inspecting \eqref{e:m1m2}, we see that $\mu = \l - \a_{l} - \a_{l+1} - \cdots - \a_{n-2} - \a_{n-1}$ affords the highest weight of a $KM^0$-composition factor, but this is not conjugate to $\l|_{M^0}$ since
$\mu|_{M^0} = \l|_{M^0} + \o_{1,l-1} - \o_{1,l} - \o_{2,n-l-1} + \o_{2,n-l}$.

\vs

\noindent \emph{Case 2.} $M$ is a $\C_2$-subgroup of type $B_{n-1} \times 2$:

\vs

Write $M = M^0 \times \la z \ra$. By \cite{Seitz2}, $V|_{M^0}$ is reducible and thus $V|_{M^0} = V_1 \oplus V_2$, where $V_1$ and $V_2$ are irreducible $KM^0$-modules. Since $z$ is central, it follows that $V|_{M^0}$ is homogeneous, but this is ruled out by Proposition \ref{p:niso}.

\vs

\noindent \emph{Case 3.} $M$ is a $\C_2$-subgroup of type $(D_l^t.2^t).{\rm Sym}_t$:

\vs

Here $n=lt$, where $l \geqs 1$ and $t \geqs 2$. The case $l=1$ can be ruled out by arguing as in the first paragraph in the proof of \cite[Lemma 4.3.8]{BGT}. If $l \geqs 3$ then by arguing as in the proof of \cite[Lemma 4.3.8]{BGT} (see \cite[p.55]{BGT}) we reduce to the case $\l = \l_{n-1}+\l_n$, and this possibility is ruled out by the argument in the penultimate paragraph on \cite[p.56]{BGT}. Finally, if $l=2$ then we quickly get down to the cases $\l \in \{\l_1+\l_{n-1}+\l_n, \l_{n-1}+\l_n\}$. The case $\l = \l_{n-1}+\l_n$ is ruled out as in \emph{loc. cit.}, and the other case is eliminated by arguing as in the final paragraph in the proof of \cite[Lemma 4.3.8]{BGT}.

\vs

\noindent \emph{Case 4.} $M$ is a $\C_3$-subgroup of type $A_{n-1}T_1.2$:

\vs

Here $n \geqs 3$ is odd and the argument given in the analysis of Case 5 in the proof of Proposition \ref{p:dn2_1} can be applied.

\vs

This completes the proof of Proposition \ref{p:dn2_2}.
\end{proof}

\section{Geometric subgroups: The connected case}\label{s:geom1}

Suppose $G$, $H$ and $V=V_G(\l)$ satisfy the conditions in Hypothesis \ref{h:our}. In this section and the next, we will establish Theorem \ref{t:main} when $H$ is contained in a geometric maximal subgroup $M$ of $G$, 
excluding the special situation described in part (b) of Theorem \ref{t:main}.
Assume $V|_{H}$ is irreducible, so $H^0$ is reductive by Lemma \ref{l:hred}. Our first task is to determine the possibilities for the irreducible triple $(G,M,V)$. If $M$ is connected then we can read off the relevant cases by applying \cite[Theorem 1]{Seitz2}; the cases that arise are recorded in Table \ref{tab:geom1}. Similarly, if $M$ is disconnected, we can appeal to the main theorem of \cite{BGT}, which yields the list of cases given in Table \ref{tab:geom2} (in the first line of the table, $T$ denotes a maximal torus of $G$). 

\begin{remk}
It is worth noting that the cases listed in \cite[Table 1]{Seitz2} are recorded in terms of the image of the underlying representation $\varphi:G \to {\rm GL}(V)$, so Seitz's table gives $(\varphi(G),\varphi(M),V)$, rather than $(G,M,V)$. For instance, at the level of subgroups, the cases labelled S$_3$ and S$_4$ in \cite[Table 1]{Seitz2} correspond to irreducible triples with $(G,M) = (C_3,G_2)$ or $(B_3,G_2)$, but only the former possibility is listed in \cite[Table 1]{Seitz2} because in both cases the image $\varphi(G)$ is of type $C_3$. We also observe that there are certain maximal rank configurations arising in \cite[Theorem 4.1]{Seitz2} which are not listed \cite[Table 1]{Seitz2}. For example, referring to the case labelled MR$_1$ in \cite[Table 1]{Seitz2} (so $p=3$), Theorem 4.1 of \cite{Seitz2} implies that the short (respectively, long) root $A_2$ in $G_2$ acts irreducibly on any $p$-restricted $G_2$-module whose highest weight has support on the short (respectively, long) roots, but only the former is listed in the table. In addition, we also note that the highest weights in \cite[Table 1]{Seitz2} are only given up to conjugacy by a graph automorphism, and we adopt the same convention in Table \ref{tab:geom1}. Note that if the graph automorphism introduces a Frobenius twist on the module, then we will list the irreducible action on the corresponding $p$-restricted module.
\end{remk}

\begin{remk}
As noted in Remark \ref{r:hyp}, if $W$ denotes the natural $KG$-module then any irreducible triple $(G,M,V)$ with $V = W^{\tau}$ (for some $\tau \in {\rm Aut}(G)$) is also excluded in \cite[Table 1]{Seitz2}. In particular, Seitz does not list the cases $(G,M) = (D_4,A_2)$ (with $p \neq 3$) and $(B_2,A_1)$ ($p \ne 2,3$), with $V$ a spin module for $G$. For example, the spin module for $B_2$ is $4$-dimensional, and it corresponds to the natural symplectic representation of $C_2$. 
\end{remk}

\begin{remk}
The triples of the form $(G,M,V)$, where $(G,p) = (B_n,2)$ and $M$ is a disconnected geometric maximal subgroup, are not stated explicitly in \cite{BGT}, but they are easily determined from the relevant list of cases in \cite[Table 1]{BGT} for the corresponding dual group of type $C_n$. The only possibilities are $M=B_l^t.{\rm Sym}_t$ (a $\C_2$-subgroup) with $\l=\l_n$, or $M=D_n.2$ and $\l=\l_n$ or $\sum_{i<n}a_i\l_i$. Note that $G$ acts reducibly on $W$, so there are no triples involving non-geometric subgroups.
\end{remk}

\renewcommand{\arraystretch}{1.1}
\begin{table}
\begin{center}
\begin{tabular}{lllcll} \hline
&  $G$ & $M$ & Collection & $\l$ & Conditions \\ \hline
(i) & $A_n$ & $C_m$ & $\C_6$ & $k\l_1$, $k \geqs 2$ & $n=2m-1$, $m \geqs 2$, $p \neq 2$ \\ 
(ii) & & $C_m$ & $\C_6$ & See Remark \ref{r:geom}(a) & $n=2m-1$, $m \geqs 2$ \\
(iii) & & $B_m$ & $\C_6$ & $\l_k$, $1<k<m$ & $n=2m$, $m \geqs 3$, $p \neq 2$ \\
(iv) & & $B_m$ & $\C_6$ & $\l_m$ & $n=2m$, $m \geqs 2$, $p \neq 2$ \\
(v) & $D_n$ & $B_{m}$ & $\C_1$ & $k\lambda_{n-1}$, $k\l_n$, $k \geqs 2$ & $n=m+1$, $m \geqs 3$, $p \neq 2$ \\
(vi) & & $B_{m}$ & $\C_1$ & See Remark \ref{r:geom}(c) & $n=m+1$, $m \geqs 3$, $p \neq 2$ \\ \hline
\end{tabular}
\caption{$M$ is a connected geometric subgroup}
\label{tab:geom1}
\end{center}
\end{table} 
\renewcommand{\arraystretch}{1}

\renewcommand{\arraystretch}{1.1}
\begin{table}
\begin{center}
\begin{tabular}{lllcll} \hline
& $G$ & $M$ & Collection & $\l$ & Conditions \\ \hline
(i) & $A_n$ & $N_G(T)$ & $\C_2$ & $\l_k$, $1<k<n$ &  \\
(ii) & & $A_{m}^2.2$ & $\C_4(ii)$ & $\l_2,\l_{n-1}$ & $n=m(m+2)$, $p \neq 2$, $m \geqs 2$ \\
(iii) & & $D_m.2$ & $\C_6$ & $\l_k$, $1<k<n$ & $n=2m-1$, $p \neq 2$ \\
(iv) & $B_n$ & $D_n.2$ & $\C_1$ & See Remark \ref{r:geom}(e) & \\
(v) &  & $D_n.2$ & $\C_1$ & $\sum_{i=1}^{n-1}a_i\l_i$ & $p=2$ \\  
(vi) & $C_n$ & $C_m^2.2$ & $\C_2$ & See Remark \ref{r:geom}(f) & $n=2m$ \\
(vii) & & $D_n.2$ & $\C_6$ & $\sum_{i=1}^{n-1}a_i\l_i$ & $p=2$ \\
(viii) & $D_n$ & $(D_m^2.2).2$ & $\C_2$ & $\l_1+\l_{n-1}, \l_{1}+\l_n$ & $n=2m$, $m \geqs 3$ odd, $p=2$ \\
& & & & & \\
(ix) & $B_4$ & $B_1^2.2$ & $\C_4(ii)$ & $\l_4$ & $p \neq 3$ \\
(x) & $C_4$ & $C_1^3.{\rm Sym}_3$ & $\C_4(ii)$ & $\l_2, \l_3$ & $p \neq 2$ ($p \neq 2,3$ if $\l=\l_3$)\\
(xi) & $D_4$ & $C_1^3.{\rm Sym}_3$ & $\C_4(ii)$ & $\l_1+\l_4, \l_3+\l_4$ & $p = 2$ \\
(xii) & $D_8$ & $C_1^4.{\rm Sym}_4$ & $\C_4(ii)$ & $\l_7$ & $p \neq 3$ \\
(xiii) & & $C_2^2.2$ & $\C_4(ii)$ & $\l_7$ & $p \neq 5$ \\ \hline
\end{tabular}
\caption{$M$ is a disconnected geometric subgroup}
\label{tab:geom2}
\end{center}
\end{table} 
\renewcommand{\arraystretch}{1}

\begin{remk}\label{r:geom}
Let us make a few comments on the cases in Tables \ref{tab:geom1} and \ref{tab:geom2}:
\begin{itemize}\addtolength{\itemsep}{0.2\baselineskip}
\item[(a)] In case (ii) in Table \ref{tab:geom1} we have $G=A_n$ and $M=C_m$, where $n=2m-1$ and $m \geqs 2$. Moreover, $\l=a\l_k+b\l_{k+1}$, where $1 \leqs k < m$, $a+b=p-1>1$ and $a \neq 0$ if $k=m-1$. In particular, $p \neq 2$.

\item[(b)] Note that $H$ is a decomposition subgroup in case (v) of Table \ref{tab:geom1}, so we may assume that $k \geqs 2$ (if $k=1$ then $V$ is a spin module).

\item[(c)] In case (vi) in Table \ref{tab:geom1} we have $G=D_n$, $M=B_{m}$ (with $n=m+1 \geqs 4$, $p \neq 2$) and $\l=b\l_k+a\l_{n-1}$, where $1 \leqs k < n-1$, $ab \neq 0$ and $a+b+n-1-k \equiv 0 \imod{p}$.

\item[(d)] In Table \ref{tab:geom1}, following \cite[Table 1]{Seitz2}, for $G$ of type $A_n$ we record the highest weight $\l$ up to conjugacy by a graph automorphism.

\item[(e)] Consider case (iv) in Table \ref{tab:geom2}, where $G=B_n$ and $M=D_n.2$ is a $\C_1$-subgroup. Here $p \neq 2$ and the conditions on the highest weight $\l=\sum_{i=1}^{n}a_i\l_i$ are given in part (b) of Remark \ref{r:main4}(v). Since $H$ is a decomposition subgroup, we may assume that $a_i \neq 0$ for some $i<n$.

\item[(f)] In case (vi) in Table \ref{tab:geom2} we have $G=C_n$ and $M=C_m^2.2$ is a $\C_2$-subgroup, where $n=2m$. Moreover, $\l=\l_{n-1}+a\l_n$, where $0 \leqs a<p$ and $2a+3 \equiv 0 \imod{p}$. In particular, $p \ne 2$.

\item[(g)] In cases (xi), (xii) and (xiii) we record $M$ and $V$ up to ${\rm Aut}(G)$-conjugacy. For instance, in case (xi), if $\tilde{M}$ denotes the image of $M$ under an appropriate triality graph automorphism of $G$ then $(G,\tilde{M},V_G(\l_1+\l_3))$ is an irreducible triple.
\end{itemize}
\end{remk} 

The main result of this section is the following.

\begin{thm}\label{t:geom0}
Let $G$, $H$ and $V$ be given as in Hypothesis \ref{h:our}, and assume $H<M<G$ where $M$ is a connected geometric maximal subgroup of $G$. Then $V|_{H}$ is reducible.
\end{thm}

\begin{proof}
Write $V=V_G(\l)$. The possibilities for $(G,M,V)$ are given in Table \ref{tab:geom1}; in each case $M$ is a simple group of rank $m$. Let $\{\eta_1, \ldots, \eta_m\}$ be a set of fundamental dominant weights for $M$. Seeking a contradiction, let us assume that $V|_{H}$ is irreducible.

First consider case (i) in Table \ref{tab:geom1} (this is the case labelled I$_1$ in \cite[Table 1]{Seitz2}). Here $G=A_{n}$ and $M=C_m$ is a $\C_6$-subgroup of $G$, where $n=2m-1$, $m \geqs 2$ and $p \neq 2$ (this is the natural embedding ${\rm Sp}(W) < {\rm SL}(W)$). Note that $\l=k\l_1$, $k \geqs 2$ and $V|_{M} = V_M(k\eta_1)$ (see \cite[Table 1]{Seitz2}). Let $J$ be a maximal subgroup of $M$ containing $H$, so
$$H \leqs J <M<G.$$
We consider the irreducible triple $(M,J,V_M(k\eta_1))$. If $V|_{J^0}$ is irreducible then the triple  $(M,J^0,V_M(k\eta_1))$ has to be in \cite[Table 1]{Seitz2}, but it is easy to check that  there are no compatible examples. Therefore $J$ is disconnected and  $V|_{J^0}$  is reducible. In this situation, $(M,J,V_M(k\eta_1))$ must be one of the triples arising in the main theorems of \cite{BGMT,BGT}, but once again we find that there are no such triples. We conclude that $V|_{H}$ is reducible in case (i).

The other cases in Table \ref{tab:geom1} are very similar, although some extra care is required in case (vi). Here $G=D_n$, $M=B_{n-1}$ is a $\C_1$-subgroup (we can view $M$ as the stabilizer in $G$ of a $1$-dimensional 
non-degenerate subspace of $W$), and $\l=b\l_k+a\l_{n-1}$ satisfies the following conditions:
\begin{equation}\label{e:cond}
1 \leqs k < n-1, \;\; a,b \neq 0, \;\; a+b+n-1-k \equiv 0 \imod{p}
\end{equation}
(see case IV$_1'$ in \cite[Table 1]{Seitz2}). We note that $V|_{M} = V_M(b\eta_k+a\eta_{n-1})$. As before, let $J$ be a maximal subgroup of $M$ containing $H$, and consider the irreducible triple $(M,J,V_M(b\eta_k+a\eta_{n-1}))$. 

If $V|_{J^0}$ is irreducible then by inspecting \cite[Table 1]{Seitz2} we see that the only possibility is the case labelled III$_1'$, where $M=B_3$ (so $n=4$), $J=G_2$, $k=2$ and $a+2b+2 \equiv 0 \imod{p}$.  By \eqref{e:cond}, we also have $a+b+1 \equiv 0 \imod{p}$, so $p$ divides $b+1$, and thus $p$ divides $a$, which is a contradiction since the highest weight $\l = b\l_k+a\l_{n-1}$ is $p$-restricted. Therefore $J$ is disconnected and $V|_{J^0}$ is reducible. We are now in a position to apply the main theorems in \cite{BGMT,BGT}. We deduce that the only possibility is the configuration found by Ford, with $J=D_{n-1}.2$ (see the case labelled ${\rm U}_2$ in \cite[Table II]{Ford1}; also see Remark \ref{r:geom}(d)). Since the highest weight of $V|_{M}$ is $b\eta_k+a\eta_{n-1}$, we must have $a=1$ and $2a \equiv -2(n-1-k)-1 \imod{p}$. But it is easy to see that this congruence condition is incompatible with the congruence condition in \eqref{e:cond}.
\end{proof}

\section{Geometric subgroups: The disconnected case}\label{s:geom2}

The main result of this section is the following, which completes the proof of Theorem \ref{t:main} when $H$ is contained in a maximal geometric subgroup of $G$.

\begin{thm}\label{t:geom1}
Let $G$, $H$ and $V=V_G(\l)$ be given as in the statement of Theorem \ref{t:main}, and assume that $H<M<G$ where $M$ is a disconnected geometric maximal subgroup of $G$. Then $V|_{H}$ is irreducible if and only if $(G,H,V)$ is one of the cases recorded in Table \ref{tab:geom}. 
\end{thm}

\renewcommand{\arraystretch}{1.1}
\begin{table}
\begin{center}
$$\begin{array}{llll} \hline
G & H & \l & \mbox{Conditions} \\ \hline
A_n & T.X & \l_k,\, 1<k<n & \mbox{$X < {\rm Sym}_{n+1}$ is $\ell$-transitive, $\ell = \min\{k,n+1-k\}$} \\
B_4 & B_1^3.X & \l_3 & \mbox{$p=2$, $X=Z_3$ or ${\rm Sym}_3$} \\
C_4 & C_1^3.X & \l_3 & \mbox{$p=2$, $X=Z_3$ or ${\rm Sym}_3$} \\
C_4 & C_1^3.Z_3 & \l_2,\, \l_3 & p \neq 2 \;\; \mbox{($p \neq 2,3$ if $\l=\l_3$)} \\ 
D_4 & C_1^3.Z_3 & \l_1+\l_4,\, \l_3+\l_4 & p=2 \\ 
D_8 & C_1^4.X & \l_7 & \mbox{$p \neq 3$, $X<{\rm Sym}_4$ is transitive} \\ \hline
\end{array}$$
\caption{The irreducible triples $(G,H,V)$ in Theorem \ref{t:geom1}}
\label{tab:geom}
\end{center}
\end{table}
\renewcommand{\arraystretch}{1}

\begin{remk}\label{r:case7}
Suppose that $H<M<G$, where $(G,M,V)$ is the case labelled (vii) in Table \ref{tab:geom2}, so $G=C_n$, $M=D_n.2$ and $p=2$. In Proposition \ref{p:geom6} we deduce that $V|_H$ is irreducible if and only if $(n,p)=(4,2)$, $\l=\l_3$ and $H=C_1^3.X$ with $X=Z_3$ or ${\rm Sym}_3$, and so this establishes Theorem \ref{t:geom1} in this situation. The result for case (vii) will be obtained from the result for (v) via an isogeny. That is, $V|_{H}$ is irreducible if and only if $(n,p)=(4,2)$, $\l=\l_3$ and $H=B_1^3.Z_3$ or $B_1^3.{\rm Sym}_3$. Therefore, for the remainder of this section we will exclude case (v) in Table \ref{tab:geom2} from our analysis.
\end{remk}

We begin with a couple of preliminary lemmas. Our first result will play an important role in the analysis of cases (vii), (x), (xi) and (xii) in Table \ref{tab:geom2}.

\begin{lem}\label{l:tens}
Let $G$ be a simple classical algebraic group with natural module $W$, and let $H$ be a closed positive-dimensional subgroup of $G$ such that $W|_{H^0}$ is reducible. Then there exists a geometric maximal subgroup $M$ of $G$ such that
\begin{itemize}
\item[{\rm (i)}] $H \leqs M$, and
\item[{\rm (ii)}] $M$ does not normalize any decomposition of the form
$W = W_1 \otimes \cdots \otimes W_t$, where $t \geqs 3$ and the $W_i$ are equidimensional.
\end{itemize}
\end{lem}

\begin{proof}
This follows from the proof of \cite[Theorem 1$'$]{LS}. In particular, we refer the reader to the proofs of Lemmas 3.1, 3.2 and 3.3 in \cite{LS}.
\end{proof}

We will also need the following lemma when dealing with cases (ii) and (viii) in Table \ref{tab:geom2}. In the statement, $G=A_n$ or $D_n$, and $\gamma$ is an involutory graph automorphism of $G$, which induces a natural action on the weight lattice of $G$.

\begin{lem}\label{l:mu}
Let $G=A_n$ or $D_n$, and set $\mu_2 = \gamma(\mu_1)$, where $\mu_1=\sum_{i}a_i\l_i$ is a $T$-weight of $G$ and $\gamma$ is an involutory graph automorphism of $G$. If we write $\mu_2-\mu_1 = \sum_{i}c_i\a_i$, then each $c_i$ is non-negative if and only if $\mu_1 = \mu_2$.
\end{lem}

\begin{proof}
First assume $G=A_n$. We may assume that $\gamma$ interchanges the fundamental dominant weights $\l_i$ and $\l_{n+1-i}$ ($1 \leqs i \leqs n/2$), so 
$\mu_2= \sum_{i} a_{n+1-i} \lambda_i$. Set $\ell=\lfloor  n/2 \rfloor$. Since
$$\lambda_i=\frac{1}{n+1}\left( \sum_{j=1}^{i-1}j(n+1-i)\alpha_j+\sum_{j=i}^n i(n+1-j)\alpha_j\right)$$
(see \cite[Table 1]{Humphreyslie}, for example) it follows that
\begin{align*}
\mu_2-\mu_1  =  & \; \frac{1}{n+1}\sum_{i=1}^n (a_{n+1-i}-a_i)\left( \sum_{j=1}^{i-1}j(n+1-i)\alpha_j+\sum_{j=i}^n i(n+1-j)\alpha_j\right) \\
= & \; \frac{1}{n+1}\left[\sum_{j=1}^{n}(n+1-j)\alpha_j \sum_{i=1}^j i(a_{n+1-i}-a_i) + \sum_{j=1}^{n-1}j\alpha_j \sum_{i=j+1}^n (n+1-i)(a_{n+1-i}-a_i)\right]
\end{align*}
so for $1\leqs j \leqs \ell$ we have 
$$c_j= \frac{1}{n+1}(n+1-2j)\sum_{i=1}^{j-1} i(a_{n+1-i}-a_i)+\frac{1}{n+1}j\sum_{i=j}^\ell(n+1-2i)(a_{n+1-i}-a_i)$$
and if $\ell+1 \leqs j \leqs n$ then 
$$  c_j= \frac{1}{n+1}(n+1-2j)\sum_{i=1}^{n-j} i(a_{n+1-i}-a_i)+\frac{1}{n+1}(n+1-j)\sum_{i=n+1-j}^\ell(n+1-2i)(a_{i}-a_{n+1-i}).$$ 
We deduce that $c_j+c_{n+1-j}=0$ for all $1 \leqs j \leqs \ell$. In addition, if $n$ is odd then $c_{\ell+1}=0$. The result follows.

The case $G=D_n$ is very similar. Here we may assume that $\gamma$ interchanges $\l_{n-1}$ and $\l_n$, so $\mu_2-\mu_1=(a_n-a_{n-1})(\lambda_{n-1}-\lambda_n)$ and the desired result quickly follows.
\end{proof}

For the remainder of this section, we will assume that $(G,H,V)$ is given as in the statement of Theorem \ref{t:main}, so the conditions in Hypothesis \ref{h:our} are satisfied. We will deal with each of the cases in Table \ref{tab:geom2} in turn, excluding case (v) as explained in Remark \ref{r:case7}. 

\subsection{Proof of Theorem \ref{t:geom1}, Part I}\label{ss:part1}

In this section, we will establish Theorem \ref{t:geom1} in the case where $H<M<G$ and 
$(G,M,V)$ is one of the cases labelled (i), (ii), (iii), (iv), (vi), (viii), (ix) or (xiii) in Table 
\ref{tab:geom2}. The remaining cases, labelled (vii), (x), (xi) and (xii), will be handled in Section \ref{ss:part2}.

\begin{prop}\label{p:geom1}
Suppose $H<M<G$ and $(G,M,V)$ is the case labelled (i) in Table \ref{tab:geom2}, so $G=A_n$, $M=N_G(T)=T.{\rm Sym}_{n+1}$ is the normalizer of a maximal torus $T$ of $G$, and $V=V_G(\l_k)$ with $1<k<n$. Then $V|_{H}$ is irreducible if and only if $H=T.X$ and $X < {\rm Sym}_{n+1}$ is $\ell$-transitive, where $\ell = \min\{k,n+1-k\}$.
\end{prop}

\begin{proof}
Here $V=V_G(\l_k)=\L^k(W)$ is the $k$-th wedge of the natural $KG$-module $W$, and by duality we may assume that $1<k \leqs (n+1)/2$. Let $\mathcal{W}(G)=N_G(T)/T = {\rm Sym}_{n+1}$ be the Weyl group of $G$.

Set $S=H^0$ and note that $S \leqs T$ is a subtorus. Let $\L_S(W)$ and $\L_S(V)$ be the set of $S$-weights of $W$ and $V$, respectively, so 
$$W|_{S} = \bigoplus_{\mu \in \L_S(W)} W_{\mu},\;\; V|_{S} = \bigoplus_{\mu \in \L_S(V)} V_{\mu}$$
where $W_{\mu}$ is the $\mu$-weight space of $W$, and similarly $V_{\mu}$ is the $\mu$-weight space of $V$.
There is a natural action of $N_G(S)$ on $\L_S(W)$ and $\L_S(V)$ given by  
$(x \cdot \mu)(s) = \mu( xsx^{-1})$.
In particular, $N_G(S)$ permutes the $S$-weight spaces on $W$ and $V$.

First assume $S=T$. Here the $S$-weight spaces on $W$ and $V$ are $1$-dimensional, and $V|_{H}$ is irreducible if and only if $H/T \leqs \mathcal{W}(G)$ acts transitively on $\L_S(V)$. This is equivalent to the condition that $H/T$ is a $k$-transitive subgroup of $\mathcal{W}(G)={\rm Sym}_{n+1}$. Indeed, we note that the $S$-weight vectors on $V$ are of the form  $w_{i_1} \wedge \cdots \wedge w_{i_k}$, where the $i_j$ are distinct and $\{w_1, \ldots, w_{n+1}\}$ is a basis of $W$ consisting of $S$-weight vectors. This gives the desired result when $S=T$, so for the remainder let us assume that $S$ is a proper subtorus of $T$.

Seeking a contradiction, suppose $V|_{H}$ is irreducible. Now $H \leqs N_G(S)$ (since $S=H^0$) and thus $V|_{N_G(S)}$ is irreducible. In particular, $N_G(S)$ acts irreducibly on $W$ (otherwise $N_G(S)$ lies in a parabolic subgroup of $G$, which would imply $V|_{H}$ is reducible), so $N_G(S)$ must transitively permute the set of $S$-weight spaces on $W$. Therefore, these $S$-weight spaces are equidimensional, whence $N_G(S) \leqs J<G$, where $J$ is a $\C_2$-subgroup of $G$. More precisely, $J$ is the normalizer in $G$ of the direct sum decomposition
$\bigoplus_{\mu \in \L_S(W)}W_{\mu}$. If we now consider the irreducible triple $(G,J,V)$ then the main theorem of \cite{BGT} implies that the $S$-weight spaces on $W$ are $1$-dimensional, so $S$ is a \emph{regular torus}. In particular, 
$$S \leqs C_G(S)=T \leqs N_G(S) \leqs N_G(T)=M$$
and we define
$$\mathcal{W}(S) := N_G(S)/C_G(S) = N_G(S)/T \leqs N_G(T)/T = \mathcal{W}(G).$$

As noted above, $\mathcal{W}(S)$ permutes the $S$-weight spaces $V_{\mu}$, and the irreducibility of $V|_{N_G(S)}$ implies that this action is transitive. In particular, the $S$-weight spaces on $V$ are equidimensional. In fact, we claim that they are $1$-dimensional. To see this, let $d$ denote the dimension of $S$ and fix a basis $\{w_1, \ldots, w_{n+1}\}$ of $W$ comprising $S$-weight vectors. Then there exist integers $c_{i,j}$, $1 \leqs i \leqs d$, $1 \leqs j \leqs n+1$ such that
$$(s_1, \ldots, s_d) \cdot w_j = \left(s_1^{c_{1,j}}s_2^{c_{2,j}} \cdots s_d^{c_{d,j}}\right)w_j$$
for all $(s_1, \ldots, s_d) \in (K^*)^d \cong S$ and all $1 \leqs j \leqs n+1$. Without loss of generality, we may assume that the $w_j$ are ordered so that the $d$-tuples
\begin{equation}\label{e:tuples}
(c_{1,1}, \ldots, c_{1,d}), \ldots, (c_{n+1,1}, \ldots, c_{n+1,d})
\end{equation}
are in lexicographic order. (Note that these $d$-tuples are distinct since the $S$-weight spaces on $W$ are $1$-dimensional.) Then $w_1 \wedge \cdots \wedge w_k \in V$ is an $S$-weight vector of weight
$$s_1^{\sum_{i=1}^{k}c_{1,i}}s_2^{\sum_{i=1}^{k}c_{2,i}} \cdots s_d^{\sum_{i=1}^{k}c_{d,i}}.$$
In view of the lexicographic ordering of the tuples in \eqref{e:tuples}, it follows that this $S$-weight has multiplicity $1$, and this justifies the claim.
 
As previously observed, the irreducibility of $V|_{N_G(S)}$ now implies that $\mathcal{W}(S) \leqs {\rm Sym}_{n+1}$ is $k$-transitive, so to complete the proof of the proposition, it suffices to show that $\mathcal{W}(S)$ is not $2$-transitive. 

To see this, first let $c$ be the codimension of $S$ in $T$ and let $X(T) \cong \mathbb{Z}^n$ and $X(S) \cong \mathbb{Z}^{n-c}$ be the corresponding character groups. The sublattice $S^{\perp}$ is defined by
$$S^{\perp}=\{\gamma \in X(T) \mid \gamma|_{S}=1\} \cong \mathbb{Z}^c$$
and we set 
$$X(T)_{\mathbb{R}}  = X(T) \otimes_{\mathbb{Z}} \mathbb{R},\;\; S^{\perp}_{\mathbb{R}} = S^{\perp} \otimes_{\mathbb{Z}} \mathbb{R}.$$ 
Now $\mathcal{W}(G)$ acts faithfully on $X(T)_{\mathbb{R}}$, and $\mathcal{W}(S)=N_G(S)/T$ stabilizes the $c$-dimensional subspace $S^{\perp}_{\mathbb{R}}$. Let $P$ be 
the pointwise stabilizer of $S^{\perp}_{\mathbb{R}}$ in $\mathcal{W}(G)$. By \cite[Corollary A.29]{MT}, $P$ is a parabolic subgroup of $\mathcal{W}(G)={\rm Sym}_{n+1}$, so it is a direct product of smaller degree symmetric groups. In particular, $P$ is intransitive. Finally, we observe that $\mathcal{W}(S)$ normalizes $P$ (since it stabilizes $S^{\perp}_{\mathbb{R}}$), so the intransitivity of $P$ implies that $\mathcal{W}(S)$ is either intransitive, or transitive and imprimitive. In particular, $\mathcal{W}(S)$ is not $2$-transitive.
\end{proof}

\begin{prop}\label{p:geom4}
Suppose $H<M<G$ and $(G,M,V)$ is the case labelled (iv) in Table \ref{tab:geom2}. Then $V|_{H}$ is reducible.
\end{prop}

\begin{proof}
Here $G=B_n$ and $M=D_n.2$, where $n \geqs 3$ and $p \neq 2$. We have
$$H < M=D_n.2 < G = B_n$$ 
and $V=V_G(\l)$, where the highest weight $\l = \sum_{i}a_i\l_i$ satisfies the conditions recorded in Remark \ref{r:geom}(e). This is the case labelled ${\rm U}_2$ in \cite[Table II]{Ford1}. In particular, we note that $a_n=1$ and $V|_{M^0}$ has exactly two composition factors, say $V|_{M^0} = V_1 \oplus V_2$, where $V_i$ has highest weight $\mu_i$, and 
$$\mu_1 = \sum_{i=1}^{n-2}a_i\eta_i + a_{n-1}\eta_{n-1}+(a_{n-1}+1)\eta_n,\;\; \mu_2 = \sum_{i=1}^{n-2}a_i\eta_i + (a_{n-1}+1)\eta_{n-1}+a_{n-1}\eta_n.$$
(with respect to fundamental dominant weights $\{\eta_1, \ldots, \eta_n\}$ for $M^0=D_n$). As noted in Remark \ref{r:geom}(e), we may assume that $a_i \neq 0$ for some $i<n$. 

Seeking a contradiction, let us assume that $V|_{H}$ is irreducible, so $H \not\leqs M^0$ since $V|_{M^0}$ is reducible. Set $H_1 = H \cap M^0$ and let $J$ be a maximal subgroup of $M^0$ that contains $H_1$. Then $H=H_1.2$ and the irreducibility of $V|_{H}$ implies that $V_1|_{H_1}$ and $V_2|_{H_1}$ are irreducible, so $V_1|_{J}$ and $V_2|_{J}$ are also irreducible.

We can now consider the irreducible triple $(M^0,J,V_1)$, which must be one of the cases recorded in \cite{BGMT,BGT,Seitz2}. Given the conditions on $\l$ (in particular, the fact that $a_n=1$ and $a_i \neq 0$ for some $i<n$), it is easy to see that there are no compatible examples in \cite{BGMT,BGT}. The only possible example in \cite[Table 1]{Seitz2} is the case labelled IV$_1'$, with $J=B_{n-1}$, $a=1$ and $b \neq 0$. However, we claim that the conditions in this configuration are incompatible with those that are given in Remark \ref{r:geom}(e). Indeed, we have $a_{n-1}=0$ and there is a unique $k<n-1$ with $a_k \neq 0$. In case IV$_1'$ we have 
$a_k+n-k \equiv 0 \imod{p}$ and the conditions in Remark \ref{r:geom}(e) yield 
$2a_k \equiv -2(n-k)-1 \imod{p}$.
If both conditions hold, then $p$ divides $a_k+n-k$ and $1+2a_k+2(n-k)$, so $p$ divides $1+a_k+n-k$. Clearly, this is impossible. 
\end{proof}

\begin{prop}\label{p:geom5}
Suppose $H<M<G$ and $(G,M,V)$ is the case labelled (vi) in Table \ref{tab:geom2}. Then $V|_{H}$ is reducible.
\end{prop}

\begin{proof}
Here $G= C_n$, $M=C_m^2.2$ is a $\C_2$-subgroup and $\l=\l_{n-1}+a\l_n$, where $n=2m$,  $0\leqs a<p$ and $2a+3 \equiv 0 \imod{p}$. In particular, note that $p \neq 2$ and $a<p-1$. Seeking a contradiction, let us assume that $V|_{H}$ is irreducible. Set $M^0=C_m^2=M_1M_2$ (a direct product of two simply connected groups of type $C_m$) and let $\{\o_{i,1}, \ldots, \o_{i,m}\}$ be fundamental dominant weights for $M_i$. 

As recorded in \cite[Table 4.2]{BGT}, we have
$V|_{M^0} = V_1 \oplus V_2$
where
\begin{align*}
V_1 & = V_{M_1}((a+1)\o_{1,m}) \otimes V_{M_2}(\o_{2,m-1}+a\o_{2,m}) \\
V_2 & = V_{M_1}(\o_{1,m-1}+a\o_{1,m}) \otimes V_{M_2}((a+1)\o_{2,m}).
\end{align*}
Set $H_1 = H \cap M^0$ and note that $H_1^0 = H^0$ and $H=H_1.2$. In particular, since $V|_{H}$ is irreducible it follows that $V|_{H_1}$ has exactly two composition factors, namely 
$$V|_{H_1} = V_1|_{H_1} \oplus V_2|_{H_1}.$$
Note that $H^0<M^0$ since we are assuming that $H$ is disconnected and non-maximal.
Let 
$\pi_i : H_1 \to M_i$
be the $i$-th projection map and note that $\ker(d\pi_1) \cap \ker(d\pi_2) = 0$ since $H_1$ is a closed positive-dimensional subgroup of $M^0$.

\vs

\noindent \emph{Claim.} $\pi_1(H_1)$ and $\pi_2(H_1)$ are infinite. 

\vs

Seeking a contradiction, suppose that $\pi_1(H_1)$ is finite, in which case $\pi_2(H_1)$ is infinite since $H$ is positive-dimensional. The finiteness of $\pi_1(H_1)$ implies that $d\pi_1 = 0$, so $\ker(d\pi_2)=0$ since $\ker(d\pi_1) \cap \ker(d\pi_2) = 0$. Also note that $\ker(\pi_1)$ is a closed subgroup of finite index in $H_1$, so $H_1^0 \leqs \ker(\pi_1)$ and thus $\pi_2|_{H^0}: H^0 \to M_2$ is injective.

Next we claim that $\pi_2$ is surjective. Suppose otherwise. Then there exists a positive-dimensional maximal subgroup $J_2$ of $M_2$ such that $\pi_2(H_1) \leqs J_2< M_2$.  The irreducibility of $V_1|_{H_1}$ and $V_2|_{H_1}$ implies that $\pi_2(H_1)$ acts irreducibly on the $KM_2$-modules with highest weights $\omega_{2,m-1}+a\omega_{2,m}$ and $(a+1)\omega_{2,m}$, so we can consider the irreducible triples 
$$(M_2,J_2, V_{M_2}(\omega_{2,m-1}+a\omega_{2,m})),\;\; (M_2,J_2, V_{M_2}((a+1)\omega_{2,m})).$$  
By inspecting \cite[Table 1]{Seitz2} we see that there are no compatible examples with $J_2$ connected. Similarly, by applying the main theorems in \cite{BGMT,BGT}, there are no examples with $J_2$ disconnected. This is a contradiction, hence $\pi_2$ is surjective. 

It follows that $\pi_2(H_1^0)=M_2$, so $\pi_2|_{H^0}: H^0 \to M_2$ is a bijective morphism. Moreover, $\ker(d(\pi_2|_{H^0}))=0$ since $\ker(d\pi_2)=0$, so $d(\pi_2|_{H^0})$ is an isomorphism of Lie algebras and thus $\pi_2|_{H^0}$ is an isomorphism of algebraic groups. In particular, $H^0$ is a simply connected group of type $C_m$. By Lemma \ref{l:steinberg} we may write $\pi_2|_{H^0}=t_{x}$ for some $x\in H^0$, where $t_x$ is an inner automorphism (conjugation by $x$). In addition, note that $H \leqs N_G(H^0)=H^0C_G(H^0)$ and thus $V|_{H^0}$ is homogeneous. 

Let $\{\eta_1,\ldots, \eta_m\}$ be a set of fundamental dominant  weights for $H^0$. Then $V|_{H^0}$ has composition factors isomorphic to $V_{H^0}(\eta_{m-1}+a\eta_m)$ and  
$V_{H^0}((a+1)\eta_m)$, which contradicts the fact that 
$V|_{H^0}$ is homogeneous. We conclude that $\pi_1(H_1)$ is infinite, and similarly 
$\pi_2(H_1)$ is also infinite.

\vs

\noindent \emph{Claim.} $\pi_1$ and $\pi_2$ are surjective.  

\vs

Seeking a contradiction, suppose $\pi_1$ is not surjective. Since $\pi_1(H_1)$ is infinite, there exists a positive-dimensional maximal subgroup $J_1$ of $M_1$ such that $\pi_1(H_1) \leqs J_1<M_1$ and we can consider the irreducible triples 
$$(M_1,J_1, V_{M_1}((a+1)\o_{1,m})), \;\; (M_1,J_1, V_{M_1}(\o_{1,m-1}+a\o_{1,m})).$$  
As before, we find that there are no compatible examples, which is a contradiction and thus $\pi_1$ is surjective. An entirely similar argument shows that $\pi_2$ is also surjective. 

\vs

By the previous claim, it follows that $\pi_i(H^0)=M_i$ for $i=1,2$, so $H^0$ is a subdirect product of the direct product $M^0=M_1M_2$. By applying Lemma \ref{l:basecase}, noting that $H^0 <M^0$, we deduce that $H^0 \cong M_1$ is diagonally embedded in $M_1M_2$, so we may write
$$H^0=\{(\tau_1(x), \tau_2(x)) \mid x \in {\rm Sp}_{2m}(K)\},$$ 
where $\tau_i: {\rm Sp}_{2m}(K)\to M_i$ is a bijective morphism. By appealing to Lemma \ref{l:steinberg}, we may write $\tau_i=t_{x_i}\sigma_{q_i}$ for some $x_i \in H^0$ and $p$-power $q_i$ (where $\sigma_{q_i}$ is a standard field automorphism), and once again we note that $V|_{H^0}$ is homogeneous. Note that at least one $q_i$ is equal to $1$ (since $H^0$ is a closed subgroup $M^0$); without loss of generality we will assume $q_2=1$.

Let $\{\eta_1,\dots, \eta_m\}$ be a set of fundamental dominant weights for $H^0$. 
Then  
$$V_1|_{H^0} = V_{H^0}((a+1)\eta_m)^{(q_1)} \otimes V_{H^0}(\eta_{m-1}+a\eta_{m})$$
and
$$V_2|_{H^0} = V_{H^0}(\eta_{m-1}+a\eta_{m})^{(q_1)} \otimes V_{H^0}((a+1)\eta_m),$$
so $V|_{H^0}$ has composition factors with highest weights
$$\eta_{m-1}+((a+1)q_1+a)\eta_m,\;\; q_1\eta_{m-1}+(aq_1+a+1)\eta_m.$$
Since $V|_{H^0}$ is homogeneous, these highest weights must be equal and thus $q_1=1$. Now $p<a-1$ so the modules $V_{H^0}((a+1)\eta_m)$ and $V_{H^0}(\eta_{m-1}+a\eta_{m})$ are $p$-restricted and thus Lemma \ref{l:ten} implies that $V|_{H^0}$ is not homogeneous. This is a contradiction.
\end{proof}

\begin{prop}\label{p:geom7}
Suppose $H<M<G$ and $(G,M,V)$ is the case labelled (viii) in Table \ref{tab:geom2}. Then $V|_{H}$ is reducible.
\end{prop}

\begin{proof}
Here $G=D_n$ and $M=(D_m^2.2).2$ is a $\C_2$-subgroup, where $n=2m$, $m\geqs 3$ is odd and $p=2$. Moreover, $V=V_G(\l)$ where $\l=\l_1+\l_{n-1}$ or $\l_1+\l_n$ (see Table \ref{tab:geom2}); without loss of generality, we will fix $\l = \l_1+\l_{n-1}$.  Seeking a contradiction, let us assume that $V|_{H}$ is irreducible. 

Write $M^0=D_m^2=M_1M_2$ and let $\{\o_{i,1}, \ldots, \o_{i,m}\}$ be fundamental dominant weights for $M_i$. Then  \cite[Table 4.2]{BGT} indicates that
$V|_{M^0} = V_1 \oplus V_2 \oplus V_3 \oplus V_4$, 
where 
\begin{align*}
V_1 & = V_{M_1}(\o_{1,1}+\o_{1,m}) \otimes V_{M_2}(\o_{2,m-1}) \\
V_2 & = V_{M_1}(\o_{1,1}+\o_{1,m-1}) \otimes V_{M_2}(\o_{2,m}) \\
V_3 & = V_{M_1}(\o_{1,m-1}) \otimes V_{M_2}(\o_{2,1}+\o_{2,m}) \\
V_4 & = V_{M_1}(\o_{1,m}) \otimes V_{M_2}(\o_{2,1}+\o_{2,m-1})
\end{align*}
Set $H_1 = H \cap M^0$ and note that $H_1^0 = H^0$ and $|H:H_1|=4$. Indeed, $H/H_1$ is isomorphic to a subgroup of $M/M^0$ and thus $|H:H_1| \leqs 4$, but $V|_{H_1}$ has at least four composition factors and thus the irreducibility of $V|_H$ implies that $|H:H_1|=4$. Therefore, $V|_{H_1}$ has exactly four composition factors, namely
$$V|_{H_1}=V_1|_{H_1}\oplus V_2|_{H_1}\oplus V_3|_{H_1}\oplus V_4|_{H_1}.$$

In order to proceed as in the proof of the previous proposition, we need to slightly modify our set-up. Indeed, $G$ is the simply connected group of type $D_n$, so $M^0=M_1M_2$ is a central product of spin groups of type $D_m$. Since the $KM^0$-module $W$ lifts to a representation $\rho: L \to {\rm GL}(W)$, where $L=L_1L_2$ is the direct product of two simply connected groups of type $D_m$, we have $M^0 = L/Y$ where $Y = \ker(\rho)$. In particular, there exist subgroups $R \leqs R_1 \leqs L$ such that $H_1=R_1/Y$ and $H^0=R/Y$. Note that $H_1/H^0 \cong R_1/R$ and $R^0=R_1^0$. Let 
$\pi_i: R_1 \rightarrow L_i$
be the $i$-th projection map and note that $\ker(d\pi_1) \cap \ker(d\pi_2) = 0$ since $R_1$ is a closed positive-dimensional subgroup of $L$.

The $KM^0$-module $V_i$ lifts to a representation 
$\rho_i: L\to {\rm GL}(V_i)$, so we can consider $V_i|_{R_1}$.  The irreducibility of $V_i|_{H_1}$ implies that $V_i|_{R_1}$ is also irreducible, whence $\pi_1(R_1)$ is irreducible on each of the $KL_1$-modules 
\begin{equation}\label{e:mods}
V_{L_1}(\o_{1,1}+\o_{1,m}),\;\; V_{L_1}(\o_{1,1}+\o_{1,m-1}),\;\; V_{L_1}(\o_{1,m-1}),\;\; V_{L_1}(\o_{1,m}),
\end{equation}
and similarly $\pi_2(R_1)$ is irreducible on the $KL_2$-modules
\begin{equation}\label{e:modsbis}
V_{L_2}(\o_{2,m-1}),\;\; V_{L_2}(\o_{2,m}),\;\; V_{L_2}(\o_{2,1}+\o_{2,m}),\;\; V_{L_2}(\o_{2,1}+\o_{2,m-1}).
\end{equation}

\vs

\noindent \emph{Claim.} $\pi_1(R_1)$ and $\pi_2(R_1)$ are infinite.  

\vs

We proceed as in the proof of Proposition \ref{p:geom5}. Suppose $\pi_1(R_1)$ is finite. Then $\pi_2(R_1)$ is infinite, $\ker(d\pi_2)=0$ and $R_1^0 \leqs \ker(\pi_1)$, so $\pi_2|_{R^0}: R^0\to L_2$ is injective.

Suppose $\pi_2$ is not surjective. Then there exists a positive-dimensional maximal subgroup $J_2$ of $L_2$ such that $\pi_2(R_1) \leqs J_2 <L_2$, and we can consider the irreducible triples $(L_2,J_2,U)$ for the four $KL_2$-modules $U$ in \eqref{e:modsbis}. By applying the main theorems of  \cite{BGMT,BGT,Seitz2} we find that there are no compatible examples (note that in the case labelled IV$_1'$ in \cite[Table 1]{Seitz2}, we require the parameters to be $a=b=k=1$, hence the given congruence condition implies that $m$ is even, which is false).
This is a contradiction, hence $\pi_2$ is surjective.

It follows that $\pi_2(R_1^0)=L_2$ and $\pi_2|_{R^0}: R^0\to L_2$ is a bijective morphism. Furthermore, $\ker(d(\pi_2|_{R^0}))=0$ so $d(\pi_2|_{R^0})$ is an isomorphism and thus $\pi_2|_{R^0}$ is an isomorphism of algebraic groups. In particular, $R^0$ is simply connected of type $D_m$. By Lemma \ref{l:steinberg}, we may write $\pi_2|_{R^0}=t_{x}\gamma^{k}$ for some $x \in R^0$ and integer $k \in \{0,1\}$, where 
$\gamma$ is an involutory graph automorphism (note that $m \neq 4$, so a triality automorphism does not arise here). Now $H \leqs N_G(H^0)$ induces algebraic group automorphisms of $H^0$ that permute the $KR^0$-composition factors of $V$, so 
$V|_{R^0}$ has the following homogeneity property:

\vspace{-2mm}

\begin{equation}\label{e:hom}
\begin{array}{l}
\mbox{\emph{Either $V|_{R^0}$ is homogeneous, or the homogeneous components of $V|_{R^0}$ are}} \\
\mbox{\emph{conjugate under an involutory graph automorphism of $R^0$.}}
\end{array}
\end{equation}
  
Let $\{\eta_1,\ldots, \eta_m\}$ be fundamental dominant weights for $R^0$. Then $V|_{R^0}$ has composition factors isomorphic to 
$V_{R^0}(\eta_{m-1})$, $V_{R^0}(\eta_{m})$, $V_{R^0}(\eta_1 + \eta_{m})$ and $V_{R^0}(\eta_1 + \eta_{m-1})$, but this is incompatible with \eqref{e:hom}. Therefore $\pi_1(R_1)$ is infinite, and similarly $\pi_2(R_1)$ is also infinite.

\vs

\noindent \emph{Claim.} $\pi_1$ and $\pi_2$ are surjective.  

\vs

Suppose $\pi_1$ is not surjective. Then there exists a positive-dimensional maximal subgroup $J_1$ of $L_1$ such that $\pi_1(R_1) \leqs J_1<L_1$ and we can consider the irreducible triples $(L_1,J_1, U)$  for the four $KL_1$-modules $U$ in \eqref{e:mods}. We have already noted that there are no compatible examples and thus $\pi_1$ is surjective. Similarly, $\pi_2$ is surjective.  

\vs

We have $\pi_i(R^0)=L_i$ for $i=1,2$, so $R^0$ is a subdirect product of $L=L_1L_2$ and thus  Lemma \ref{l:basecase} implies that either $R^0 =L$, or $R^0\cong L_1$ is diagonally embedded in $L_1L_2$. If $R^0=L$ then $H^0=M^0$ and the irreducibility of $V|_{H}$ implies that $H=M$, which is false. Therefore $R^0$ is diagonally embedded, so
$$R^0 = \{(\tau_1(x), \tau_2(x)) \mid x \in {\rm Spin}_{2m}(K)\}$$
and $\tau_i: {\rm Spin}_{2m}(K)\to L_i$ is a bijective morphism. In particular, we may write 
$\tau_i=t_{x_i}\sigma_{q_i}\gamma^{k_i}$ for some $x_i \in R^0$, $p$-power $q_i$ and $k_i\in \{0,1\}$ (see Lemma \ref{l:steinberg}). Again, we observe that \eqref{e:hom} holds. Since $R^0$ is a closed subgroup of $L$, it follows that at least one $q_i$ is equal to $1$. We may assume $q_2=1$.

Let $\{\eta_1,\ldots, \eta_m\}$ be a set of fundamental dominant  weights for $R^0$. By considering the restriction of $V$ to $R^0$, we deduce that $V_i|_{R^0}$ has a composition factor with highest weight $\mu_i$ as follows:
$$\begin{array}{ccccc} \hline
(k_1,k_2) & \mu_1 & \mu_2  & \mu_3 & \mu_4 \\ \hline
(0,0) & q_1\eta_{1}+\eta_{m-1}+q_1\eta_m   &  q_1\eta_{1}+q_1\eta_{m-1}+\eta_m & \eta_{1}+q_1\eta_{m-1}+\eta_m  & \eta_{1}+\eta_{m-1}+q_1\eta_m \\
(1,0) & q_1\eta_{1}+(q_1+1)\eta_{m-1}  &  q_1\eta_{1}+(q_1+1)\eta_m & \eta_{1}+(q_1+1)\eta_m  & \eta_{1}+(q_1+1)\eta_{m-1} \\
(0,1) & q_1\eta_{1}+(q_1+1)\eta_m  &  q_1\eta_{1}+(q_1+1)\eta_{m-1} & \eta_{1}+(q_1+1)\eta_{m-1}  & \eta_{1}+(q_1+1)\eta_m \\
(1,1) & q_1\eta_{1}+q_1\eta_{m-1}+\eta_m  &  q_1\eta_{1}+\eta_{m-1}+q_1\eta_m & \eta_{1}+\eta_{m-1}+q_1\eta_m  & \eta_{1}+q_1\eta_{m-1}+\eta_m \\ \hline
\end{array}$$  

In view of \eqref{e:hom}, we deduce that $q_1=1$ in all four cases. By applying Lemma \ref{l:ten} it follows that $V|_{R^0}$ is not homogeneous. More precisely, $\mu_1$ affords the highest weight of a composition factor of $V_1|_{R^0}$ and if $\nu$ is the highest weight of any other composition factor of $V_1|_{R^0}$, then $\nu \neq \mu_1$ and $\nu \preccurlyeq \mu_1$. However, in view of Lemma \ref{l:mu}, this is incompatible with \eqref{e:hom}.
\end{proof}

\begin{prop}\label{p:geom2}
Suppose $H<M<G$ and $(G,M,V)$ is the case labelled (ii) in Table \ref{tab:geom2}. Then $V|_{H}$ is reducible.
\end{prop}

\begin{proof}
Here $G=A_n$ and $M=A_m^2.2$ is a $\C_4(ii)$-subgroup, where $n=m(m+2)$, $p \neq 2$ and $m \geqs 2$. Moreover, $V=V_G(\l)$ and $\l=\l_2$ or $\l_{n-1}$. By duality, we may assume that $V=V_G(\l_2) = \L^2(W)$. Seeking a contradiction, let us assume that $V|_{H}$ is irreducible. 

Write $M^0 = A_m^2 = M_1M_2$ and note that this is a central product. Let $\{\o_{i,1}, \ldots, \o_{i,m}\}$ be fundamental dominant weights for $M_i$. As recorded in \cite[Table 6.2]{BGT}, we have
$V|_{M^0} = V_1 \oplus V_2$,
where
$$V_1 = V_{M_1}(\o_{1,2}) \otimes V_{M_2}(2\o_{2,1}),\;\; V_2 = V_{M_1}(2\o_{1,1}) \otimes V_{M_2}(\o_{2,2}).$$
Set $H_1 = H \cap M^0$ and note that $H_1^0 = H^0$ and $H=H_1.2$. Since $V|_{H}$ is irreducible it follows that $V|_{H_1}$ has exactly two composition factors, namely 
$$V|_{H_1} = V_1|_{H_1} \oplus V_2|_{H_1}.$$
Note that $H^0<M^0$ since $H$ is disconnected and non-maximal.

As in the proof of Proposition \ref{p:geom7}, we need to modify this initial set-up in order to proceed as we did in the proof of Proposition \ref{p:geom5} (the main difference here is the fact that $M^0$ is a central product, rather than a direct product). Since $W$ is a $KM^0$-module, it lifts to a representation $\rho: L \to {\rm GL}(W)$, where $L=L_1L_2$ is the direct product of two simply connected groups $A_m = {\rm SL}_{m+1}(K)$. Then $M^0 = L/Y$, where $Y = \ker(\rho)$, and so there exist subgroups $R \leqs R_1 \leqs L$ such that $H_1 = R_1/Y$ and $H^0=R/Y$. Note that $H_1/H^0 \cong R_1/R$ is finite, so $R^0=R_1^0$. Let 
$\pi_i : R_1 \to L_i$ 
be the $i$-th projection map and observe that $\ker(d\pi_1) \cap \ker(d\pi_2) = 0$.

Since the $KM^0$-module $V_i$ lifts to a representation 
$\rho_i : L \to {\rm GL}(V_i)$, we can consider the restriction of $V_i$ to $R_1$. The irreducibility of $V_1|_{H_1}$ and $V_2|_{H_1}$ implies that $V_1|_{R_1}$ and $V_2|_{R_1}$ are also irreducible, whence $\pi_1(R_1)$ is irreducible on the $KL_1$-modules $V_{L_1}(\o_{1,2})$ and $V_{L_1}(2\o_{1,1})$, and $\pi_2(R_1)$ acts irreducibly on $V_{L_2}(2\o_{2,1})$ and $V_{L_2}(\o_{2,2})$.

\vs

\noindent \emph{Claim.} $\pi_1(R_1)$ and $\pi_2(R_1)$ are infinite.  

\vs

We proceed as in the proof of Proposition \ref{p:geom5}; the details are very similar. Suppose that $\pi_1(R_1)$ is finite, so $\ker(d\pi_2)=0$. Then $\pi_2(R_1)$ has to be infinite since $H$ (and thus $H_1$, and also $R_1$) is infinite. Since $\ker(\pi_1) \leqs R_1$ has finite index, it follows that $R_1^0 \leqs \ker(\pi_1)$ and thus $\pi_2|_{R^0}: R^0 \to L_2$ is injective. 

Suppose $\pi_2$ is not surjective. Then there exists a positive-dimensional maximal subgroup $J_2$ of $L_2$ such that $\pi_2(R_1) \leqs J_2< L_2$.  As noted above, $\pi_2(R_1)$ acts irreducibly on $V_{L_2}(2\omega_{2,1})$ and 
$V_{L_2}(\omega_{2,2})$, so we may consider the irreducible triples 
$(L_2,J_2, V_{L_2}(2\omega_{2,1}))$ and $(L_2,J_2, V_{L_2}(\omega_{2,2}))$.  
In the usual way, by inspecting \cite{BGMT,BGT,Seitz2}, we deduce that there are no compatible examples, whence $\pi_2$ is surjective. 

Therefore $\pi_2(R_1^0)=L_2$ and thus $\pi_2|_{R^0}: R^0 \to L_2$ is a bijective morphism. Moreover, $\ker(d(\pi_2|_{R^0}))=0$ so $d(\pi_2|_{R^0})$ is an isomorphism and thus $\pi_2|_{R^0}$ is an isomorphism of algebraic groups. By Lemma \ref{l:steinberg}, we can write 
$\pi_2|_{R^0}=t_{x}\gamma^{k}$ for some $x \in R^0$ and integer $k\in \{0,1\}$, where $\gamma$ is a graph automorphism. Note that \eqref{e:hom} holds.
  
Let $\{\eta_1,\ldots, \eta_m\}$ be a set of fundamental dominant weights for $R^0$.  Then $V|_{R^0}$ has composition factors isomorphic to $V_{R^0}(2\eta_1)$ and $V_{R^0}(\eta_2)$ if $k=0$, and $V_{R^0}(2\eta_m)$ and $V_{R^0}(\eta_{m-1})$ if $k=1$. But the corresponding highest weights are incompatible with \eqref{e:hom}, so we have reached a contradiction. We conclude that $\pi_1(R_1)$ is infinite, and similarly $\pi_2(R_1)$ is also infinite.

\vs

\noindent \emph{Claim.} $\pi_1$ and $\pi_2$ are surjective.  

\vs

Suppose $\pi_1$ is not surjective. Then there exists a positive-dimensional maximal subgroup $J_1$ of $L_1$ such that $\pi_1(R_1) \leqs J_1<L_1$ and we can consider the irreducible triples 
$(L_1,J_1, V_{L_1}(2\o_{1,1}))$ and $(L_1,J_1, V_{L_1}(\o_{1,2}))$.  
As above, there are no compatible examples and thus $\pi_1$ is surjective. An entirely similar argument shows that $\pi_2$ is also surjective. 

\vs

Now $\pi_i(R^0)=L_i$ for $i=1,2$, so $R^0$ is a subdirect product of $L=L_1L_2$ and thus  Lemma \ref{l:basecase} implies that $R^0 \cong L_1$ is diagonally embedded in $L_1L_2$ (note that $H^0<M^0$, so $R^0 < L$). Therefore
$$R^0=\{(\tau_1(x), \tau_2(x)) \mid x \in {\rm SL}_{m+1}(K)\},$$ 
where $\tau_i: {\rm SL}_{m+1}(K)\to L_i$ is a bijective morphism. As before, we may write 
$\tau_i=t_{x_i}\sigma_{q_i}\gamma^{k_i}$ for some $x_i \in R^0$, $p$-power $q_i$ and $k_i\in \{0,1\}$.  Note that \eqref{e:hom} holds. As before, we may assume that $q_2=1$.

Let $\{\eta_1,\dots, \eta_m\}$ be a set of fundamental dominant weights for $R^0$. Now $V|_{R^0} = V_1|_{R^0} \oplus V_2|_{R^0}$ and we calculate that $V|_{R^0}$ has composition factors with the following highest weights $\mu_1$ and $\mu_2$:
$$\begin{array}{ccc} \hline
(k_1,k_2) & \mu_1 & \mu_2 \\ \hline
(0,0) & 2\eta_{1}+q_1\eta_2 & 2q_1\eta_{1}+\eta_2 \\
(1,0) & 2\eta_{1}+q_1\eta_{m-1} & \eta_{2}+2q_1\eta_m \\
(0,1) & q_1\eta_{2}+2\eta_m & 2q_1\eta_{1}+\eta_{m-1} \\
(1,1) & q_1\eta_{m-1}+2\eta_m & \eta_{m-1}+2q_1\eta_m \\ \hline
\end{array}$$
In all four cases, \eqref{e:hom} implies that $q_1=1$. 

Now $V|_{R^0}$ is non-homogeneous by Lemma \ref{l:ten}. More precisely, $V_1|_{R^0}$ has a composition factor of highest weight $\mu_1$ as in the table (with $q_1=1$), occurring with multiplicity $1$. If $\nu$ denotes the highest weight of any other composition factor of $V_1|_{R^0}$, then $\nu \neq \mu_1$ and $\nu \preccurlyeq \mu_1$ (so $\mu_1 - \nu = \sum_{i}c_i\a_i$ for some $c_i \in \mathbb{N}_0$). Therefore, Lemma \ref{l:mu} implies that $V|_{R^0}$ does not satisfy the homogeneity condition in \eqref{e:hom} and this final contradiction completes the proof of the proposition.
\end{proof}

\begin{prop}\label{p:geom3}
Suppose $H<M<G$ and $(G,M,V)$ is the case labelled (iii) in Table \ref{tab:geom2}. Then $V|_{H}$ is reducible.
\end{prop}

\begin{proof}
Here $G=A_n$, $V=V_G(\l_k)$ with $1 < k < n$, and $M=D_m.2$ is a $\C_6$-subgroup with $n=2m-1$, $m \geqs 2$ and $p \neq 2$. Let $\{\eta_1, \ldots, \eta_m\}$ be a set of fundamental dominant weights for $M^0=D_m$. There are three separate cases to deal with here, depending on the value of $k$ (by duality, we may assume that $2 \leqs k \leqs m$):
\begin{itemize}\addtolength{\itemsep}{0.2\baselineskip}
\item[(a)] $k=m$: $V|_{M^0}=V_1 \oplus V_2$ is reducible, where $V_1$ and $V_2$ have highest weights $2\eta_{m-1}$ and $2\eta_m$, respectively (see \cite[Table 3.2]{BGT}).
\item[(b)] $k=m-1$: $V|_{M^0}$ is irreducible, with highest weight $\eta_{m-1}+\eta_m$ (see case I$_5$ in \cite[Table 1]{Seitz2}). Note that $m \geqs 3$.
\item[(c)] $2 \leqs k<m-1$: $V|_{M^0}$ is irreducible, with highest weight $\eta_{k}$ (see case I$_4$ in \cite[Table 1]{Seitz2}). Note that $m \geqs 4$.
\end{itemize}
Seeking a contradiction, let us assume that $V|_{H}$ is irreducible.

First assume that (a) holds. Note that $H \not\leqs M^0$ since $V|_{M^0}$ is reducible. To begin with, let us assume  $m=2$. Here $G=A_3$ and $H^0<M^0=A_1A_1$, so 
$$H^0 \in \{T_2, A_1, A_1T_1\}$$ 
(recall that $H^0$ is reductive; see Lemma \ref{l:hred}). Also note that $V=\L^2(W)$, where $W$ is the natural $KG$-module. We claim that $H^0=A_1$. To see this, suppose $S \leqs H^0$ is  a central torus. Then $H \leqs N_G(S)$ and thus the set of fixed points of $S$ on $V$ is $H$-invariant. But $S$ lies in a maximal torus of $M^0$, which has nontrivial fixed points on $V$, so this contradicts the irreducibility of $V|_{H}$. This justifies the claim. Therefore 
$$H \leqs N_G(H^0) =  H^0C_G(H^0)$$ and thus $V|_{H^0}$ is homogeneous. 

By considering the possible embeddings of $H^0$ in $M^0$, it follows that $W|_{H^0}$ is the two-fold tensor product $U\otimes U$,
 where $U$ is the natural $KH^0$-module. Hence, $W|_{H^0} = W_1\oplus W_2$, where 
$W_1 = V_{H^0}(2\omega)$ and $W_2 = V_{H^0}(0\omega)$ is the trivial irreducible $KH^0$-module (here $\omega$ is the fundamental dominant weight for $H^0$). Since $\dim W_1 \ne \dim W_2$ we deduce that $W|_{N_G(H^0)}$ is reducible and thus $N_G(H^0)$ lies in a parabolic subgroup of $G$. This contradicts the irreducibility of $V|_{H}$.

Now assume $m \geqs 3$. Set $H_1 = H \cap M^0$ and note that $H = H_1.2$, so $V|_{H_1}$ has exactly two composition factors, namely $V_1|_{H_1}$ and $V_2|_{H_1}$. Note that $H_1<M^0$ since we are assuming that $H$ is disconnected and non-maximal. Let $J$ be a maximal subgroup of $M^0$ that contains $H_1$, so we have
$$H_1 = H \cap M^0 \leqs J < M^0 = D_m.$$
We consider the irreducible triples $(M^0,J,V_1)$ and $(M^0,J,V_2)$, where $V_1 = V_{M^0}(2\eta_{m-1})$ and $V_2 = V_{M^0}(2\eta_{m})$. By inspecting \cite[Table 1]{Seitz2}, and using the main theorems of \cite{BGMT,BGT}, we deduce that $J=B_{m-1}$ is the only possibility and
$$V_1|_{J} = V_2|_{J} = V_{B_{m-1}}(2\xi_{m-1})$$
(where $\{\xi_1, \ldots, \xi_{m-1}\}$ are fundamental dominant weights for $B_{m-1}$) -- see case IV$_1$ in \cite[Table 1]{Seitz2}. Note that if $H_1 = J = B_{m-1}$ then the two $KH^0$-composition factors of $V|_{H^0}$ (namely $V_1|_{H^0}$ and $V_2|_{H^0}$) are isomorphic, but this is ruled out by Proposition \ref{p:niso}. Therefore $H_1$ is a proper subgroup of $J$, so let $L$ be a maximal subgroup of $J$ that contains $H_1$, in which case
$$H_1 = H \cap M^0 \leqs L < J = B_{m-1} < M^0=D_m.$$
We now consider the irreducible triple $(J,L,V_{B_{m-1}}(2\xi_{m-1}))$. In the usual way, by inspecting \cite{BGMT, BGT,Seitz2}, we deduce that there are no compatible configurations and this completes the analysis of case (a). 

Next consider case (b). First assume $H \leqs M^0$. Since we are assuming $H$ is disconnected and non-maximal, it follows that $H \leqs J<M^0$ for some maximal subgroup $J$ of $M^0$, and we may consider the irreducible triple $(M^0,J,V_{M^0}(\eta_{m-1}+\eta_m))$. By inspecting \cite{BGMT, BGT, Seitz2}, it is easy to check that there are no compatible examples. In the same way, we deduce that $H \not\leqs M^0$ in case (c). 

Finally, let us consider cases (b) and (c), with $H \not\leqs M^0$. Let $J$ be a maximal subgroup of $M = D_m.2$ such that
$$H \leqs J < M = D_m.2 = {\rm GO}(W).$$
Note that $J$ is disconnected, and $J$ is either geometric or non-geometric (as described in Section \ref{ss:dn2}). Given the highest weight of $V|_{M^0}$, we can rule out the latter possibility by applying \cite[Theorem 3]{BGMT}, so we may assume $J$ is geometric. (Note that we can appeal to \cite[Theorem 3]{BGMT} since $V|_{M^0}$ is irreducible.) The possibilities for $J$ are determined in Proposition \ref{p:dn2} and they are listed in Table \ref{tab:dn2}. We now apply Proposition \ref{p:dn2_1}, which implies that $V|_{J}$ is reducible. This final contradiction completes the proof of the proposition.
\end{proof}

\begin{prop}\label{p:geom13}
Suppose $H<M<G$ and $(G,M,V)$ is one of the cases labelled (ix) or (xiii) in Table \ref{tab:geom2}. Then $V|_{H}$ is reducible.
\end{prop}

\begin{proof}
First consider the case labelled (xiii). Here $G=D_8$, $M=C_2^2.2$ is a $\C_4(ii)$-subgroup, $p \neq 5$ and $V=V_G(\l_7)$. Seeking a contradiction, let us assume that $V|_{H}$ is irreducible. We proceed as in the proof of Proposition \ref{p:geom7}. 

Write $M^0=M_1M_2$, which is a central product of two simply connected groups of type $C_2$, and let $\{\o_{1,1},\o_{1,2}\}$ and $\{\o_{2,1},\o_{2,2}\}$ be fundamental dominant weights for $M_1$ and $M_2$, respectively. As recorded in \cite[Table 6.2]{BGT}, we have $V|_{M^0} = V_1 \oplus V_2$, where
$$V_1 = V_{M_1}(\o_{1,1}) \otimes V_{M_2}(\o_{2,1}+\o_{2,2}),\;\; V_2 = V_{M_1}(\o_{1,1}+\o_{1,2}) \otimes V_{M_2}(\o_{2,1}).$$
Set $H_1 = H \cap M^0$, so $H=H_1.2$ and $H_1^0 = H^0$. Since  $V|_{H}$ is irreducible it follows that $V|_{H_1}$ has exactly two composition factors, namely 
$V_1|_{H_1}$ and $V_2|_{H_1}$.

The $KM^0$-module $W$ lifts to a representation $\rho:L \to {\rm GL}(W)$, where $L = L_1L_2$ is a direct product of two simply connected groups of type $C_2$, so $M^0 = L/Y$ where $Y = \ker(\rho)$. Since $H^0 \leqs H_1 \leqs M^0$, there exist subgroups $R \leqs R_1 \leqs L$ such that $H_1 = R_1/Y$ and $H^0 = R/Y$. Note that $R^0 = R_1^0$. Let 
$\pi_i : R_1 \to L_i$ 
be the $i$-th projection map and observe that $\ker(d\pi_1) \cap \ker(d\pi_2) = 0$.
Since the $KM^0$-module $V_i$ lifts to a representation $\rho_i:L \to {\rm GL}(V_i)$, we can consider $V_{i}|_{R_1}$. The irreducibility of $V_i|_{H_1}$ implies that $V_i|_{R_1}$ is also irreducible, so we deduce that $\pi_1(R_1)$ acts irreducibly on the $KL_1$-modules $V_{L_1}(\o_{1,1})$ and $V_{L_1}(\o_{1,1}+\o_{1,2})$, and similarly, $\pi_2(R_1)$ is irreducible on $V_{L_2}(\o_{2,1})$ and $V_{L_2}(\o_{2,1}+\o_{2,2})$.

\vs

\noindent \emph{Claim.} $\pi_1(R_1)$ and $\pi_2(R_1)$ are infinite.

\vs

We proceed as in the previous cases. Suppose $\pi_1(R_1)$ is finite. Then $\pi_2(R_1)$ is infinite, $\ker(d\pi_2)=0$ and $R_1^0 \leqs \ker(\pi_1)$, so $\pi_2|_{R^0}: R^0 \to L_2$ is injective.

Suppose $\pi_2$ is not surjective. Then there exists a positive-dimensional maximal subgroup $J_2$ of $L_2$ such that $\pi_2(R_1) \leqs J_2< L_2$, and we can consider the irreducible triples 
$(L_2,J_2, V_{L_2}(\omega_{2,1}))$ and $(L_2,J_2, V_{L_2}(\o_{2,1}+\omega_{2,2}))$. 
By inspecting \cite{BGMT,BGT,Seitz2} we find that there are no compatible examples, which is a contradiction and thus $\pi_2$ is surjective. Therefore $\pi_2(R_1^0)=L_2$ and thus 
$\pi_2|_{R^0}: R^0 \to L_2$ is a bijective morphism. Since $\ker(d(\pi_2|_{R^0}))=0$ we deduce that $\pi_2|_{R^0}$ is an isomorphism of algebraic groups, so by Lemma \ref{l:steinberg} we can write $\pi_2|_{R^0}=t_{x}$ for some $x \in R^0$. If $\{\eta_1,\eta_2\}$ is a set of fundamental dominant  weights for $R^0$, then  
$V|_{R^0}$ has composition factors isomorphic to $V_{R^0}(\eta_1)$ and $V_{R^0}(\eta_1+\eta_2)$, but $V|_{R^0}$ is homogeneous since $N_G(H^0) = H^0C_G(H^0)$, so this is a  contradiction. We conclude that $\pi_1(R_1)$ is infinite, and similarly $\pi_2(R_1)$ is infinite. 

\vs

\noindent \emph{Claim.} $\pi_1$ and $\pi_2$ are surjective.

\vs
 
Suppose $\pi_1$ is not surjective. Then there exists a positive-dimensional maximal subgroup $J_1$ of $L_1$ such that $\pi_1(R_1) \leqs J_1<L_1$, and we can consider the irreducible triples 
$(L_1,J_1, V_{L_1}(\o_{1,1}))$ and $(L_1,J_1, V_{L_1}(\o_{1,1}+\o_{1,2}))$. 
As noted above, there are no compatible examples, so $\pi_1$ must be surjective and an entirely similar argument shows that $\pi_2$ is also surjective. 

\vs

We have $\pi_i(R^0)=L_i$ for $i=1,2$, so $R^0$ is a subdirect product of $L=L_1L_2$ and thus  Lemma \ref{l:basecase} implies that either $R^0=L$, or $R^0 \cong L_1$ is simply connected and diagonally embedded in $L$. If $R^0=L$ then $H^0=M^0$ and thus $H=M$ (since $H$ is disconnected), which is false. Therefore, $R^0 \cong L_1$ is diagonally embedded and thus
$$R^0=\{(\tau_1(x), \tau_2(x)) \mid x \in {\rm Sp}_{4}(K) \}$$
where each $\tau_i: {\rm Sp}_{4}(K) \to L_i$ is a bijective morphism.  By Lemma \ref{l:steinberg} we may write 
$\tau_i=t_{x_i}\sigma_{q_i}\gamma^{k_i}$ for some $x_i \in R^0$, $p$-power $q_i$ and  $k_i\in \{0,1\}$, where $\gamma$ is a graph automorphism of $C_2$ if $p=2$, otherwise $\gamma=1$. We may assume that $q_2=1$. Since $N_G(H^0) =H^0C_G(H^0)$, it follows that $V|_{R^0}$ is homogeneous.  

As above, let $\{\eta_1,\eta_2\}$ be a set of fundamental dominant  weights for $R^0$. First assume that $p \neq 2$, so $(k_1,k_2) = (0,0)$. Then $V|_{R^0}$ has composition factors with highest weights $(q_1+1)\eta_1+\eta_2$ and $(q_1+1)\eta_1+q_1\eta_2$,
and the homogeneity of $V|_{R^0}$ implies that $q_1=1$. But Lemma \ref{l:ten} implies that 
$V|_{R^0}$ is non-homogeneous, so we have reached a contradiction.

Now assume $p=2$. Here $V|_{R^0}$ has composition factors with highest weights $\mu_1$ and $\mu_2$ as follows:
$$\begin{array}{ccc} \hline
(k_1,k_2) & \mu_1 & \mu_2 \\ \hline
(0,0) & (q_1+1)\eta_1+\eta_2 & (q_1+1)\eta_1+q_1\eta_2 \\
(1,0) &   \eta_1+2(q_1+1)\eta_2 &  (q_1+1)\eta_1+q_1\eta_2 \\
(0,1) & (q_1+1)\eta_1+\eta_2 &  q_1\eta_1+2(q_1+1)\eta_2 \\
(1,1) & \eta_1+2(q_1+1)\eta_2  &  q_1\eta_1+2(q_1+1)\eta_2\\ \hline
\end{array}$$
Since $V|_{R^0}$ is homogeneous, we deduce that $k_1=k_2$ and $q_1=1$.

If $(k_1,k_2)=(0,0)$ then $V_1|_{R^0}  =   V_{R^0}(\eta_1)\otimes V_{R^0}(\eta_1+\eta_2)$ and thus Lemma \ref{l:ten} contradicts the homogeneity of 
$V|_{R^0}$. Similarly, if $(k_1,k_2) = (1,1)$ then 
\begin{align*}
V_1|_{R^0} = V_{R^0}(2\eta_2)\otimes V_{R^0}(\eta_1+2\eta_2) & \cong   V_{R^0}(\eta_2)^{(2)} \otimes V_{R^0}(\eta_1)\otimes V_{R^0}(\eta_2)^{(2)} \\
 & \cong  (V_{R^0}(\eta_2)\otimes V_{R^0}(\eta_2))^{(2)}\otimes V_{R^0}(\eta_1)
\end{align*}
Now $V_{R^0}(\eta_2)\otimes V_{R^0}(\eta_2)$ has composition factors with highest weights $2\eta_2$ and $2\eta_1$, whence $V_1|_{R^0}$ has composition factors with highest weights 
$\eta_1+4\eta_2$ and $5\eta_1$. This final contradiction completes the analysis of case (xiii) in Table \ref{tab:geom2}.

The case labelled (ix) is similar (and easier). Here $G=B_4$, $V=V_G(\l_4)$ and $M=B_1^2.2$ is a $\C_4(ii)$-subgroup, where $p \neq 2,3$. The connected component $M^0=M_1M_2$ is a central product of two simply connected groups of type $B_1$, and we note that 
$$V|_{M^0}=V_1 \oplus V_2 = (V_{M_1}(\omega_1) \otimes V_{M_2}(3\omega_2)) \oplus (V_{M_1}(3\omega_1) \otimes V_{M_2}(\omega_2))$$ 
(see \cite[Table 6.2]{BGT}), where $\omega_1$ and $\omega_2$ are fundamental dominant weights for $M_1$ and $M_2$, respectively. We leave the details to the reader. 
\end{proof}

\subsection{Proof of Theorem \ref{t:geom1}, Part II}\label{ss:part2}

In order to complete the proof of Theorem \ref{t:geom1}, it remains to deal with the cases labelled (vii), (x), (xi) and (xii) in Table \ref{tab:geom2}. 

\begin{remk}\label{r:irre}
Suppose that $V|_{H}$ is irreducible, where $H<M<G$ and $(G,M,V)$ is one of the cases (x), (xi) or (xii). Here $M$ is the normalizer of a tensor product decomposition $W = W_1 \otimes \cdots \otimes W_t$ of the natural $KG$-module, with $t=3$ or $4$. Therefore, by combining  Lemma \ref{l:tens} with our earlier work in Sections \ref{s:geom1} and \ref{ss:part1}, we deduce that $W|_{H^0}$ is irreducible. Indeed, if $W|_{H^0}$ is reducible then Lemma \ref{l:tens} implies that we may replace $M$ by some other geometric maximal subgroup of $G$ that does not normalize such a decomposition, in which case our earlier work implies that $V|_{H}$ is reducible.
\end{remk}

In order to deal with cases (x) and (xi), we first establish some preliminary reductions.

\begin{lem}\label{l:case9_10_1}
Let $G=C_4$ and let $H<G$ be a closed positive-dimensional subgroup that is contained in a $\C_4(ii)$-subgroup $M=C_1^3.{\rm Sym}_3$ of $G$. Set $V=V_G(\l)$, where $\l = \lambda_2$ and $p \neq 2$, or $\l=\lambda_3$ and $p \neq 3$. If $V|_{H}$ is irreducible, then $H^0$ is a subdirect product of $M^0=C_1^3$.
\end{lem}

\begin{proof}
Here $M^0 = C_1^3=M_1M_2M_3$ is a central product of three simply connected groups of type $C_1$. Let $\o_i$ be the fundamental dominant weight for $M_i$, and note that 
\begin{equation}\label{e:v123}
V|_{M^0} = V_1 \oplus V_2 \oplus V_3,
\end{equation}
where 
\begin{align}\label{e:vmi}
\begin{split}
V_1 & = V_{M_1}(2\o_{1}) \otimes V_{M_2}(2\o_{2})\otimes V_{M_3}(0\o_3)\\
V_2 & = V_{M_1}(2\o_{1}) \otimes V_{M_2}(0\o_{2})\otimes V_{M_3}(2\o_3)\\
V_3 & = V_{M_1}(0\o_{1}) \otimes V_{M_2}(2\o_{2})\otimes V_{M_3}(2\o_3)
\end{split}
\end{align}
if $\l=\l_2$, and
\begin{align}\label{e:vmi2}
\begin{split}
V_1 & = V_{M_1}(\o_{1}) \otimes V_{M_2}(\o_{2})\otimes V_{M_3}(3\o_3)\\
V_2 & = V_{M_1}(\o_{1}) \otimes V_{M_2}(3\o_{2})\otimes V_{M_3}(\o_3)\\
V_3 & = V_{M_1}(3\o_{1}) \otimes V_{M_2}(\o_{2})\otimes V_{M_3}(\o_3)
\end{split}
\end{align}
if $\l=\l_3$ (see \cite[Table 6.2]{BGT}). As noted in Remark \ref{r:irre}, the irreducibility of $V|_{H}$ implies that $W|_{H^0}$ is irreducible, where $W$ denotes the natural $KG$-module. 

As a $KM^0$-module, $W$ lifts to a representation $\rho: L \to {\rm GL}(W)$, where $L=L_1L_2L_3$ is the direct product of three simply connected groups of type $C_1$. Then $M^0 = L/Y$, where $Y = \ker(\rho)$, and there exists a subgroup $R$ of $L$ such that $H^0=R/Y$. 

We need to show that $R^0$ is a subdirect product of $L$. Note that the irreducibility of $W|_{H^0}$ implies that $W|_{R^0}$ is also irreducible. Set $J=[R^0,R^0]$ and recall that $H^0$ is reductive (see Lemma \ref{l:hred}), so $R^0$ is reductive and thus 
$$J \in \{C_1^3, C_1^2,C_1,1\}.$$ 

If $J=C_1^3$ then we are done, so assume otherwise. If $J=1$ then $R^0$ is a torus, contradicting the irreducibility of $W|_{R^0}$. Finally, suppose $J=C_1^2$ or $C_1$.   
Since $R^0$ is the product of $J$ and a central torus, the irreducibility of $W|_{R^0}$ implies that $W|_{J}$ is irreducible. This immediately implies that the projection maps $\pi_i:J \to L_i$ are surjective, so $R^0$ is a subdirect product of $L$ as required.
\end{proof}

\begin{lem}\label{l:case9_10_2}
Suppose $H<M<G$ and $(G,M,V)$ is the case labelled (x) in Table \ref{tab:geom2}, where $V=V_G(\l_2)$ and $p \neq 2$. If $V|_{H}$ is irreducible, then $H^0=M^0$.
\end{lem}

\begin{proof}
As in the previous lemma, $G=C_4$ and $M^0 = C_1^3=M_1M_2M_3$ is a central product of  simply connected groups of type $C_1$. Define $L=L_1L_2L_3$, $Y$ and $R$ as above, so $Y \leqs Z(L)$, $M^0=L/Y$ and $H^0 = R/Y$. Note that $V|_{M^0} = V_1 \oplus V_2 \oplus V_3$, where the $V_i$ are given in \eqref{e:vmi}. Also recall that the irreducibility of $V|_{H}$ implies that $W|_{H^0}$ is also irreducible (see Remark \ref{r:irre}).

By Lemma \ref{l:case9_10_1}, $R$ is a subdirect product of  $L_1L_2L_3$, so Proposition \ref{p:useful} implies that $R^0$ is isomorphic to a commuting product of simple groups of type $C_1$. If $R^0$ is of type $C_1^3$ then $H^0=M^0$ and we are done, so let us assume that $R^0$ is of type $C_1$ or $C_1^2$.

Suppose $R^0$ is of type $C_1$. Let $\eta_1$ be the fundamental dominant weight for $R^0$.  By Proposition \ref{p:useful}, $R^0 \cong L_1$ is simply connected and diagonally embedded in $L$, so we may write 
$$R^0=\{(\tau_1(x), \tau_2(x), \tau_3(x)) \mid x \in {\rm Sp}_{2}(K)\}$$
where $\tau_i: {\rm Sp}_{2}(K)\to L_i$ is a bijective morphism. By Lemma \ref{l:steinberg}, 
$\tau_i=t_{x_i}\sigma_{q_i}$ for some $x_i \in R^0$ and $p$-power $q_i$, and we may assume that $q_3=1$ (since $R^0$ is a closed subgroup of $L$). Then $V|_{R^0}$ 
has composition factors with highest weights
$(2q_1+2q_2)\eta_1$, $(2q_1+2)\eta_1$ and $(2q_2+2)\eta_1$.
Since $V|_{R^0}$ is homogeneous (note that $N_G(H^0) = H^0C_G(H^0)$), it follows that 
$q_1=q_2=1$. But now Lemma \ref{l:ten} contradicts the homogeneity of $V|_{R^0}$.

Finally, let us assume that $R^0  = R_1R_2$ is of type $C_1^2$. Let $\{\eta_{1}, \eta_{2}\}$ be  fundamental dominant weights for $R^0$.  Once again, Proposition \ref{p:useful} implies that $R_1$ and $R_2$ are simply connected groups of type $C_1$, and without loss of generality we may assume that 
$$R_1=\{(\tau_{1}(x),1,1) \mid x \in {\rm Sp}_2(K)\},\;\; R_2=\{(1,\tau_{2}(x),\tau_{3}(x)) \mid  x \in {\rm Sp}_2(K)\}$$ 
where $\tau_{i}: {\rm Sp}_2(K)\to L_i$ is a bijective morphism. As before, we may write 
$\tau_{i}=t_{x_{i}}\sigma_{q_{i}}$, so $V|_{R^0}$ has composition factors with highest weights
\begin{equation}\label{e:wts1}
2q_{1}\eta_1+2q_{2}\eta_2, \;\; 2q_{1}\eta_1+2q_{3}\eta_2, \;\; 2(q_{2}+q_{3})\eta_2.
\end{equation}
Since $R^0$ is a closed subgroup of $L$, at least one $q_i$ is equal to $1$. By considering $N_G(H^0)$, it follows that $V|_{R^0}$ is either homogeneous, or the homogeneous components of $V|_{R^0}$ are conjugate under an involutory automorphism of $R^0$ interchanging $R_1$ and $R_2$. However, this observation is incompatible with the weights recorded in \eqref{e:wts1}. This is a contradiction.
\end{proof}

The next lemma gives an analogous reduction for $V=V_G(\l_3)$ in case (x) in Table \ref{tab:geom2}. Note that we include the additional case $p=2$, which will be need in Propositions \ref{p:geom11} and \ref{p:geom6}.

\begin{lem}\label{l:case9_10_3}
Suppose $H<M<G$ with $G=C_4$, $M = C_1^3.{\rm Sym}_3$ and $p \neq 3$. Set $V=V_G(\l_3)$ and assume that $V|_{H}$ is irreducible. Then $H^0=M^0$.  
\end{lem}

\begin{proof}
This is entirely similar to the proof of Lemma \ref{l:case9_10_2}, and we omit the details. In  particular, we note that there are no additional difficulties when $p=2$. 
\end{proof}

We are now in a position to settle cases (x) and (xi).

\begin{prop}\label{p:geom9}
Suppose $H<M<G$ and $(G,M,V)$ is the case labelled (x) in Table \ref{tab:geom2}, so $G=C_4$, $M=C_1^3.{\rm Sym}_3$ and $V=V_G(\l)$, where $\l=\l_2$ or $\l_3$. Then $V|_{H}$ is irreducible if and only if $H=C_1^3.Z_3$. 
\end{prop}

\begin{proof}
If $V|_{H}$ is irreducible, then Lemmas \ref{l:case9_10_2} and \ref{l:case9_10_3} imply that $H^0=M^0$. Therefore, $H$ transitively permutes the $V_i$ in \eqref{e:v123}, so $H/H^0 \leqs {\rm Sym}_3$ is transitive. Since $H<M$, we conclude that $H=C_1^3.Z_3$ is the only possibility.
\end{proof}

\begin{prop}\label{p:geom11}
Suppose $H<M<G$ and $(G,M,V)$ is the case labelled (xi) in Table \ref{tab:geom2}, so $G=D_4$, $M=C_1^3.{\rm Sym}_3$, $p=2$ and $V=V_G(\l)$, where $\l \in \{\l_1+\l_4, \l_3+\l_4\}$.  Then $V|_{H}$ is irreducible if and only if $H=C_1^3.Z_3$. 
\end{prop}

\begin{proof}
Write $M^0=C_1^3=M_1M_2M_3$ and let $\o_i$ be the fundamental dominant weight for $M_i$. Note that 
$$H < M = C_1^3.{\rm Sym}_3 < G=D_4< N = C_4.$$ 
Suppose that $V|_{H}$ is irreducible. 

First consider the case $\l=\l_3+\l_4$. Here $V|_{M^0} = V_1 \oplus V_2 \oplus V_3$ and \eqref{e:vmi2} holds 
(see \cite[Table 6.2]{BGT}). Let  $\{\xi_1,\xi_2,\xi_3,\xi_4\}$ be fundamental dominant weights for $N$. Since $V$ is the restriction of the $KN$-module $V_N(\xi_3)$ to $G$ (see the case labelled ${\rm MR}_{4}$ in \cite[Table 1]{Seitz2}), Lemma \ref{l:case9_10_3} implies that $H^0=M^0$ and thus $H=C_1^3.Z_3$ is the only possibility.  An entirely similar argument applies if $\l=\l_1+\l_4$, and once again we deduce that $H=C_1^3.Z_3$.
\end{proof}

To complete the proof of Theorem \ref{t:geom1} it remains to consider cases (vii) and (xii) in Table \ref{tab:geom2}. First we establish an important reduction for case (vii).

\begin{lem}\label{l:geom6}
Suppose $H<M<G$ and $(G,M,V)$ is the case labelled (vii) in Table \ref{tab:geom2}, so $G=C_n$, $M=D_n.2$, $p=2$ and $\l=\sum_{i=1}^{n-1}a_i\l_i$. If $V|_{H}$ is irreducible, then $n=4$, $\l = \l_3$ and $H < M^0$.
\end{lem}

\begin{proof}
Here $M=D_n.2$ is a $\C_6$-subgroup of $G$, where $G=C_n$, $n \geqs 2$ and $p=2$. Set $V=V_G(\l)$ and let $\{\eta_1, \ldots, \eta_n\}$ be a set of fundamental dominant weights for $M^0=D_n$. We have $\l=\sum_{i=1}^{n-1}a_i\l_i$ and $V|_{M^0}$ is irreducible with highest weight 
$$\mu = \sum_{i=1}^{n-2}a_i\eta_i+a_{n-1}(\eta_{n-1}+\eta_n)$$ 
(see case MR$_4$ in \cite[Table 1]{Seitz2}).

Note that $n \geqs 3$ since we are assuming that $\l$ is nontrivial, $p$-restricted and $V \neq W$. Suppose that $V|_{H}$ is irreducible.

First assume $H \not\leqs M^0$. As in the proof of Proposition \ref{p:geom3}, let $J$ be a maximal subgroup of $M=D_n.2={\rm GO}(W)$ containing $H$. Then $J$ is disconnected, 
$(M,J,V)$ is an irreducible triple, and $V|_{M^0} = V_{M^0}(\mu)$. By applying \cite[Theorem 3]{BGMT}, we deduce that $J$ is a geometric subgroup of $M$, so the possibilities for $J$ are listed in Table \ref{tab:dn2}. By applying Proposition \ref{p:dn2_2}, we conclude that $V|_{J}$ is reducible, which is a contradiction.

Now suppose $H \leqs M^0$. Let $J$ be a maximal subgroup of $M^0$ containing $H$ (note that $H <M^0$ since we are assuming $H$ is disconnected). Then $V|_{J}$ is irreducible and we can consider the possibilities for the irreducible triple $(M^0,J,V|_{M^0})$. By inspecting \cite{BGMT,BGT,Seitz2}, given the highest weight of $V|_{M^0}$, we quickly deduce that $n=4$ is the only possibility (note that if $n=3$ then the highest weight of $V|_{M^0}$ has at least two non-zero coefficients and it is easy to check that there are no compatible examples), $J=C_1^3.{\rm Sym}_3$ is a $\C_4(ii)$-subgroup of $M^0=D_4$ and $V|_{M^0} = V_{M^0}(\eta_3+\eta_4)$, so $\l=\l_3$. 
\end{proof}

\begin{prop}\label{p:geom6}
Suppose $H<M<G$ and $(G,M,V)$ is the case labelled (vii) in Table \ref{tab:geom2}, so $G=C_n$, $M=D_n.2$, $p=2$ and $\l=\sum_{i=1}^{n-1}a_i\l_i$. Then $V|_{H}$ is irreducible if and only if $n=4$, $\l=\l_3$ and $H= C_1^3.Z_3$ or $C_1^3.{\rm Sym}_3$.
\end{prop}

\begin{proof}
Suppose that $V|_{H}$ is irreducible. By the proof of the previous lemma, $n=4$, $\l = \l_3$ and 
$$H \leqs J = C_1^3.{\rm Sym}_3 < M^0 = D_4 < G = C_4.$$
In addition, if $\{\eta_1, \ldots, \eta_4\}$ are fundamental dominant weights for $M^0 = D_4$, then $V|_{M^0} = V_{M^0}(\eta_3+\eta_4)$. Therefore, we have now reduced the problem to the case numbered (xi) in Table \ref{tab:geom2}, which was handled in Proposition \ref{p:geom11}. In particular, we conclude that $H=C_1^3.Z_3$ or $C_1^3.{\rm Sym}_3$, as required.
\end{proof}

\begin{remk}\label{r:geom6}
As noted in Remark \ref{r:case7}, Proposition \ref{p:geom6} implies that $V|_{H}$ is irreducible in case (v) if and only if $G=B_4$, $p=2$, $\l=\l_3$ and $H=B_1^3.Z_3$ or $B_1^3.{\rm Sym}_3$.
\end{remk}

To complete the proof of Theorem \ref{t:geom1}, it remains to deal with the case labelled (xii) in Table \ref{tab:geom2}.

\begin{prop}\label{p:geom12}
Suppose $H<M<G$ and $(G,M,V)$ is the case labelled (xii) in Table \ref{tab:geom2}, so $G=D_8$, $M=C_1^4.{\rm Sym}_4$, $p \neq 3$ and $V=V_G(\l)$, where $\l=\l_7$. Then $V|_{H}$ is irreducible if and only if $H=C_1^4.X$, where $X < {\rm Sym}_4$ is transitive. 
\end{prop}

\begin{proof}
Write $M^0 = C_1^4=M_1M_2M_3M_4$, which is a central product of simply connected groups of type $C_1$, and let $\o_{i}$ be the fundamental dominant weight for $M_i$. Then  \cite[Table 6.2]{BGT} indicates that
$V|_{M^0} = V_1 \oplus V_2 \oplus V_3\oplus V_4$,
where
\begin{align}
\begin{split}
V_1 & = V_{M_1}(\o_{1}) \otimes V_{M_2}(\o_{2})\otimes V_{M_3}(\o_3)\otimes V_{M_4}(3\o_{4})\\
V_2 & = V_{M_1}(\o_{1}) \otimes V_{M_2}(\o_{2})\otimes V_{M_3}(3\o_3)\otimes V_{M_4}(\o_{4})\\
V_3 & = V_{M_1}(\o_{1}) \otimes V_{M_2}(3\o_{2})\otimes V_{M_3}(\o_3)\otimes V_{M_4}(\o_{4})\\
V_4 & = V_{M_1}(3\o_{1}) \otimes V_{M_2}(\o_{2})\otimes V_{M_3}(\o_3)\otimes V_{M_4}(\o_{4})
\end{split}
\end{align}
Assume that $V|_{H}$ is irreducible, so $W|_{H^0}$ is also irreducible (see Remark \ref{r:irre}), where $W$ is the natural $KG$-module. 

Since $W$ is a $KM^0$-module, it lifts to a representation $\rho: L \to {\rm GL}(W)$, where $L=L_1L_2L_3L_4$ is the direct product of simply connected groups of type $C_1$. Then $M^0 = L/Y$, where $Y = \ker(\rho)$, and there exists a subgroup $R$ of $L$ such that $H^0=R/Y$. 

\vs

\noindent \emph{Claim.} $H^0$ is a subdirect product of $M^0$.

\vs

We need to show that $R^0$ is a subdirect product of $L$. To do this, we proceed as in the proof of Lemma \ref{l:case9_10_1}; the argument is very similar (using the irreducibility of $W|_{H^0}$) and we omit the details.

\vs

\noindent \emph{Claim.} $H^0=M^0$

\vs

Since $R^0$ is a subdirect product of $L$, Proposition \ref{p:useful} implies that $R^0$ is isomorphic to a commuting product of simply connected groups of type $C_1$. If $R^0$ is of type $C_1^4$ then $H^0=M^0$, so we may assume that $R^0$ is of type $C_1, C_1^2$ or $C_1^3$.

Suppose $R^0$ is of type $C_1$. Let $\eta$ be the fundamental dominant weight for $R^0$. By Proposition \ref{p:useful}, we may write
$$R^0=\{(\tau_1(x), \tau_2(x), \tau_3(x), \tau_4(x)) \mid x \in {\rm Sp}_{2}(K) \},$$ 
where $\tau_i: {\rm Sp}_{2}(K) \to L_i$ is a bijective morphism. As before, Lemma \ref{l:steinberg} implies that $\tau_i=t_{x_i}\sigma_{q_i}$ for some $x_i \in R^0$ and $p$-power $q_i$, and we may assume that $q_4=1$. Then $V|_{R^0}$ has composition factors with highest weights
$$(q_1+q_2+q_3+3)\eta, \;\; (q_1+q_2+3q_3+1)\eta, \;\; (q_1+3q_2+q_3+1)\eta, \;\; (3q_1+q_2+q_3+1)\eta.$$
Since $N_G(H^0) =  H^0C_G(H^0)$, it follows that $V|_{R^0}$ is homogeneous, so $q_i=1$ for all $i$. But Lemma \ref{l:ten} implies that $V|_{R^0}$ is non-homogeneous, so we have reached a contradiction. Note that if $p=2$ then one of the tensor factors in $V_i|_{R^0}$ is non-restricted, but we can still argue as in the proof of Lemma \ref{l:ten} by considering the tensor product of the restricted factors.

Next suppose $R^0=R_1R_2$, where each $R_i \cong C_1$ is simply connected. Let $\{\eta_{1}, \eta_{2}\}$ be fundamental dominant weights for $R^0$. In view of Proposition \ref{p:useful}, we may assume that either 
\begin{equation}\label{e:r11}
R_1=\{(\tau_{1}(x),1,1,1) \mid x \in {\rm Sp}_2(K)\},\;\; R_2=\{(1,\tau_{2}(x),\tau_{3}(x),\tau_{4}(x)) \mid x \in {\rm Sp}_2(K)\}
\end{equation}
or
\begin{equation}\label{e:r12}
R_1=\{(\tau'_{1}(x),\tau'_{2}(x),1,1) \mid x \in {\rm Sp}_2(K)\},\;\; R_2=\{(1,1,\tau'_{3}(x),\tau'_{4}(x)) \mid x \in {\rm Sp}_2(K)\}
\end{equation}
where $\tau_i, \tau_i'$ are bijective morphisms from ${\rm Sp}_{2}(K)$ to $L_i$. Note that 
$V|_{R^0}$ is either homogeneous, or the homogeneous components of $V|_{R^0}$ are conjugate under an involutory automorphism of $R^0$ that interchanges $R_1$ and $R_2$. 

First assume \eqref{e:r11} holds. As above, we may write $\tau_{i}=t_{x_{i}}\sigma_{q_{i}}$, so $V|_{R^0}$ has composition factors with highest weights
$q_{1}\eta_1+ (q_{2}+q_{3}+3q_{4})\eta_2$, $q_{1}\eta_1+ (q_{2}+3q_{3}+q_{4})\eta_2$,  $q_{1}\eta_1+(3q_{2}+q_{3}+q_{4})\eta_2$
and $3q_{1}\eta_1+ (q_{2}+q_{3}+q_{4})\eta_2$, 
and at least one $q_i$ is equal to $1$. But this contradicts the above homogeneity properties of $V|_{R^0}$. Similarly, suppose \eqref{e:r12} holds and write $\tau_{i}'=t_{x_{i}'}\sigma_{q_{i}'}$, where at least one $q_i'$ is equal to $1$. Then $V|_{R^0}$ has composition factors with highest weights
$(q'_{1}+q'_{2})\eta_1 +(q'_{3}+3q'_4)\eta_2$, $(q'_{1}+q'_{2})\eta_1+ (3q'_{3}+q'_4)\eta_2$,  $(q'_{1}+3q'_{2})\eta_1+ (q'_{3}+q'_4)\eta_2$ and 
$(3q'_{1}+q'_{2})\eta_1+ (q'_{3}+3q'_4)\eta_2$,
and we deduce that $q'_i=1$ for all $i$. Then $V|_{R^0}$ has composition factors with highest weights $2\eta_1+4\eta_2$ and $2\eta_1+ 2\eta_2$, but once again this is incompatible with the observed homogeneity properties of $V|_{R^0}$.  

Finally, let us assume $R^0 = R_1R_2R_3$ is of type $C_1^3$. Let $\{\eta_{1}, \eta_{2},\eta_3\}$ be  fundamental dominant weights for $R^0$. In view of Proposition \ref{p:useful}, we may assume that 
\begin{align*}
R_1 & = \{(\tau_{1}(x),\tau_{2}(x),1,1) \mid x \in {\rm Sp}_2(K)\} \\
R_2 & = \{(1,1,\tau_{3}(x),1) \mid x \in {\rm Sp}_2(K)\} \\
R_3 & = \{(1,1,1,\tau_{4}(x)) \mid x \in {\rm Sp}_2(K)\}
\end{align*}
where $\tau_{i}=t_{x_{i}}\sigma_{q_{i}}$ as above. Note that $V|_{R^0}$ is either homogeneous, or any two homogeneous components of $V|_{R^0}$ are conjugate via an automorphism of $R^0$ of order $2$ or $3$, induced by a suitable permutation of the three factors of $R^0$. However, $V|_{R^0}$ has composition factors with  highest weights
$(q_{1}+q_{2})\eta_1+ q_{3}\eta_2+ 3q_{4}\eta_3$, $(q_{1}+q_{2})\eta_1+ 3q_{3}\eta_2+ q_{4}\eta_3$, 
$(q_{1}+3q_{2})\eta_1+ q_{3}\eta_2+ q_{4}\eta_3$ and $(3q_{1}+q_{2})\eta_1+ q_{3}\eta_2+ q_{4}\eta_3$, so $V|_{R^0}$ does not have the stated homogeneity property. This is a contradiction.

\vs

We have now reduced to the case $H^0=M^0=C_1^4$. From the above description of $V|_{M^0}$ it is clear that $V|_{H}$ is irreducible if and only if $H=C_1^4.X$, where $X < {\rm Sym}_4$ is transitive. The result follows. 
\end{proof}

This completes the proof of Theorem \ref{t:geom1}.

\section{Non-geometric subgroups}\label{s:sfamily}

In order to complete the proof of Theorem \ref{t:main}, it remains to determine the irreducible triples $(G,H,V)$ satisfying Hypothesis \ref{h:our}, where 
$V|_{H^0}$ is reducible and $H$ is \emph{not} contained in a geometric subgroup of $G$. 
Here the latter condition implies that $H$ is one of the \emph{non-geometric} subgroups that arise in part (ii) of Theorem \ref{t:ls} in Section \ref{s:ss}. In particular, $W|_{H^0}$ is irreducible and tensor indecomposable, so we can apply the main theorem of \cite{BGMT}. (Note that if $(G,p)=(C_n,2)$ and $H$ fixes a non-degenerate quadratic form on $W$, then $H$ is contained in a geometric $\C_6$-subgroup $D_n.2<C_n$, which is a situation we dealt with in Proposition \ref{p:geom6}.)

\begin{thm}\label{t:smain}
Let $G, H$ and $V=V_G(\l)$ be given as in Hypothesis \ref{h:our}, and assume that $H$ is not contained in a geometric subgroup of $G$. Then $V|_{H}$ is reducible. 
\end{thm}
 
\begin{proof}
By \cite[Theorem 1]{BGMT}, the only possibility is the case $(G,H,\l)=(C_{10},A_5.2, \l_3)$ with  $p \neq 2,3$. However, we claim that $H$ is a maximal subgroup, so this example does not satisfy the conditions in Hypothesis \ref{h:our}. To see this, let $\{\eta_1, \ldots, \eta_5\}$ be a set of fundamental dominant weights for $H^0$ and note that $W=V_{H^0}(\eta_3)$ (where  $W$ is the natural $KG$-module). Suppose $H$ is non-maximal, say $H<M<G$ with $M$ maximal. Then $M$ is non-geometric, so $M^0$ is a simple group that acts irreducibly on $W$ and thus $(M^0,H^0,W)$ is an irreducible triple. By inspecting \cite[Table 1]{Seitz2}, we see that there are no compatible examples. This is a contradiction, so $H$ is maximal as claimed.
\end{proof}

\vs

In view of Theorems \ref{t:geom0}, \ref{t:geom1} and \ref{t:smain}, the proof of Theorem \ref{t:main} is complete.

\section{Spin modules}\label{s:spin}

In this section, we briefly consider the special case arising in part (b) of Theorem \ref{t:main}, where $G$ is a simply connected group of type $B_n$ or $D_n$ (or type $C_n$ if $p=2$), $V$ is a spin module and $H$ is a \emph{decomposition subgroup} of $G$, as defined in the introduction. Recall that $H$ normalizes an orthogonal decomposition 
$$W = W_1 + \cdots + W_t$$
of the natural $KG$-module $W$, where the $W_i$ are pairwise orthogonal subspaces. 

Our goal here is simply to highlight the difference between this very specific situation and the general case we have studied in Sections \ref{s:geom1}, \ref{s:geom2} and \ref{s:sfamily}. We will do this by establishing a preliminary result (see Proposition \ref{p:spin1}); a detailed analysis of spin modules and decomposition subgroups will be given in a forthcoming paper.

Let $G$ be a simply connected simple algebraic group of type $B_n$ or $D_n$ over an algebraically closed field $K$ of characteristic $p$. For convenience, we will assume that $p \neq 2$. Let $W$ be the natural $KG$-module. As before, fix a set of simple roots $\{\a_1, \ldots,\a_n\}$ and fundamental dominant weights $\{\l_1, \ldots, \l_n\}$ for $G$. We will assume that $n \geqs 3$ if $G=B_n$ and $n \geqs 5$ if $G=D_n$ (note that the spin modules for $D_4$ are excluded in Hypothesis \ref{h:our}; see Remark \ref{r:hyp}). We may write $\overline{G} = G/Z = {\rm SO}(W)$, where $Z \leqs Z(G)$. Similarly, for a subgroup $J$ of $G$ we set $\overline{J} = JZ/Z \leqs \overline{G}$.

Let $V$ be a spin module for $G$. In terms of highest weights, either $V=V_{G}(\l_n)$, or $G=D_n$ and $V=V_G(\l_{n-1})$ (in the latter case, note that $V_G(\l_{n-1})=V_G(\l_{n})^{\tau}$, where $\tau$ is a graph automorphism of $G$). The next result is well known (see \cite[Lemma 2.3.2]{BGT} for a proof).

\begin{lem}\label{l:dmspin}
$\dim V_{B_n}(\l_n)=2^n$ and $\dim V_{D_n}(\l_{n-1}) = \dim V_{D_n}(\l_{n}) = 2^{n-1}$.
\end{lem}

Let $W=W_1 \perp W_2$ be an orthogonal decomposition, where $W_1$ and $W_2$ are non-degenerate subspaces with $\dim W_i \geqs 3$. Let $H$ be the stabilizer in $G$ of this decomposition, so $H = H^0.2$ and $H^0$ is a central product of two simply connected orthogonal groups. More precisely, 
$$\overline{H}  = \overline{G}_{W_1} \cap \overline{G}_{W_2} = ({\rm GO}(W_1) \times {\rm GO}(W_2)) \cap \overline{G} = \overline{H}^0.2$$
and $\overline{H}^0 = {\rm SO}(W_1) \times {\rm SO}(W_2)$ is semisimple. 

\begin{prop}\label{p:spin}
Let $V$ be a spin module for $G$ and let $H$ be the stabilizer in $G$ of an orthogonal decomposition $W = W_1 \perp W_2$ as above. Then $V|_{H}$ is irreducible. Moreover, each $KH^0$-composition factor of $V$ is a tensor product of spin modules for both orthogonal factors of $H^0$. 
\end{prop}

\begin{proof}
If $G=B_n$ then $H/Z(G)$ is a disconnected subgroup in the collection $\C_1$ of geometric maximal subgroups of $G$, and the result follows immediately from \cite[Proposition 3.1.1]{BGT}. Now assume $G=D_n$. If $W_1$ is odd-dimensional, then $H/Z(G)$ is connected and Seitz's main theorem \cite[Theorem 1]{Seitz2} implies that $V|_{H^0}$ is an irreducible tensor product of appropriate spin modules (see the cases labelled ${\rm IV}_{1}$ (with $k=1$) and ${\rm IV}_{2}$ in \cite[Table 1]{Seitz2}). 
Finally, suppose $W_1$ is even-dimensional. If $\dim W_1 \neq n$, then $H/Z(G)$ is a disconnected $\C_1$-subgroup and \cite[Proposition 3.1.1]{BGT} applies. If $\dim W_1 = n$ is even then $H$ is properly contained in a $\C_2$-subgroup of $G$ (namely, the full normalizer in $G$ of the orthogonal decomposition); the proof of \cite[Lemma 3.2.3]{BGT} goes through unchanged, and the result follows.
\end{proof}

Now consider an orthogonal decomposition of the form
\begin{equation}\label{eq:w}
W = W_1 \perp \ldots \perp W_t,
\end{equation}
where $t \geqs 2$ and each $W_i$ is a non-degenerate subspace with $\dim W_i \geqs 3$. Let $H$ be the stabilizer in $G$ of this decomposition, so 
$$\overline{H}  = \bigcap_{i=1}^{t}\overline{G}_{W_i}  = ({\rm GO}(W_1) \times \cdots \times {\rm GO}(W_t)) \cap \overline{G} = \overline{H}^0.2^{t-1}$$
If $t=2$ then $V|_{H}$ is irreducible by Proposition \ref{p:spin}, so let us assume $t \geqs 3$. We claim that $V|_{H}$ is still irreducible. To see this, we first handle the special case where the $W_i$ are equidimensional.

\begin{lem}\label{l:spinc2}
If $\dim W_i = d \geqs 3$ for all $i$, then $V|_{H}$ is irreducible.
\end{lem}

\begin{proof}
If $d=2l+1$ is odd, then $H = 2^{t-1} \times B_l^t$ and $H.{\rm Sym}_t$ is a $\C_2$-subgroup of $G$. Here the proof of \cite[Lemma 4.3.2]{BGT} goes through unchanged (the symmetric group ${\rm Sym}_t$ in the $\C_2$-subgroup plays no role in the argument) and we deduce that $V|_{H}$ is irreducible. 

Now assume $d=2l$ is even, so $H=D_l^t.2^{t-1}$ and we may write $H^0 = X_1 \cdots X_t$, where each $X_i = D_l$ is simply connected. Here the elementary abelian $2$-group $2^{t-1}$ is generated by involutions $z_1, \ldots, z_{t-1}$, where $z_i$ acts as a graph automorphism on $X_i$ and $X_{i+1}$, and centralizes the remaining factors of $H^0$. Now $V_1 \otimes \cdots \otimes V_t$ is a composition factor of $V|_{H^0}$, where each $V_i$ is a spin module for $X_i$. By 
repeatedly applying the $z_i \in H$ to conjugate this composition factor, we deduce that $V|_{H^0}$ has at least $2^{t-1}$ distinct, $H$-conjugate $KH^0$-composition factors. Since  
$$2^{t-1} \cdot \dim (V_1 \otimes \cdots \otimes V_t) = 2^{t-1} \cdot 2^{t(l-1)}  = 2^{tl-1} = \dim V$$
(see Lemma \ref{l:dmspin}) we conclude that $V|_{H}$ is irreducible.
\end{proof}

We can now establish our main result for spin modules and decomposition subgroups.

\begin{prop}\label{p:spin1}
Let $H$ be the stabilizer in $G$ of the decomposition in \eqref{eq:w}, and assume $\dim W_i \geqs 3$ for each $i$. Then $V|_{H}$ is irreducible.
\end{prop}

\begin{proof}
Let $d_1, \ldots, d_s$ be the distinct dimensions of the summands in \eqref{eq:w}, and let $a_i$ be the number of summands of dimension $d_i$. If $s>1$ then we may assume that $d_i < d_{i+1}$ for all $1 \leqs i <s$. We may re-order and re-label the subspaces in \eqref{eq:w} so that
\begin{equation}\label{eq:w2}
W = (W_{1,1} \perp \ldots \perp W_{1,a_1}) \perp \ldots \perp (W_{s,1} \perp \ldots \perp W_{s,a_s}),
\end{equation}
where $\dim W_{i,j} = d_i$ for all $i,j$. Then
$$\overline{H} = ({\rm GO}(W_{1,1}) \times \cdots \times {\rm GO}(W_{1,a_1}) \times \cdots \times {\rm GO}(W_{s,1}) \times \cdots \times {\rm GO}(W_{s,a_s})) \cap \overline{G}.$$
We proceed by induction on $s$.

The base case $s=1$ was handled in Lemma \ref{l:spinc2}, so let us assume $s>1$. Set 
$$U_1 = (W_{1,1} \perp \ldots \perp W_{1,a_1}) \perp \ldots \perp (W_{s-1,1} \perp \ldots \perp W_{s-1,a_{s-1}})$$
and
$$U_2 = W_{s,1} \perp \ldots \perp W_{s,a_s},$$
so $W = U_1 \perp U_2$ and $H \leqs M$, where 
$$\overline{M} = ({\rm GO}(U_1) \times {\rm GO}(U_2)) \cap \overline{G} = \overline{M}^0.2.$$
Write $M^0 = M_1M_2$, where $M_1$ and $M_2$ are simply connected orthogonal groups with natural modules $U_1$ and $U_2$, respectively. 
Let $H_1 = H \cap M^0$ and note that $H=H_1.2 \not\leqs M^0$. 

Here $M^0=M_1M_2$ is a central product. The $KM^0$-module $W$ lifts to a representation $\rho: L \to {\rm GL}(W)$, where $L=L_1 \times L_2$ is the direct product of two simply connected orthogonal groups with $L_i \cong M_i$, so $M^0 = L/Y$ where $Y = \ker(\rho)$. In particular, there exist subgroups $R \leqs R_1 \leqs L$ such that $H_1=R_1/Y$ and $H^0=R/Y$. Note that $H_1/H^0 \cong R_1/R$ and $R^0=R_1^0$ (since $H^0 = H_1^0$). Let 
$\pi_i: R_1 \rightarrow L_i$
be the $i$-th projection map and note that $R_1 = \pi_1(R_1) \times \pi_2(R_1)$. There are several cases to consider.

First assume $U_1$ and $U_2$ are odd-dimensional. Here \cite[Table 1]{Seitz2} indicates that $V|_{M^0} = V_1 \otimes V_2$ is irreducible, where $V_i$ is the spin module for $M_i$. The $KM^0$-module $V$ lifts to a representation $\varphi:L \to {\rm GL}(V)$, so we can consider $V|_{R_1}$. By induction, $V_i|_{\pi_i(R_1)}$ is irreducible for $i=1,2$, so $V|_{R_1}$ is irreducible and thus $V|_{H_1}$ is also irreducible.  The result follows.

Next suppose $\dim U_1$ is even and $\dim U_2$ is odd, so \cite[Proposition 3.1.1]{BGT} implies that $V|_{M^0}$ has exactly two composition factors, namely
$$V|_{M^0} = (V_1 \otimes V_2) \oplus (V_1' \otimes V_2),$$
where $V_1$ and $V_1'$ are the two spin modules for $M_1$, and $V_2$ is the spin module for $M_2$. Here the $KM^0$-modules $V_1 \otimes V_2$ and $V_1' \otimes V_2$ arise from representations $\varphi : L \to {\rm GL}(V_1 \otimes V_2)$ and $\varphi' : L \to {\rm GL}(V_1' \otimes V_2)$, so we can consider $(V_1 \otimes V_2)|_{R_1}$ and $(V_1' \otimes V_2)|_{R_1}$. By induction, $V_2|_{\pi_2(R_1)}$ is irreducible, and $\pi_1(R_1)$ acts irreducibly on $V_1$ and $V_1'$.
Therefore, $V|_{R_1}$ has precisely two composition factors, which are interchanged by an element in $R_1.2$ that acts as a graph automorphism on $L_1$ and centralizes $L_2$. 
Therefore, $V|_{H}$ is irreducible. An entirely similar argument applies if $\dim U_1$ is odd and $\dim U_2$ is even.

Finally, suppose $U_1$ and $U_2$ are both even-dimensional. Here 
$$V|_{M^0} = (V_1 \otimes V_2) \oplus (V_1' \otimes V_2')$$ 
where $V_i$ and $V_i'$ are the two spin modules for $M_i$, and the inductive hypothesis implies that $\pi_i(R_1)$ acts irreducibly on $V_i$ and $V_i'$, for $i=1,2$. 
As before, it follows that $V|_{R_1}$ has exactly two composition factors, which are interchanged by an element in $R_1.2$ that acts simultaneously as a graph automorphism on both $L_1$ and $L_2$. Therefore, the two $KH_1$-composition factors of $V$ are $H$-conjugate, and thus $V|_{H}$ is irreducible.
\end{proof}

In view of the previous proposition, we can easily construct chains of positive-dimensional closed subgroups
$$H_k < H_{k-1} < \cdots < H_1 <G$$
such that $V|_{H_i}$ is irreducible for all $i$. For instance, take any sequence of successive  refinements of a fixed orthogonal decomposition of $W$ such that each refinement is also an orthogonal decomposition that only contains subspaces of dimension at least three. Then the stabilizers in $G$ of these decompositions form a chain of subgroups with the desired irreducibility property. In particular, the length of such a chain can be arbitrarily long. This is in stark contrast to the general situation, where the length of an irreducible chain is at most five (see Theorem \ref{t:chains}, which will be proved in the next section).

\section{Irreducible chains}\label{s:chains}

In this final section we prove Theorem \ref{t:chains}. Recall that if $G$ is a simple algebraic group and $V=V_G(\l)$ is a nontrivial $p$-restricted irreducible $KG$-module, then we write $\ell=\ell(G,V)$ for the length of the longest
chain of closed positive-dimensional subgroups
$$H_{\ell}<H_{\ell-1} < \cdots < H_2 < H_1 = G$$
such that $V|_{H_{\ell}}$ is irreducible. 

As noted in the previous section, if $G$ is an orthogonal group (or a symplectic group with $p=2$) and $V$ is a spin module, then $\ell(G,V)$ can be arbitrarily large (one can simply take an appropriate chain of decomposition subgroups, for example). Similarly, if $V=W$ or $W^*$ (where $W$ is the natural $KG$-module) then $\ell(G,V)$ is unbounded. For instance, if $G={\rm Sp}_{2n}(K)$ then set $H_{i} = {\rm Sp}_{2}(K) \wr X_i$, where $X_i \leqs {\rm Sym}_n$ is transitive. The transitivity of $X_i$ implies that $W|_{H_i}$ is irreducible, and it is easy to see that if $n$ is sufficiently large then we can find arbitrarily long chains of transitive subgroups 
$$X_i < X_{i-1} < \cdots < X_1 = {\rm Sym}_n.$$ 
In fact, if we choose $n$ appropriately, then we may assume that each $X_i$ is $3$-transitive (see Remark \ref{r:3trans} below). 

Now let us assume that $V \neq W,W^*$, and also assume that $V$ is not a spin module. In this situation, it is natural to ask whether or not $\ell(G,V)$ is bounded above by an absolute constant. Our main theorem is the following, which immediately yields Theorem \ref{t:chains}. (Note that in Table \ref{tab:chains}, $T$ is a maximal torus of $G$ and ${\rm M}_{n+1}$ is the Mathieu group of degree $n+1$.)

\begin{thm}\label{t:chai}
Let $G$ be a simply connected cover of a simple classical algebraic group with natural module $W$. Let $V=V_G(\l)$ be a $p$-restricted irreducible $KG$-module, where $V \neq W,W^*$ and $V$ is not a spin module. Then one of the following holds:
\begin{itemize}\addtolength{\itemsep}{0.2\baselineskip}
\item[{\rm (i)}] $\ell(G,V) \leqs 5$; or 
\item[{\rm (ii)}] $G=A_n$ and $\l \in \{\l_2, \l_3, \l_{n-2}, \l_{n-1}\}$. 
\end{itemize}
More precisely, excluding the cases in (ii), we have $\ell(G,V) \leqs 3$, unless $(G,V)$ is one of the cases listed in Table \ref{tab:chains}.
\end{thm}

\renewcommand{\arraystretch}{1.1}
\begin{table}
\begin{center}
$$\begin{array}{llcll} \hline
G & V & \ell(G,V) & \mbox{Conditions} & \mbox{Chain} \\ \hline
B_4 & \l_3 & 5 & p = 2 & B_1^3.Z_3 < B_1^3.{\rm Sym}_3 < D_4 < D_4.2 < B_4 \\ 
C_4 & \l_3 & 5 & p = 2 & C_1^3.Z_3 < C_1^3.{\rm Sym}_3 < D_4 < D_4.2 < C_4 \\ 
A_{n} & \l_4, \l_{n-3} & 4 & n \in \{10,11,22,23\} & T.{\rm M}_{n+1} < T.{\rm Alt}_{n+1}< T.{\rm Sym}_{n+1} < A_n \\
A_{23} & \l_5, \l_{19} & 4 & & T.{\rm M}_{24} < T.{\rm Alt}_{24}< T.{\rm Sym}_{24} < A_{23} \\ 
B_3 & 2\l_1 & 4 & p = 3 & A_2 < A_2.2 < G_2 < B_3 \\ \hline
\end{array}$$
\caption{The irreducible chains in Theorem \ref{t:chai}}
\label{tab:chains}
\end{center}
\end{table}
\renewcommand{\arraystretch}{1}

\begin{remk}\label{r:3trans}
The cases in part (ii) of Theorem \ref{t:chai} are genuine exceptions; for suitable values of $n$, $\ell(G,V)$ can be arbitrarily large. By duality, we only need to consider the cases $\l=\l_2$ and $\l_3$. Recall that if $H=T_n.X <G$, then $V_G(\l_3)|_{H}$ is irreducible if and only if $X \leqs {\rm Sym}_{n+1}$ is $3$-transitive (and similarly, $X$ has to be $2$-transitive if $\l=\l_2$); see Proposition \ref{p:geom1}. Suppose $n=q=2^e$ for some positive integer $e \geqs 2$. The finite simple group ${\rm PSL}_{2}(q)$ has a faithful $3$-transitive action on the projective line $\mathbb{F}_{q} \cup \{\infty\}$, which extends to a faithful action of its automorphism group ${\rm P\Gamma L}_{2}(q) = {\rm PSL}_{2}(q).Z_e$. Therefore, ${\rm PSL}_{2}(q).d \leqs {\rm P\Gamma L}_{2}(q)$ is a $3$-transitive subgroup of ${\rm Sym}_{n+1}$ for every divisor $d$ of $e$. In particular, by choosing $e$ appropriately we can construct arbitrarily long chains of $3$-transitive subgroups of ${\rm Sym}_{n+1}$, and each of the corresponding subgroups $T_n.X < G$ acts irreducibly on $V_G(\l_2)$ and $V_G(\l_3)$. 
\end{remk}

\begin{proof}[Proof of Theorem \ref{t:chai}]
The proof is a combination of the main theorems in \cite{Seitz2}, \cite{BGT} and \cite{BGMT}, together with Theorem \ref{t:main}. To illustrate the general approach, we will consider the case  $G=A_n$. Set $V = V_G(\l)$. In order to prove the theorem, we may assume that  
\begin{equation}\label{e:lambda}
\l \not\in \{\l_1,\l_2,\l_3,\l_{n-2},\l_{n-1},\l_n\}
\end{equation}
and $\ell(G,V) \geqs 4$, so there is an irreducible chain $H_4 < H_3 < H_2 < H_1 = G$. 

Consider the irreducible triple $(G,H_4,V)$. First assume that $H_4$ is disconnected and $V|_{H_4^0}$ is reducible. Then 
$(G,H_4,V)$ must be one of the irreducible triples arising in Theorem \ref{t:main}, so $H_4 = T.X_4$ and $\l = \l_k$, where $X_4<{\rm Sym}_{n+1}$ is $s$-transitive and $s = \max\{k, n+1-k\} \geqs 4$. Moreover, $H_3 = T.X_3$ and $H_2 = T.X_2$, where $X_3<X_2 \leqs {\rm Sym}_{n+1}$ are also $s$-transitive groups of degree $n+1$. 

Using the classification of finite simple groups, it can be shown that ${\rm Sym}_{n+1}$ and the alternating group ${\rm Alt}_{n+1}$ (for $n \geqs 5$) are the only $4$-transitive groups of degree $n+1$, unless $n \in \{10,11,22,23\}$ in which case the Mathieu group ${\rm M}_{n+1}$ is also $4$-transitive (see \cite[Theorem 4.11]{Cam}, for example). Similarly, if $t \geqs 5$ then ${\rm Sym}_{n+1}$ and ${\rm Alt}_{n+1}$ (for $n \geqs t+1$) are the only $t$-transitive groups of degree $n+1$, with the single exception of ${\rm M}_{24}$ when $t=5$ and $n = 23$. Since $s \geqs 4$, it follows that 
either $s =4$ and $n \in \{10,11,22,23\}$, or $s=5$ and $n=23$. In each case, $X_4 = {\rm M}_{n+1}$, $X_3 = {\rm Alt}_{n+1}$, $X_2 = {\rm Sym}_{n+1}$ and no proper positive-dimensional subgroup of $H_4$ acts irreducibly on $V$. These special cases are recorded in Table \ref{tab:chains}.

To complete the proof of Theorem \ref{t:chai} for $G=A_n$, we may assume that $(G,H_4,V)$ is an irreducible triple with $H_4$ connected.  The possibilities are recorded in \cite[Table 1]{Seitz2}, and in view of \eqref{e:lambda} we see that the relevant cases therein are labelled 
\begin{equation}\label{e:irrcases}
\mbox{I$_1$, I$'_{1}$, I$_2$, I$_3$, I$_4$, I$_5$, I$_{12}$.}
\end{equation}

Consider the irreducible triple $(G,H_3,V)$. If $H_3$ is connected, then $(G,H_3,V)$ also corresponds to one of the cases in \eqref{e:irrcases}, but it is routine to check that this collection of cases does not contain a pair of triples $(G,H_4,V)$ and $(G,H_3,V)$ with $H_4<H_3$. Finally, suppose that $H_3$ is disconnected. The connectivity of $H_4$ implies that $H_4 \leqs H_3^0$, so $V|_{H_3^0}$ is irreducible and we have an irreducible chain
$$H_4 \leqs H_3^0 < H_3 < H_2 < H_1 = G.$$
By the argument above, we have $H_4 = H_3^0$ and thus $H_4 < H_3 \leqs N_G(H_4)$. It is now easy to see that $(G,H_4,V)$ must correspond to the case I$_4$ or I$_5$ in \cite[Table 1]{Seitz2}, so $n$ is odd, $H_4=D_{(n+1)/2}$ and $H_3 = D_{(n+1)/2}.2$. But $D_{(n+1)/2}.2<A_n$ is a maximal subgroup, so we have reached a contradiction.

This completes the proof of Theorem \ref{t:chai} for $G=A_n$. The other cases are similar, and we leave it to the reader to check the details.
\end{proof}  

This completes the proof of Theorem \ref{t:chains}.

\end{document}